\documentclass[11pt,leqno,oneside]{amsart}
\usepackage{amsmath,amssymb,,comment,graphicx,psfrag}
\usepackage{bm,mathrsfs}
\usepackage{mathtools}
\usepackage{color}
\usepackage{enumerate}
\usepackage{pdfsync}
\usepackage{hhline}
\usepackage[margin=1in]{geometry}
\usepackage{multicol}
\usepackage{multirow}
\usepackage{booktabs}
\usepackage[colorlinks,linktocpage,linkcolor=blue]{hyperref}
\usepackage{placeins} 
\usepackage{subfigure}
\usepackage{float}
\allowdisplaybreaks
\usepackage{algorithmic}
\def\baselinestretch{1.0}

\newcommand{\al}{\alpha}
\newcommand{\fy}{\varphi}

\def\Dal{{\partial_t^\al}}
\def\bDal{{\bar\partial_\tau^\al}}

\usepackage{algorithm}
\theoremstyle{plain}
\newtheorem{theorem}{Theorem}[section]

\newtheorem{remark}{Remark}[section]

\newtheorem{lemma}{Lemma}[section]

\newtheorem{corollary}{Corollary}[section]
\numberwithin{equation}{section}

\usepackage{titlesec}
\titleformat{\section}{\vskip10pt\large\bfseries}{\thesection.}{0.5em}{\centering\vspace{5pt}}
\titleformat{\subsection}{\vskip10pt\normalsize\bfseries}{\thesubsection.}{0.5em}{}

\def\al{\alpha}

\def\Dal{{\partial^\alpha_t}}

\def\dH#1{\dot H^{#1}(\Omega)}

\def\tu{\widetilde{u}}
\def\II{(\Omega)}

\newcommand{\ds}{\mathrm ds}

\newcommand{\dy}{\mathrm dy}
\newcommand{\dz}{\mathrm dz}

\def\Dal{{\partial_t^\al}}

\def\Om{\Omega}
\def\II{(\Om)}

\def\dH#1{\dot H^{#1}(\Omega)}

\def\bDal{\bar\partial_\tau^\alpha}

\def \contour{{\Gamma_{\theta,\sigma}}}
\def\L2Om{{L^2(\Omega)}}

\def\tu{{u_\gamma}}
\def\tud{{u_\gamma^\delta}}

  \def\tuh{u_{\gamma,h}}

\begin{document}

\title[Backward problems of semilinear subdiffusion]{Numerical Reconstruction and Analysis of Backward Semilinear Subdiffusion Problems}


\author[Xu Wu]{$\,\,$Xu Wu$\,$}
\address{Department of Applied Mathematics,
The Hong Kong Polytechnic University, Kowloon, Hong Kong}
\email {polyuxu.wu@polyu.edu.hk}

\author[Jiang Yang]{$\,\,$Jiang Yang$\,$}
\address{Department of Mathematics, Southern University of Science and Technology, Shenzhen,China }
\email {yangj7@sustech.edu.cn}

\author[Zhi Zhou]{$\,\,$Zhi Zhou$\,$}
\address{Department of Applied Mathematics,
The Hong Kong Polytechnic University, Kowloon, Hong Kong}
\email {zhizhou@polyu.edu.hk}

\keywords{semilinear subdiffusion, backward problem, stability, numerical discretization, error estimate, iterative algorithm}

\begin{abstract}
This paper aims to develop and analyze a numerical scheme for solving the backward problem of semilinear subdiffusion equations. We establish the existence, uniqueness, and conditional stability of the solution to the inverse problem by applying the smoothing and asymptotic properties of solution operators and constructing a fixed-point iteration. This derived conditional stability further inspires a numerical reconstruction scheme. To address the mildly ill-posed nature of the problem, we employ the quasi-boundary value method for regularization. A fully discrete scheme is proposed,   utilizing the finite element method for spatial discretization and convolution quadrature for temporal discretization. A thorough error analysis of the resulting discrete system is provided for both smooth and nonsmooth data. This analysis relies on the smoothing properties of discrete solution operators, some nonstandard error estimates optimal with respect to data regularity in the direct problem, and the arguments used in stability analysis. The derived \textsl{a priori}  error estimate offers guidance for selecting the regularization parameter and discretization parameters based on the noise level. Moreover, we propose an easy-to-implement iterative algorithm for solving the fully discrete scheme and prove its linear convergence. Numerical examples are provided to illustrate the theoretical estimates and demonstrate the necessity of the assumption required in the analysis.
\end{abstract}

\maketitle

\renewcommand{\baselinestretch}{0.95}
\setlength\abovedisplayskip{4.0pt}
\setlength\belowdisplayskip{4.0pt}

\section{Introduction}
Let $\Omega \subset \mathbb{R}^d$ with $d \geq 1$ be
a bounded convex polygonal domain. 
We consider the following initial boundary value problem of the  semilinear time-fractional diffusion 
\begin{align}\label{eqn:fde}
\left\{
\begin{aligned}
\partial_t^\alpha u - \Delta u &= f(u) &&\text{in}\,\,\Omega \times (0,T],\\
u &= 0 &&\text{on}\,\,\partial\Omega \times (0,T],\\
u(0) &= u_0 &&\text{in}\,\,\Omega,
\end{aligned}
\right.
\end{align}
where $f(u)$ and $u(0) = u_0$ represent the nonlinear source term and initial value, respectively. The fractional order $\alpha \in (0,1)$ is fixed, and the notation $\partial^{\alpha}_t u$ denotes the Djrbashian--Caputo fractional derivative of order $\alpha$  with respect to time, as defined in  \cite[Definition 2.3]{Jin:book2021}
\begin{equation}
    \partial^{\alpha}_t u(t) = \frac{1}{\Gamma (1 - \alpha)} \int_0^t (t-s)^{-\alpha} u'(s){\rm d}s,
\end{equation}
where $\Gamma(z) = \int_0^{\infty} s^{z-1}e^{-s}\ds$ for
$\Re (z)>0$ denotes Euler's Gamma function.

The model \eqref{eqn:fde} is frequently employed to describe the subdiffusive process that occurs in complex systems where the path of a particle or an ensemble of particles is hindered by obstacles or constraints, leading to a slower-than-normal spread over time. Unlike normal diffusion, where the mean squared displacement (MSD) of a particle grows linearly with time, subdiffusion is characterized by the MSD growing less rapidly, typically following a power-law relation with an exponent less than one. This phenomenon is observed in various fields such as physics, biology, and geology, and it is particularly relevant in the study of transport through cellular membranes, movement in disordered media, and the spread of pollutants in the environment. See thorough reviews \cite{MetzlerKlafter:2000,MetzlerJeon:2014} for the applications and monographs \cite{Du:book,Jin:book2021} for more details about the modeling.

The direct problem associated with the semilinear subdiffusion model \eqref{eqn:fde} has been extensively studied from both theoretical and numerical perspectives. The well-posedness and pointwise-in-time regularity for this model were established in \cite{JinLiZhou:SINUM2018} under the assumption that $u_0 \in H^2\II\cap H_0^1\II$. This proof utilized fractional maximal $L^p$ regularity, and the authors also proposed a fully discrete scheme with error estimates optimal with respect to the data regularity. Subsequent analysis, extended to nonsmooth initial data $u_0 \in \dot H^s\II$ with $s\in [0,2)$, was conducted  in \cite{Karaa:2019}. For smooth initial condition $u_0 \in W^{2,\infty}\II$, high-order time stepping schemes using convolution quadrature generated by backward differentiation formulas were constructed and analyzed in \cite{WangZhou:2020}. In cases of nonsmooth initial data $u_0\in L^\infty\II$, high-order schemes utilizing exponential convolution quadrature and exponential spectral methods were developed in \cite{LiMa:2022} and \cite{LiLinMaRao:2023}, respectively.
A typical example of a semilinear subdiffusion model \eqref{eqn:fde} includes nonlocal-in-time phase-field models,
 which has recently seen significant advancements in mathematical and numerical analysis.
For further reading, see \cite{FritzKhristenkoWohlmuth:2023, DuYangZhou:2020, LiSalgado:2023, Karaa:2021, QuanTangWangYang:2023, TangYuZhou:2019} for a selection of relevant references.
 Additionally, \cite{Banjai:2022, Kopteva:2023} provided insights into posterior error estimation, \cite{Lopez:2008, Fischer:2019} discussed convolution quadrature-based fast algorithms, and \cite{BonitoLeiPasciak:2017, Melenk:2023} explored sinc quadrature-based methods.
We also recommend a recent monograph on the numerical analysis of time-fractional evolution models \cite{JinZhou:2023book}, as well as a monograph discussing various applications of convolution quadrature for evolutionary PDEs \cite{Banjai:CQ}.
 
In the past decade, inverse problems related to subdiffusion models have also been extensively studied, primarily from a theoretical perspective.  We direct readers to the comprehensive review articles \cite{JinRundell:2015,LiYamamoto:2019,LiLiuYamamoto:2019,LiuLiYamamoto:2019}, as well as the references therein for further details. 
In this paper, we focus on the backward problem associated with the subdiffusion model \eqref{eqn:fde}, aiming to reconstruct the initial data
$u_0(x)$ for $x \in \Omega$ from the terminal observation:
\begin{equation}\label{eqn:obs-data}
u(x,T) = g(x), \quad \text{for all}\quad x \in \Omega.
\end{equation}
In practice, observational data often contains noise. In this work, we consider the empirical observational data $g_\delta$ satisfying
\begin{equation}\label{eqn:noisy data}
\|g_\delta-g\|_{L^2{(\Omega)}} = \delta,
\end{equation}
where $\delta$ denotes the noise level. Our objectives are to discuss the solvability of the backward problem, develop a numerical scheme to solve it, and provide an error estimate for the numerical reconstruction of the initial data. This derived error estimate will serve as a guideline for selecting appropriate discretization parameters, namely the spatial mesh size and temporal step size, as well as the regularization parameter in our numerical scheme.

The backward subdiffusion problem has attracted considerable attention in recent literature, primarily focusing on linear variants. The pioneer work \cite{sakamoto2011initial} provided results on uniqueness and some useful stability estimates for linear models. Notably, unlike its integer-order parabolic counterpart ($\alpha=1$), which is severely ill-posed, the backward subdiffusion problem is only mildly ill-posed, as highlighted in  \cite[Theorem 2.1]{sakamoto2011initial}. This work  subsequently inspired numerous studies on
the development and analysis of regularization methods for solving the backward subdiffusion problem  \cite{liu2010backward, wang2013total, wei2014modified, yang2013solving, wei2019variational}. 
 Interestingly, the fractional backward problem could also serve as a regularization method for backward parabolic problems, a strategy explored in \cite{KaltenbacherRundell:2019}.
 Despite the extensive theoretical work, research on numerical discretization and error analysis remains limited. Zhang et al.~\cite{zhang2020numerical} investigated a fully discrete scheme for solving the backward problem and extended their analysis to include time-dependent coefficients using a perturbation argument in \cite{zhang2023stability}. However, the methods predominantly depend on the asymptotic behaviors of Mittag--Leffler functions and the smoothing properties of linear solution operators, which do not readily extend to nonlinear models. This presents a major challenge for theoretical analysis and also complicates the development and rigorous examination of numerical approximations. In \cite{tuan2020existence}, the authors presented a compelling discussion on the existence and regularity of the solution to the inverse problem in a Bochner space $L^p(0,T;H^q\II)$ employing a fixed-point argument. However, the result cannot be extended to the determination of the initial value $u(0)$.  A similar argument for the backward problem for the fractional diffusion-wave model with $\alpha\in(1,2)$  can be found in \cite{BaoCaraballoTuanZhou:2021}.
 A related model incorporating the Riemann–Liouville fractional derivative was discussed in \cite{tuan2020final}, where the authors devised regularized problems using the truncated expansion method and the quasi-boundary value method for numerical approximation.
 Nevertheless, the argument, that highly relies on the explicit form of eigenvalues and eigenfunctions, is restricted to the case that the domain $\Omega$ is rectangular, and cannot be generalized to arbitrary domains. In conclusion, the theoretical framework for determining the initial data $u_0$ in the semilinear model \eqref{eqn:fde} from the terminal observation \eqref{eqn:obs-data} is not yet adequately developed. 
Moreover, we currently lack an effective numerical algorithm with appropriate discretization that can recover the initial data and yield provable error estimates. This gap highlights the need for further research into both the theoretical study and numerical analysis for this inverse problem,  thereby motivating the current work.

The first contribution of this paper is to establish the existence, uniqueness, and stability estimates of the backward semilinear subdiffusion problem. 
The proof combines several nonstandard \textsl{a priori} estimates of the direct problem, the smoothing properties of solution operators, and a constructive fixed point iteration.
 The argument in the stability estimate
lays a key role in the analysis of the regularization scheme proposed in Section \ref{sec:3} and the completely discrete approximation in Section~\ref{sec:4}. 

The next contribution of this paper is to develop a fully discrete scheme with thorough error analysis.
To numerically recover the initial data, we discretize the proposed regularization scheme 
using piecewise linear finite element method (FEM) in space with spatial mesh size $h$, 
and backward Euler convolution quadrature scheme (CQ-BE) in time with temporal step size $\tau$.
The numerical discretization introduces additional discretization errors. 
We establish \textsl{a priori} error bounds for the numerical reconstruction of the initial data.
Specifically, let $U_{h,\gamma}^{0,\delta}$ be the numerical reconstruction of initial data derived by the fully discrete scheme \eqref{eqn:back-fully}, where the positive constant $\gamma$ denotes the regularization parameter. For an arbitrarily and fixed $\mu\in(0,1]$, under some mild conditions on terminal time $T$,
we show that (Theorem~\ref{thm:fully-err})
\begin{equation*}
\|U_{h,\gamma}^{0,\delta}-u_0\|_{\dot{H}^{-\mu}(\Omega)}\le c\left(\gamma^\frac q2+ \gamma^{-1} \delta +
\gamma^{-1}h^2|\log h|+\tau {|\log \tau|^2} \left(\gamma^{-1}h^2|\log h|+h^{\min\{-\mu+q,0\}} \right)  \right),
\end{equation*}
provided that $\| u_0 \|_{\dH {-\mu+q}}\le c$ with some $q\in (0,2]$. Then with the choice
$\gamma \sim \delta^{\frac{2}{q+2}}$, $h^2|\log h|\sim \delta$ and $\tau |\log \tau|^2h^{\min\{-\mu+q,0\}} \sim \delta^{\frac{q}{q+2}}$,
we obtain the optimal approximation error of order 
$O(\delta^{\frac{q}{q+2}})$.
Moreover, for $u_0\in {\dH {-\mu}}$, there holds
\begin{equation*}
\| U_{h,\gamma}^{0,\delta} -  u_0\|_{\dH {-\mu}}\rightarrow 0\quad \text{as}~~\delta,\gamma,h\rightarrow0^+,
~~\frac\delta\gamma\rightarrow0^+,~~\frac{\tau{|\log \tau|^2}}{h^{\mu}}\rightarrow0^+ ~~\text{and}~~ \frac{h^2|\log h|}{\gamma} \rightarrow0^+.
\end{equation*}
To prove the error bound, we first establish new error estimates for the direct problem that is optimal with respect to the regularity of the problem data, as detailed in Lemma~\ref{lem:error-linear} through Lemma~\ref{lem:err-reg:H2}. We then apply the smoothing properties of discrete solution operators, combined with the methodology outlined in the stability analysis (i.e., Theorem \ref{thm:stability}), to derive the desired results. These error estimates are crucial for guiding the selection of discretization parameters $h$ and $\tau$, as well as the regularization parameter $\gamma$, according to the \textsl{a priori} known noise level $\delta$. It is important to note that our theory imposes a restriction on the terminal time $T$, which cannot be arbitrarily large, even though the solution to the direct problem exists for any $T>0$ provided the global Lipschitz condition on the function $f$ is satisfied.
The necessity of this restriction is supported by numerical experiments.
This presents a significant difference from its linear counterpart \cite{sakamoto2011initial, zhang2020numerical} where the reconstruction is always feasible for any $T>0$.

Moreover, we propose an iterative algorithm based on Theorem \ref{thm:stability}, 
as outlined in Algorithm~\ref{alg}. In each iteration,  
 a linear backward problem needs to be solved, which could be efficiently addressed using
conjugated gradient method \cite{zhang2020numerical,zhang2023stability}. 
The contraction property established in Theorem \ref{lemma:stability_fully} guarantees the convergence of the iteration. Numerical results are presented to illustrate our theoretical findings and demonstrate the effectiveness of the proposed algorithm.

The rest of the paper is organized as follows.  In Section \ref{sec:2}, we present preliminary results on solution regularity and the smoothing properties of solution operators. Additionally, we establish the existence, uniqueness, and stability of the inverse problem. Section \ref{sec:3} is dedicated to discussing the regularization approach using the quasi-boundary value method. In Section \ref{sec:4}, we introduce and analyze semi-discrete and fully discrete schemes for solving the backward problem. Finally, in Section \ref{sec:5}, we provide numerical examples to illustrate the theoretical estimates and demonstrate the necessity of the assumption required in the analysis. 
Concluding remarks are given in Section \ref{sec:conclusion}. In the appendices, we show several technical error estimates for the direct problems.
The notation $c$ denotes a generic constant that may change at each occurrence, but it is always independent of the noise level $\delta$ and the discretization parameters $h$ and $\tau$, and the regularization parameter $\gamma$.

\section{Well-posedness of the backward semilinear subdiffusion problem}\label{sec:2}
In this section, we will present  some preliminary results about the semilinear subdiffusion problem~\eqref{eqn:fde},
  including solution representation, and solution regularity. 
 Subsequently, we will establish the well-posedness of the backward problem for the semilinear subdiffusion equation~\eqref{eqn:fde}, specifically addressing the existence and uniqueness of the reconstructing initial data from terminal observation.

\subsection{Preliminaries}

Let $A=-\Delta$ with homogeneous Dirichlet boundary condition.  
$\{(\lambda_j, \fy_j)\}_{j=1}^\infty$ denote the eigenpairs of $A$, where $\{\fy_j\}_{j=1}^\infty$ forms an orthonormal basis in $L^2(\Omega)$. 
Throughout, we denote by $\dH q$ the Hilbert space induced by the norm
$\|v\|_{\dH q}^2:=\|A^\frac{q}{2}v\|_{L^2\II}^2=\sum_{j=1}^{\infty}\lambda_j^q ( v,\fy_j )^2, \  q\ge -1.$
It is easy to see that $\|v\|_{\dH 0}=\|v\|_{L^2\II} $ is the norm in $L^2(\Omega)$,  
$\|v\|_{\dH 1}= \|\nabla v\|_{L^2\II}$ is a norm in $H_0^1(\Om)$, and $\|v\|_{\dH 2}=\|A v\|_{L^2\II}$ is a norm in $H^2(\Om)\cap H^1_0(\Om)$. In general, the space $\dH q$ is the interpolation space $(L^2\II, H^2\II\cap H_0^1\II)_{\frac{q}2}$ for $q\in (0,2)$. Besides, for the negative norm, it is easy to see that $\| \cdot \|_{\dH{-q}}$ is a norm of the dual space of $\dH q$, for $q\in[0,1]$.

Throughout this paper, we assume that the function $f$ satisfies the following global Lipschitz continuity condition:
\begin{equation}\label{eqn:lipconstant}
|f(u) - f(v)| \le L |u - v| \quad \text{for all} \quad u, v \in \mathbb{R},
\end{equation}
where $L > 0$ is the Lipschitz constant.

The argument in this paper can be easily extended to the case where $f$ is locally Lipschitz continuous and the solution to \eqref{eqn:fde} is uniformly bounded. A notable example is the time-fractional Allen–Cahn equation, which satisfies the maximum bound principle; See e.g., \cite{DuYangZhou:2020,TangYuZhou:2019,LiSalgado:2023,FritzKhristenkoWohlmuth:2023}.


For simplicity, we further assume that
\begin{equation}\label{eqn:f0}
f(0) = 0.
\end{equation}
However, our discussion can be readily extended to the case where $f(0) \neq 0$.

By mean of Laplace Transform,  the solution of the semilinear problem \eqref{eqn:fde} can be represented by \cite[equation 3.12]{JinLiZhou:SINUM2018}
 \begin{equation}\label{eqn:sol-rep}
\begin{aligned}
   u(t)=F(t)u_0+\int_0^t E(t-s)f(u(s))\ \ds
   =:  S(t) u_0.
\end{aligned}
\end{equation}
Here,
$ F(t)$ and $E(t)$ denotes linear solution operators defined by
\begin{equation}\label{eqn:FE-MLcop}
\begin{aligned}
F(t) = \frac{1}{2\pi i }\int_\contour e^{zt} z^{\alpha-1}(z^\alpha+A)^{-1}dz\quad\text{and}\quad
E(t) = \frac{1}{2\pi i}\int_\contour  e^{zt} (z^\alpha+A)^{-1} dz,
 \end{aligned}
\end{equation}
respectively.
Here $\contour$ denotes the integral contour in the complex plane $\mathbb{C}$, defined by
\begin{equation*}
\contour =\{z\in \mathbb{C}: |z| = \delta , |\arg z|\le \theta\} \cup \{ z\in \mathbb{C}: z =\rho e^{\pm i\theta},\rho\ge \sigma\}
\end{equation*}
with $\sigma\ge 0$ and $\frac{\pi}{2}<\theta< \frac{\pi}{\alpha}$,
oriented counterclockwise. 
In addition, we employ 
$ S(t)$ to denote the nonlinear solution operator. Then we  can rewrite \eqref{eqn:sol-rep} as 
\begin{equation}\label{eq:sol-rep-ope}
     u(t)=S(t)u_0=F(t)u_0+\int_0^t E(t-s)f(S(s)u_0)\ \ds.
\end{equation}

The following lemma provides smoothing properties and asymptotic behavior of solution operators $F(t)$ and $E(t)$ defined in \eqref{eqn:FE-MLcop}. 
The proof of (i) was provided in \cite[Theorems 6.4 and 3.2]{Jin:book2021}, while (ii) was established by Sakamoto and Yamamoto in \cite[Theorem 4.1]{sakamoto2011initial}. 
We will present the proof of (iii) subsequently.
\begin{lemma}\label{lem:op}
Let $F(t)$ and $E(t)$ be the solution operators defined in \eqref{eqn:FE-MLcop}.
Then they satisfy the following properties for all $t>0$
\begin{itemize}
\item[$\rm(i)$] 
$\|A^{\nu} F (t)v\|_{\dot H^p(\Omega)}+   
t^{1-\alpha}  \| A^{\nu}E (t)v  \|_{\dot H^p(\Omega)}\le c_1\,  \min(t^{-\alpha},t^{-\nu\alpha})\|v\|_{\dot H^p(\Omega)}$ with $0\le \nu\le 1$, $p\in \mathbb{R}$;
\item[$\rm(ii)$] $\| F(t)^{-1} v \|_{L^2(\Omega)}\le c_2 \, (1+ t^\alpha) \| v \|_{\dH2}$ for all $v\in \dH 2$;
\item[$\rm(iii)$] $ \|A^{-\nu}F(T)^{-1}E(t)v\|_{L^2(\Omega)}\le c_3(t^{\alpha-1}+t^{\alpha\nu-1}T^\alpha) \|v\|_{L^2(\Omega)} $  with $0\le \nu\le 1$.  

\end{itemize}
The constants $c_1$, $c_2$ and $c_3$ are independent of $t$.
\end{lemma}
\begin{proof} 
 We have the following equivalence formulas of the solution operators $F(t)$ and $E(t)$
\begin{equation*}
\begin{aligned}
F(t) v &= \sum_{j=1}^\infty E_{\alpha,1}(-\lambda_jt^\alpha)(v,\fy_j)\fy_j, \quad
E(t)v &= \sum_{j=1}^\infty t^{\alpha-1}E_{\alpha,\alpha}(-\lambda_j t^\alpha)(v,\fy_j)\fy_j,
\end{aligned}
\end{equation*}
for any $v\in L^2(\Omega)$, where $E_{\alpha,\beta}(z)$ denotes the two-parameter Mittag--Leffler function.  It is well-known that, with $\alpha\in(0,1)$, there hold \cite[Theorem 3.3 and Corollary 3.3]{Jin:book2021}  for all $t\ge0$
\begin{equation*}
   0\le E_{\alpha,\alpha}(-t) \le \frac{c}{1+t}\quad \text{and}\quad\frac{1}{1+\Gamma(1-\alpha)t}\le E_{\alpha,1}(-t)\le \frac{1}{1+\Gamma(1+\alpha)^{-1}t}.
\end{equation*}
  Therefore, we can obtain
\begin{equation*}
   \begin{aligned}
    & \|A^{-
    \nu}F(T)^{-1}E(t)v\|^2_{L^2(\Omega)}\le c\sum_{n=1}^{\infty}\Big|\frac{(1+\lambda_nT^\alpha)t^{\alpha-1}}{\lambda_n^\nu(1+\lambda_nt^\alpha)}\Big|^2(v,\fy_n)^2
     \\=&c\left(\sum_{n\in \{\lambda_nT^\alpha\le 1\}}\bigg|\frac{(1+\lambda_nT^\alpha)t^{\alpha-1}}{\lambda_n^\nu(1+\lambda_nt^\alpha)}\bigg|^2(v,\fy_n)^2 +\sum_{n\in \{\lambda_nT^\alpha> 1\}}\bigg|\frac{(1+\lambda_nT^\alpha)t^{\alpha-1}}{\lambda_n^\nu(1+\lambda_nt^\alpha)}\bigg|^2(v,\fy_n)^2\right)\\
     \le& c\left(\sum_{n\in \{\lambda_nT^\alpha\le 1\}}t^{2\alpha-2}(v,\fy_n)^2 +\sum_{n\in \{\lambda_nT^\alpha> 1\}}\bigg|\frac{(\lambda_n^{1-\nu}T^\alpha)t^{\alpha-1}}{(1+\lambda_nt^\alpha)^{1-\nu}(1+\lambda_nt^\alpha)^{\nu}}\bigg|^2(v,\fy_n)^2\right)\\
      \le& c\left(\sum_{n\in \{\lambda_nT^\alpha\le 1\}}t^{2\alpha-2}(v,\fy_n)^2 +\sum_{n\in \{\lambda_nT^\alpha> 1\}}T^{2\alpha} t^{2\alpha\nu-2}(v,\fy_n)^2\right)\\
      \le& c\left(t^{2\alpha-2} +t^{2\alpha\nu-2}T^{2\alpha}\right)\sum_{n=1}^{\infty}(v,\fy_n)^2.
\end{aligned} 
\end{equation*}
This completes the proof of the desired estimate  (iii).
\end{proof}

 
In our analysis, we employ a generalized version of Gronwall's inequality, which is given in the following lemma. Although the proof is available in \cite[Lemma 1]{ChenThomeeWahlbin:1992}, we provide a detailed proof that highlights how the constants explicitly depend on $T$ and $\beta_0$. This explicit dependence is of particular significance for the stability analysis of the inverse problem we are examining.
 
\begin{lemma}\label{lemma:Gronwall}
Assume that $y$ is a nonnegative function in $L^1(0,T )$ which satisfies
\begin{equation}\label{eqn:gron}
y(t)\le b(t) + \beta_0 \int_0^t (t-s)^{\alpha-1} y(s)\ds\quad \text{for}\quad  t\in (0,T] ,
\end{equation}
where $b(t)\ge 0,\ \beta_0 \ge 0$, and $0<\alpha<1$. There exists a constant $c_{\al}$ independent of $T$ and $\beta_0$, such that
\begin{equation*}
y(t)\le b(t) + c_{\al}\beta_0 K(\beta_0T^\alpha) \int_0^t (t-s)^{\alpha-1} b(s)\ds\quad \text{for}\quad   t \in (0,T],
\end{equation*}
where the function $K(s)$ is given by 
\begin{equation} \label{eqn:gbeatT}
K(s)=  \frac{1-s^{i-1}}{1-s} + \exp(c_{\al}s^i) \Big( 
s^{i-1} + \frac{s^i-s^{2i-1}}{\alpha(1-s)} \Big)\quad \text{for all}~s\neq1
\end{equation}
with $i=\lceil \frac{1}{\al} \rceil$ and $K(1) = \lim_{s\rightarrow 1} K(s)$.
\end{lemma} 
\begin{proof}
Let $K_1(s)=\beta_0s^{\al-1}$ for $0<s<T$ and $(K_1*f)(t)=\int_0^tK_1(t-s)f(s)\ds$. With $K_i$ the kernel of the $i$ times  iterated convolution, we have $K_i(s)\le c(i,\al)\beta_0^is^{i\al-1}$,
and we can see that 
$$ (K_i*b)(t)\le c\beta_0^{i-1}T^{(i-1)\al}(K_1*b)(t)\quad \text{for}  ~~2\le i\le \lceil \frac{1}{\al} \rceil.$$
Hence, applying the convolution with kernel $K_1$ on the relation \eqref{eqn:gron}
$i$ times in succession, we deduce, assuming $\beta_0 T^{\alpha} \neq 1$, 
\begin{align*}
    y(t)\le b(t)+c\frac{1-(\beta_0T^{\al})^{i-1}}{1-\beta_0T^{\al}} (K_1*b)(t)+(K_i*y)(t).
\end{align*}
When $i=\lceil \frac{1}{\al} \rceil$,  we have  $i\al-1\ge0$ and 
 $  (K_i*y)(t)\le c\beta_0^iT^{i\al-1}\int_0^t y(s)\ds.$ Then we arrive at
\begin{align*}
    y(t)\le b(t)+c\frac{1-(\beta_0T^{\al})^{i-1}}{1-\beta_0T^{\al}} (K_1*b)(t)+c\beta_0^iT^{i\al-1}\int_0^t y(s)\ds.
\end{align*}
Using the standard Gronwall's inequality gives
\begin{align*}
    y(t)\le&   b(t)+c\frac{1-(\beta_0T^{\al})^{i-1}}{1-\beta_0T^{\al}} (K_1*b)(t)\\&+ c\beta_0^iT^{i\al-1}\exp(c\beta_0^i T^{i\al})\int_0^t[ b(s)+\frac{1-(\beta_0T^{\al})^{i-1}}{1-\beta_0T^{\al}} (K_1*b)(s)]\ds\\
    \le & b(t)+c \beta_0K(\beta_0 T^\alpha)\int_0^t(t-s)^{\al-1}b(s)\ds.
\end{align*}
In the second inequality, we use the facts  
$$\int_0^tb(s)\ds \le \beta_0^{-1}T^{1-\al} (K_1*b)(t)\quad \text{and}\quad \int_0^t(K_1*b)(s)\ \ds\le \frac{T}{\al}(K_1*b)(t).$$
The estimate for the case that $\beta_0T^{\al}=1$ follows analogously. 
\end{proof}

We now state the well-posedness and regularity of the nonlinear time-fractional diffusion problem~\eqref{eqn:fde}.

\begin{lemma}\label{lem:Reg}
Let $u_0\in \dH{p}$ with $ p\in[0,2] $, and let $f(u)$ satisfy the Lipschitz assumption~\eqref{eqn:lipconstant}.
Then the problem \eqref{eqn:fde} has a unique mild solution 
$u\in C([0,T];L^2\II)\cap C((0,T];\dH{2})$, given by \eqref{eqn:sol-rep}, satisfying for all $t\in(0,T]$
\begin{align}\label{lem:reg-dt}
 \|\partial_tu(t)\|_{L^2\II}\le c_Tt^{p\alpha/2-1}\| u_0 \|_{\dH p} ~ \text{and}~
\|u(t)\|_{\dH2}\le c_T t^{-(1-\frac{p}2)\alpha} \| u_0 \|_{\dH p}.
\end{align}
Here the constant $c_T$ depends on $T$ and $L$. 
\end{lemma}

\begin{proof}
The well-posedness of the problem is established in  
\cite[Theorem 3.1 and 3.2]{Karaa:2019}.
The proof of the first \textsl{a priori} estimate in \eqref{lem:reg-dt} can be found in \cite[Theorem 3.1 and 3.2]{Karaa:2019} for $0<p\le2$, and in \cite[Theorem 3.2]{LiMa:2022} for the case $p=0$. The second estimate is derived as follows. Using  solution representation  \eqref{eqn:sol-rep} and identity $AE(t) = - F'(t)$,  gives
\begin{align*}
    Au(t)
=&AF(t)u_0+\int_0^{t_0}AE(t-s)f(u(s)
)\ \ds-\int_{t_0}^tF'(t-s)f(u(s)
)\ \ds\\=&AF(t)u_0+\int_0^{t_0}AE(t-s)f(u(s)
)\ \ds\\&+F(t-t_0)f(u(t_0)
)-f(u(t))+\int_{t_0}^t F(t-s)f'(u(s))u'(s)\ds=:\sum_{i=1}^5{\rm I_i}.
\end{align*}
 Using Lemma~\ref{lem:op} and the Lipschitz  condition \eqref{eqn:lipconstant} and  setting $t_0=\frac{t}{2}$, we obtain for $0\le p\le2$:
\begin{align*}
    &\|{\rm I_1}\|_{L^2\II}\le c_T t^{-(1-\frac{p}2)\alpha} \| u_0 \|_{\dH p},  \quad  \|{\rm I_2}\|_{L^2\II}\le cL\int_0^{\frac t2} (t-s)^{-1}\| u(s) \|_{L^2\II} \ds\le c_T \|u_0\|_{L^2\II},\\&\|{\rm I_3+I_4}\|_{L^2\II}\le c (\|u(t_0)\|_{L^2\II}+\|u(t)\|_{L^2\II})\le c\|u_0\|_{L^2\II},\\&
    \|{\rm I_5}\|_{L^2\II}\le cL\int _{\frac t2}^t\| u'(s) \|_{L^2\II}\ds\le cL\int _{\frac t2}^t s^{-1}\| u_0 \|_{L^2\II}\ds\le c_T\| u_0 \|_{L^2\II}.
\end{align*}
Combining these results leads to the desired conclusions.
\end{proof}


The same argument as \cite[Theorem 3.1 and 3.2]{Karaa:2019} also leads to the well-posedness in the case of the very weak initial data, which is presented in the following corollary. The detailed proof of the estimates is presented in the Appendix.
\begin{corollary}\label{lem:reg-weak}
Let $u_0 \in \dH{-\mu}$ with $\mu\in(0,1]$ and let $f(u)$ satisfy the Lipschitz assumption \eqref{eqn:lipconstant}.
Then the problem \eqref{eqn:fde} has a unique mild solution \eqref{eqn:sol-rep} such that 
$u\in C([0,T];\dH{-\mu})\cap C((0,T];\dH{2-\mu})$. Moreover, we have the following estimates
\begin{align*}
    \|\partial_tu(t)\|_{L^2\II}\le c_Tt^{-\alpha\mu/2-1}\| u_0 \|_{\dH {-\mu}}, \quad \| A\int_0^tE(t-s)f(u(s))\ds\|_{L^2\II}\le  c_Tt^{-\al\mu/2} \|u_0\|_{\dH{-\mu}}.
\end{align*}
\end{corollary}


\subsection{Well-posedness of the backward problem.}\label{sub2.2} 
Next, we aim to show the well-posedness of the backward nonlinear subdiffusion problem: for  a fixed parameter $\mu\in (0,1]$,
look for a initial data $u_0 = u(0) \in \dH{-\mu}$, 
such that $u\in C([0,T];\dH{-\mu})\cap C((0,T];\dH{2-\mu})$ satisfying
\begin{equation}\label{eq:back_non}
   \partial^\alpha_t u+ A u =f(u) \quad \text{for all}~ t\times(0,T] \quad \text{and}\quad u(T)=g(x).
\end{equation}
Using the solution representation \eqref{eqn:sol-rep} gives 
\begin{equation*}
    g(x)=F(T)u_0+\int_0^TE(T-s)f(u(s))\ds = F(T)u_0+\int_0^TE(T-s)f(S(s)u_0)\ds,
\end{equation*}
which leads to the relation
\begin{align}\label{eq:back_u0}
u_0=F(T)^{-1}\Big(g-\int_0^T E(T-s)f(u(s)) \ds \Big)
= F(T)^{-1}\Big(g-\int_0^T E(T-s)f(S(s)u_0) \ds \Big).
\end{align}

We will investigate the existence and uniqueness of $u_0$ satisfying \eqref{eq:back_u0}, which pertains to the well-posedness of the backward problem \eqref{eq:back_non}. 
Note that the relation \eqref{eq:back_u0} naturally provides a fixed point iteration where the initial value $u_0$ is the fixed point. Then the existence and uniqueness of $u_0$ follows from the contraction mapping theorem.
The following lemma serves as an important preliminary to the proof of the contraction mapping.
\begin{lemma}\label{lam:phi12}
Let $S(t)$ be the solution operator defined in \eqref{eqn:sol-rep}, and let $L$ be the Lipschitz constant in the \eqref{eqn:lipconstant}. Then, for any $\phi_1, \phi_2 \in \dH{-\mu}$ with $\mu\in[0,1]$ the following inequality holds:
\begin{align*}
    \| (S(t)\phi_1-S(t)\phi_2)\|_{L^2\II}
     \le B_0(\al,T,L,\nu) t^{-\al\mu/2}\|\phi_1-\phi_2\|_{\dH{-\mu}}  \quad \text{for}\quad   t \in (0,T].
\end{align*}
\end{lemma}
\begin{proof} 
From the relation \eqref{eq:sol-rep-ope} and Lemma \ref{lem:op} (i), we have 
\begin{align*}
     & \|S(t)\phi_1-S(t)\phi_2\|_{L^2\II}\\\le& \|F(t) (\phi_1-\phi_2)\|_{L^2\II}+\|\int_0^tE(t-s)[f(S(s)\phi_1)-f(S(s)\phi_2)]\ \ds\|_{L^2\II}\\
      \le &c_1 t^{-\al\mu/2}\|\phi_1-\phi_2\|_{\dH{-\mu}}+ c_1 L\int_0^t(t-s)^{\alpha-1} \|S(s)\phi_1-S(s)\phi_2\|_{L^2\II}\ds.
\end{align*}
Then the Gronwall's inequality in Lemma \ref{lemma:Gronwall} leads to
\begin{equation}\label{eqinter:l2}
    \begin{aligned}
    \|(S(t)\phi_1-S(t)\phi_2)\|_{L^2\II}&
    \le\left( c_1 t^{-\al\mu/2}+ c_{\al}c_1L K(c_1 LT^\alpha)\int_0^t(t-s)^{\alpha-1} s^{-\al\mu/2}\ds\right)\|\phi_1-\phi_2\|_{\dH{-\mu}}\\
    &=\left( c_1 t^{-\al\mu/2}+ c_{1,\al,\mu}L K(c_1 LT^\alpha)t^{\al-\al\mu/2}\right)\|\phi_1-\phi_2\|_{\dH{-\mu}}\\
     &=: B_0(\al,T,L,\mu) t^{-\al\mu/2}\|\phi_1-\phi_2\|_{\dH{-\mu}}.
\end{aligned} 
\end{equation}
This completes the proof of this lemma.
\end{proof}

The following theorem establishes the existence and uniqueness of the solution to the backward problem associated with the semilinear subdiffusion model. Additionally, the argument advances to provide a stability estimate comparable with those found in linear models.

To this end, for a given $g\in \dH{2-\mu}$, we define a mapping $M: \dH{-\mu} \rightarrow \dH{-\mu}$ by
\begin{equation}\label{eqn:M}
    M \phi =F(T)^{-1}\left(g(x)-\int_0^T E(T-s)f(S(s)\phi)\ \ds
    \right) ~~\text{for any}~\phi\in\dH{-\mu},
\end{equation}
where $S(t)$ is the solution operator defined in \eqref{eqn:sol-rep}.
Note that the backward problem \eqref{eq:back_non} is equivalent to finding a fixed point of the operator $M$. With the help of Lemmas 
\ref{lem:Reg}-\ref{lam:phi12}, we are ready to show that $M$ is a contraction mapping and hence possesses a unique fixed point.

\begin{theorem}\label{thm:stability} 
For a fixed parameter $\mu\in(0,1]$, there exists a threshold $T_*>0$ (depending on the parameter $\mu$, the fractional order $\alpha$, the Lipschitz constant $L$ in \eqref{eqn:lipconstant}) such that 
for any $T\in(0,T_*)$,
there holds the following stability estimate for $\phi_1,\phi_2\in \dH{-\mu}$:
\begin{equation}\label{eqn:stab} 
\|\phi_1-\phi_2\|_{\dH{-\mu}}\le c\|S(T)\phi_1-S(T)\phi_2\|_{\dH{2-\mu}},
\end{equation}
where $S(T)$ is the solution operator defined in \eqref{eqn:sol-rep}.
\end{theorem}
\begin{proof} 
First of all, we show that the operator $M$ is a contraction mapping in $\dH{-\mu}$. 
For a given $g\in \dot H^{2-\mu}(\Omega)$, based on Lemma \ref{lem:Reg}, we can conclude that $M\phi \in \dH{-\mu}$ for any $\phi\in\dH{-\mu}$,
and hence the operator is well-defined. Additionally, using Lemma \ref{lem:op} and the Lipschitz condition \eqref{eqn:lipconstant}, we conclude that
\begin{equation*}\label{eqn:contraction-00}
 \begin{aligned}
   &\| M (\phi_1-\phi_2)\|_{\dH{-\mu}} 
   \le \int_0^T \|A^{-\frac{\mu}{2}}F(T)^{-1} E(T-s)[f(S(s)\phi_1)-f(S(s)\phi_2)]\|_{L^2\II}\ \ds \\
   \le &\int_0^T  \|A^{-\frac\mu2}F(T)^{-1} E(T-s)\| \,\| f(S(s)\phi_1)-f(S(s)\phi_2)\|_{L^2\II} \ \ds\\
   \le & \, c_3  L\int_0^T [(T-s)^{\alpha-1}+(T-s)^{\alpha\mu/2-1}T^\alpha] \|S(s)\phi_1 - S(s)\phi_2 \|_{L^2\II} \ds.
\end{aligned}
\end{equation*}
Applying Lemma  \ref{lam:phi12}  gives
\begin{equation}\label{eqn:contraction-0}
 \begin{aligned}
  \| M (\phi_1-\phi_2)\|_{\dH{-\mu}}   \le  c_3  L B_0(\al,T,L,\mu) \int_0^T [(T-s)^{\alpha-1}+(T-s)^{\alpha\mu/2-1}T^\alpha] s^{-\alpha\mu/2}\ \ds \|\phi_1-\phi_2\|_{\dH{-\mu}} .
\end{aligned}
\end{equation}


Now we define the function $  B_{\mu}(\cdot)$ as: 
\begin{equation}\label{eqn:contraction}
 \begin{aligned}
   B_{\mu}(T) =  c_3  L B_0(\al,T,L,\mu) \int_0^T [(T-s)^{\alpha-1}+(T-s)^{\alpha\mu/2-1}T^\alpha] s^{-\alpha\mu/2}\ \ds.
\end{aligned}
\end{equation}
Let $T_*$ be the constant such that $ B_{\mu}(T_*) <1.$
Note that $ B_{\mu}(T)$ is increasing with respect to $T$. Therefore, we conclude that, for any $T \in (0, T_*]$, the operator $M$ is a contraction, and hence admits a unique fixed point. As a result, the backward problem \eqref{eq:back_non} admits a unique solution in $\dH{-\mu}$.

Finally, we show the stability estimate. Let $g_i = S(T)\phi_i$ for  $i=1,2$. Then we observe
\begin{equation*}
 \begin{aligned}
 \phi_1-\phi_2 = F(T)^{-1} (g_1 - g_2) -  F(T)^{-1} \int_0^T E(T-s) \Big(f(S(s)\phi_1)-f(S(s)\phi_2)\Big)\,\ds.
\end{aligned}
\end{equation*}
Let $T_*$ be the constant such that $B_\mu(T_*)<1$ with the function $B_\mu(\cdot)$ defined in \eqref{eqn:contraction}.
Taking $\dot H^{-\mu}$ norm on both sides of the above relation, using Lemma \ref{lem:op} and the argument in the estimate \eqref{eqn:contraction-0}, we obtain for any $T\in(0,T_*)$
\begin{equation*} 
 \begin{aligned}
\| \phi_1-\phi_2\|_{\dH{-\mu}}
&\le c \| g_1 - g_2 \|_{\dH{2-\mu}}  +  B_\mu(T) \|\phi_1-\phi_2\|_{L^2\II}\\
&\le  c \| g_1 - g_2 \|_{\dH{2-\mu}} +  B_\mu(T_*) \|\phi_1-\phi_2\|_{\dH{-\mu}}.
\end{aligned}
\end{equation*}
Then the desired stability estimate follows immediately from the fact that $B_{\mu}(T_*) <1.$
\end{proof}

\begin{remark}\label{rem:-mu}
The stability estimate in Theorem \ref{thm:stability} implies that the backward problem of the semilinear subdiffusion model \eqref{eqn:fde} is mildly ill-posed. Note that Theorem \ref{thm:stability} requires $\mu > 0$. This requirement arises from the fact that 
$$\| F(T)^{-1} E(T-s) \| \le c \Big((T-s)^{\alpha-1}+(T-s)^{-1} T^\alpha\Big),$$
which is non-integrable.  Nevertheless, a similar argument can be applied to handle the case of $\mu = 0$. In particular, we can show that  
\begin{equation*} 
 \|\phi_1-\phi_2\|_{L^2\II}\le c\|S(T)\phi_1-S(T)\phi_2\|_{\dH{2}}
\end{equation*}
for sufficiently small $T$,  provided that the following Lipschitz condition holds: 
\begin{equation}\label{eqn:lipconstant-2}
\|f(u) - f(v)\|_{\dH \nu} \le L \|u - v\|_{\dH \nu} \quad \text{for all} \quad u, v \in \dH \nu~~ \text{and}~~\nu\in[0,\beta)
\end{equation}
with some $\beta\in(0,1)$. However, this Lipschitz condition is far more restrictive than the standard condition in \eqref{eqn:lipconstant}. It remains unclear how to establish stability for \(\mu = 0\) under the standard Lipschitz condition \eqref{eqn:lipconstant}, and this warrants further theoretical investigation.
\end{remark}

\section{Regularization and convergence analysis}\label{sec:3}
From the stability estimate \eqref{eqn:stab}, we observe that the backward problem exhibits mild ill-posedness; that is, it experiences a loss equivalent to a second-order derivative. Furthermore, the practical observational data, denoted by $g_\delta$, often contains noise, as indicated by \eqref{eqn:noisy data}, implying that the empirical observations fail to function in the $\dot{H}^{2-\mu}$  space, for fixed  $\mu\in(0,1]$. Consequently, regularization is necessary to solve the backward problem.

In this section, we investigate a straightforward regularization approach utilizing the quasi-boundary value method \cite{hao2019stability,yang2013solving}. 
Let $\tud(t)\in  C([0,T];\dH{-\mu})\cap C((0,T];\dH{2-\mu})$,  be the function satisfying
\begin{equation}\label{eqn:fde:reg:noisy} 
\begin{aligned}
   \partial^\alpha_t u_\gamma^\delta + A u_\gamma^\delta =f(u_\gamma^\delta) \quad\text{for all}\quad t \in (0,T]\quad \text{and}\quad 
 \gamma u_\gamma^\delta(0)+ u_\gamma^\delta(T)&=g_\delta.
\end{aligned}
\end{equation}
Here $\gamma$ denotes a positive regularization parameter. Then we aim to establish an error estimate for  $\tud(0) - u(0)$.
To this end, we introduce an auxiliary function $\tu(t) \in C([0,T];\dH{-\mu})\cap C((0,T];\dH{2-\mu}) $ satisfying
\begin{equation}\label{eqn:fde-back-reg}
\begin{aligned}
   \partial^\alpha_t u_\gamma-\Delta u_\gamma=f(u_\gamma) \quad\text{for all}\quad t \in (0,T]\quad \text{and}\quad  
   \gamma u_\gamma(0)+ u_\gamma(T)=g.
\end{aligned}
\end{equation}
Utilizing  the solution representation~\eqref{eqn:sol-rep}  gives
 \begin{align}
     u_\gamma(0)& = (\gamma I+F(T))^{-1}\Big(g-\int_0^T E(T-s)f(S(s)u_\gamma(0))\ \ds
    \Big),\label{ugamma}\\
     u_\gamma^\delta(0)& = (\gamma I+F(T))^{-1}\Big(g_\delta-\int_0^T E(T-s)f(S(s)u_\gamma^\delta(0)) \,\ds 
  \Big).\label{ugammadelta}
 \end{align}
The following lemma elucidates the smoothing properties of the solution operator $(\gamma I+F(T))^{-1}$. Since the proof is identical to that presented in \cite[Lemma 3.3]{zhang2022backward}, it is omitted here to avoid redundancy.
\begin{lemma}\label{lemma:gammainv}
 For $ p\le q\le p+2$, the following estimates hold for any $\gamma\in (0,1]$:
    \begin{align*}
        \|(\gamma I +F(T))^{-1}v\|_{\dot H^p(\Omega)}\le c\gamma^{-(1+\frac{p-q}{2})}\|v\|_{\dot H^q(\Omega)}
        \quad \text{and}\quad
         \|F(T)(\gamma I +F(T))^{-1}v\|_{L^2\II}\le \|  v \|_{L^2\II},
    \end{align*}
 where the constant $c$ is independent of $\gamma$, but may depend on $T$. 
\end{lemma}

The next lemma provides an error bound  $u_\gamma(0)-u_0$.
\begin{lemma}\label{lemma:gammainequ} 
Suppose that $u$ is the exact solution to the backward problem \eqref{eq:back_non} with the terminal data $g$,
while $\tu$ is the solution to the regularized problem \eqref{eqn:fde-back-reg}. For a fixed parameter $\mu\in(0,1]$, 
let $T_*$ be the constant such that $B_\mu(T_*)<1$ with the function $B_\mu(\cdot)$ defined in \eqref{eqn:contraction}, and assume that $T<T_*$.
If $ u_0 \in \dH {-\mu+q}$ with $q\in (0,2]$,  there holds the estimate 
\begin{align}\label{eqn:err-eg}
 \|u_\gamma(0)-u_0\|_{\dH{-\mu}}\le c\gamma^{\frac{q}{2}}\|u_0\|_{\dot H^{-\mu+q}(\Omega)}.
\end{align}
Moreover,  in case that $u_0\in \dH{-\mu}$, there holds
\begin{equation}\label{eqn:err-eg2}
        \lim_{\gamma\rightarrow 0^+}  \|u_\gamma(0) - u_0\|_{\dH{-\mu}} =0.
\end{equation}
\end{lemma}
\begin{proof}
Let $ e_\gamma(t) =  u_\gamma(t)-u(t)$. Note that  the function $e_\gamma(t)$ satisfies
\begin{equation*}
\begin{aligned}
   \partial^\alpha_t e_\gamma + A e_\gamma=f(u_\gamma)-f(u) \quad\text{with}\quad 
 \gamma e_\gamma(0)+ e_\gamma(T)=-\gamma u_0.
\end{aligned}
\end{equation*}
Using the solution representation  \eqref{eqn:sol-rep} yields
\begin{equation*}
    e_\gamma(0)=(\gamma I+F(T))^{-1}\Big(-\gamma u_0-\int_0^T E(T-s)[f(S(s)u_\gamma(0))-f(S(s)u_0)]\ \ds  \Big).
\end{equation*}
From Lemma \ref{lemma:gammainv}  and the fact that 
\begin{equation}\label{eqn:gf-1}
\| (\gamma I + F(T))^{-1} v \|_{\dH{-\mu}}\le \| F(T)^{-1} v \|_{\dH{-\mu}}~~\text{for all} ~v\in \dH{2-\mu},
\end{equation}
we obtain 
\begin{equation*}
\begin{aligned}
    \|e_\gamma(0)\|_{\dH{-\mu}}
    &\le c\gamma^{\frac{q}{2}}\|u_0\|_{\dot H^{-\mu+q}(\Omega)}
    + \int_0^T \| A^{-\frac{\mu}{2}}F(T)^{-1} E(T-s)
    [f(S(s)u_\gamma(0))-f(S(s)u_0)]\|_{L^2\II}\ \ds.
    \end{aligned}
\end{equation*} 
Then the estimate \eqref{eqn:err-eg} is derived using the arguments presented in the proof of stability \eqref{eqn:stab}.

Next, we turn to the case that $u_0 = u(0)\in \dH{-\mu}$. For an arbitrary function 
$\Tilde{u}_0 \in \dot H^{2-\mu} (\Omega)$, let $\Tilde{u}(t)$ and $\Tilde{u}_\gamma(t)$ be the functions  respectively satisfying 
\begin{align*}
   \partial^\alpha_t  \Tilde u +A \Tilde u=f(\Tilde u)&\quad\text{for all}\quad t \in (0,T]\quad \text{with}\quad  
 \Tilde u(0) =\Tilde u_0,\\
  \partial^\alpha_t\Tilde  u_\gamma + A \Tilde u_\gamma=f(\Tilde u_\gamma)&\quad\text{for all}\quad t \in (0,T]\quad \text{with}\quad  
 \gamma \Tilde u_\gamma(0)+ \Tilde u_\gamma(T)=  \Tilde u(T).
\end{align*}
We have proved that $   \| \Tilde u_\gamma(0)-
    \Tilde u_0\|_{\dH{-\mu}} \le c\gamma \| \Tilde u_0\|_{\dot H^{2-\mu}(\Omega)}$.
Meanwhile, applying the argument in Theorem \ref{thm:stability} and Lemma \ref{lam:phi12} yields $ \| \Tilde u_\gamma(0)-
     u_\gamma(0)\|_{\dH{-\mu}} \le c \|u_0-\Tilde u_0\|_{\dH{-\mu}}.$ 
As a result, we apply triangle inequality to obtain 
\begin{align*}
     \|  u_\gamma(0)-
     u_0\|_{\dH{-\mu}}&\le \|  u_0-
    \Tilde u_0\|_{\dH{-\mu}}+\|  u_\gamma(0)-
    \Tilde u_\gamma(0)\|_{\dH{-\mu}}+\| \Tilde u_\gamma(0)-
    \Tilde u_0\|_{\dH{-\mu}}\\
    &\le c\|  u_0-
    \Tilde u_0\| _{\dH{-\mu}}+c\gamma\|\Tilde u_0\|_{\dH{2-\mu}}.
\end{align*}
Let $\varepsilon$ be an arbitrarily small number. Using the density of $\dH{2-\mu}$ in $\dH{-\mu}$, we choose $\Tilde u_0$ such that $c\| u_0- \Tilde u_0\|_{\dH{-\mu}}\le \frac{\varepsilon}{2}$.  Moreover, let $\gamma_0$ be the constant that $c\gamma_0\|\Tilde u_0\|_{\dot H^{2-\mu}(\Omega)}<\frac{\varepsilon}{2}$. Therefore, for all $\gamma\le \gamma_0$, we have $ \|  u_\gamma(0)-  u_0\|_{\dH{-\mu}}\le \varepsilon$. Then we obtain \eqref{eqn:err-eg2} and hence the proof is complete.
\end{proof}

\begin{theorem} \label{thm:err-reg}
Suppose that $u$ is the exact solution to the backward problem \eqref{eq:back_non} with the terminal data $g$,
while $\tud$ is the solution to the regularized problem \eqref{eqn:fde:reg:noisy}. 
For a fixed $\mu\in(0,1]$, 
let $T_*$ be the constant such that $B_\mu(T_*)<1$ with the function $B_\mu(\cdot)$ defined in \eqref{eqn:contraction}, and assume that $T<T_*$.
If $ u_0 \in \dH {-\mu+q}$ with $q\in (0,2]$, we have the estimate \begin{equation*}
         \|u_\gamma^\delta(0)-u_0\|_{\dH{-\mu}}\le c\left (\gamma^{-1}\delta +\gamma^\frac{q}{2}\right).
    \end{equation*}
Moreover, if $ u_0 \in \dH {-\mu}$, then there holds  
    \begin{equation*}
         \|u_\gamma^\delta(0)-u_0\|_{\dH{-\mu}}\rightarrow 0\quad \text{as} \quad\delta,\ \gamma\rightarrow 0\  \text{and } \ \frac{\delta}{\gamma}\rightarrow 0.
    \end{equation*}
\end{theorem}
\begin{proof}
We employ the splitting 
    \begin{equation*}
        u^\delta_\gamma(t)-u(t)=\left(u^\delta_\gamma(t)-u_\gamma(t)\right)+\left(u_\gamma(t)-u(t)\right)=e_\delta (t)+e_\gamma(t).
    \end{equation*}
Applying the solution representations \eqref{ugamma}--\eqref{ugammadelta}, Lemma \ref{lemma:gammainv}, the assumption \eqref{eqn:noisy data}, and the fact~\eqref{eqn:gf-1} leads to
\begin{equation*}
\begin{aligned}
    \|e_\delta(0)\|_{\dH{-\mu}} 
     \le & \|(\gamma I+F(T))^{-1}(g_\delta-g) \|_{\dH{-\mu}} \\
     &+  \int_0^T\|(\gamma I+F(T))^{-1} E(T-s)
     [f(S(s)u^\delta_\gamma(0))-f(S(s)u_\gamma(0))] \|_{\dH{-\mu}}\ \ds\\
   \le & c\gamma^{-1}\delta+\int_0^T \|F(T)^{-1} E(T-s)[f(S(s)u^\delta_\gamma(0))-f(S(s)u_\gamma(0))] \|_{\dH{-\mu}} \ \ds  .
    \end{aligned}
\end{equation*}
Then using the argument in the proof of the stability estimate \eqref{eqn:stab} yields  
$ \|e_\delta(0)\|_{\dH{-\mu}}\le  c\gamma^{-1}\delta$.
 Combining this estimate with Lemma \ref{lemma:gammainequ} leads to the desired result.
\end{proof}

At the end of this section, we present the following regularity of $u_\gamma(0)$, which is extensively used in the numerical analysis in Section \ref{sec:4}.
\begin{lemma}\label{lem:regugamma} Let $\tu$ be the solution to the regularized problem \eqref{eqn:fde-back-reg}.  For a fixed parameter $\mu\in(0,1]$,
let $T_*$ be the constant such that $B_{\mu}(T_*)<1$ with the function $B_{\mu}(\cdot)$ defined in \eqref{eqn:contraction}, and assume that $T<T_*$.
Then for $p\in [-\mu ,2-\mu]$, there holds
\begin{equation*}
    \|u_\gamma(0)\|_{\dot H^p(\Omega)}\le c_T\gamma^{-\frac{p+\mu}{2}} \|u_0\|_{\dH{-\mu}}.
\end{equation*}
\end{lemma}
\begin{proof}
From the relation \eqref{ugamma}, and the estimate \eqref{eqn:gf-1}, we derive
 \begin{align*}
\| u_\gamma(0)\|_{\dH{-\mu}}
\le&   \|F(T)^{-1}g\|_{\dH{-\mu}}  +  \int_0^T \|F(T)^{-1} E(T-s)[f(u_\gamma(s))-f(0)] \|_{\dH{-\mu}}\  \ds.
\end{align*} 
Applying Lemma \ref{lem:op}, Lemma \ref{lem:Reg} and  Lemma \ref{lemma:gammainv} gives
 \begin{align*}
 \|F(T)^{-1}g\|_{\dH{-\mu}}  \le  c\|  g \|_{\dH{2-\mu}}  \le c_T \| u_0 \|_{\dH{-\mu}} .
\end{align*} 
Then, provided that $T<T_*$, the argument in the proof of the stability estimate \eqref{eqn:stab} yields that 
 \begin{align} \label{eqn:ug0}
 \|u_\gamma(0)\|_{\dH{-\mu}} \le  c_T  \| u_0 \|_{\dH{-\mu}}  .
 \end{align} 
Meanwhile, using $\gamma u_\gamma(0)+ u_\gamma(T)=g = u(T)$ and the regularity estimate in Lemma \ref{lem:Reg} leads to 
\begin{align*}
\| u_\gamma(0)\|_{\dH{2-\mu}} 
&\le\gamma^{-1}( \| u(T)\|_{\dH{2-\mu}} + \| u_\gamma(T)\|_{\dH{2-\mu}} )\\
& \le c\gamma^{-1}(\|u_0\|_{\dH{-\mu}}+ \|u_\gamma(0)\|_{\dH{-\mu}})\le  c\gamma^{-1} \|u_0\|_{\dH{-\mu}},
 \end{align*} 
where for the last inequality we use the proved estimate \eqref{eqn:ug0}.
 Then the intermediate results with $p\in(-\mu ,2-\mu)$ followed by the complex interpolation.
\end{proof}

\section{Fully discretization scheme and error analysis}\label{sec:4} 
This section will focus on proposing and analyzing a fully discrete scheme for solving the backward problem~\eqref{eq:back_non}. Initially, we study the semidiscrete scheme using the finite element methods. 
The semidiscrete solution is crucial in the analysis of the fully discrete scheme.

\subsection{Semidiscrete scheme for solving the problem} 
We begin by studying the semidiscrete scheme using finite element methods. 
Let  ${\{\mathcal{T}_h\}}_{0<h<1}$ represent a  family
of shape-regular and quasi-uniform partitions of the domain $\Omega$ into $d$-simplexes, known as
finite elements, with $h$ representing the maximum diameter of the elements.
We consider the finite element space $X_h$ defined by
\begin{equation*}\label{eqn:fem-space}
  X_h =\left\{\chi\in C(\bar\Omega)\cap H_0^1: \ \chi|_{K}\in P_1(K),
 \,\,\,\,\forall K \in \mathcal{T}_h\right\},
\end{equation*}
where $P_1(K)$ denotes the space of linear polynomials on $K$.
We then define the $L^2(\Omega)$ projection $P_h:L^2(\Omega)\to X_h$ and
Ritz projection $R_h:\dH1\to X_h$, respectively, defined by (recall that $(\cdot, \cdot)$
denotes the $L^2(\Omega)$ inner product)
\begin{align*}
    (P_h \psi,\chi) & =(\psi,\chi) \quad\forall~ \chi\in X_h,\psi\in L^2(\Omega),\\
    (\nabla R_h \psi,\nabla\chi) & =(\nabla \psi,\nabla\chi) \quad \forall ~\chi\in X_h, \psi\in \dot H^1(\Omega).
\end{align*}
The approximation properties of $R_h$ and $P_h$ are well known and can be found in \cite[Chapter 1]{thomee2007galerkin}:
\begin{equation*}
\begin{aligned}
\|P_h\psi-\psi\|_{L^2\II} +h\|\nabla(P_h\psi-\psi)\|_{L^2\II} & \leq ch^q\|\psi\|_{H^q(\Omega)}\quad \forall\psi\in \dH q, q=1,2,\\
\|R_h\psi-\psi\|_{L^2\II} +h\|\nabla(R_h\psi-\psi)\|_{L^2\II} & \leq ch^q\|\psi\|_{H^q(\Omega)}\quad \forall\psi\in \dH q, q=1,2.
\end{aligned}  
\end{equation*}
Moreover, we have the following negative norm estimate \cite[p. 69]{thomee2007galerkin}
\begin{equation}\label{eqn:Ph-neg}
\|P_h\psi-\psi\|_{\dH{-\nu}} \leq ch^2\|\psi\|_{\dH{2-\nu}} .
\end{equation}
The semidiscrete scheme for the direct  problem \eqref{eqn:fde}
is to find $ u_h (t)\in X_h$ such that
\begin{equation*}\label{fem}
( \Dal u_{h}(t),\chi)+ (\nabla u_h(t),\nabla \chi) = {(f(u_h(t)), \chi)},~~\forall \chi\in X_h,\ t\in (0,T]
 \quad \text{with}\quad u_h(0)=P_h u_0.
\end{equation*}
We now  introduce the negative discrete Laplacian $ A_h: X_h\to X_h$ such that
\begin{equation*}
  (A_h\psi,\chi)=(\nabla\psi,\nabla\chi)\quad\forall\psi,\,\chi\in X_h.
\end{equation*}
Then the spatially semidiscrete problem  \eqref{fem} could be written as
\begin{equation}\label{fem-operator}
   \Dal u_{h}(t)+A_h u_h(t) =P_hf(u_h(t)), \quad\forall t\in (0,T] \quad \mbox{with} \quad  u_h(0)=P_h u_0.
\end{equation}
Using the Laplace Transform, the semidiscrete solution  can be represented by 
\begin{equation}\label{eq:semisolutionrep}
    u_h(t)=F_h(t)u_h(0)+\int_0^t E_h(t-s)P_hf(u_h(s))\ \ds =: S_h(t) u_h(0),
\end{equation}
where
\begin{equation}\label{sol_dis_op}
\begin{aligned}
F_h(t) &= \frac{1}{2\pi i }\int_\contour e^{zt} z^{\alpha-1}(z^\alpha+A_h)^{-1}dz,\quad
E_h(t) &= \frac{1}{2\pi i}\int_\contour  e^{zt} (z^\alpha+A_h)^{-1} dz.
 \end{aligned}
\end{equation}

We recall the following inverse  inequality \cite[Lemma 2.2]{JinZhou:2023book} 
\begin{equation}\label{eqn:inverse}
\| \phi_h  \|_{L^2\II} \le c h^{-2\nu} \|  A_h^{-\nu} \phi_h \|_{L^2\II} \quad \text{for all}~~ \nu\ge0.
\end{equation}
Meanwhile, we note that the following norm equivalence \cite[Lemma 2.7]{JinZhou:2023book} 
\begin{equation}\label{eqn:equiv}
c \| \phi_h \|_{\dH{\nu}} \le \| A_h^{\frac{\nu}{2}} \phi_h \|_{L^2\II} \le C\| \phi_h \|_{\dH{\nu}},\quad \text{for all}~~ \nu\in[-1,1].
\end{equation}

The discrete operators $F_h(t)$ and $E_h(t)$ satisfy the following smoothing property,
whose proof is identical to that of Lemma~\ref{lem:op}.
\begin{lemma}\label{lem:op-d}
Then they satisfy the following properties for all $t>0$ and $v_h \in X_h$
\begin{itemize}
\item[$\rm(i)$] $\|A_h^{\nu} F_h (t)v_h\|_{L^2\II} +   
t^{1-\alpha}  \| A_h^{\nu}E_h (t)v_h  \|_{L^2\II} \le c t^{-\nu\alpha} \|v_h\|_{L^2\II} $ with $0\le \nu\le 1$;
\item[$\rm(ii)$] $\|F_h(t)^{-1} v_h \|_{L^2\II} \le c (1+ t^\alpha) \| A_h v_h \|_{L^2\II} .$
\end{itemize}
The constant $c$ is independent of  $t$.
\end{lemma}

The following lemma is a discrete analogue to Lemma \ref{lemma:gammainv}, the proof follows from spectral decomposition as well as the asymptotic behavior of Mittag--Leffler functions, and hence omitted here.
\begin{lemma}\label{lem:op-reg-semi}
Let $F_h(t)$ be the discrete solution operator defined in \eqref{sol_dis_op}.
For $v_h\in X_h$, we have
\begin{equation*}
\|   (\gamma I+F_h(T))^{-1}v_h \|_{L^2\II}  \le c \gamma^{-1}\| v_h \|_{L^2\II}~~\text{and}
\quad \|F_h(T)(\gamma I+F_h(T))^{-1}v_h\|_{L^2\II}\le \|v_h\|_{L^2\II},
\end{equation*}
where the constant $c$ is independent of $\gamma$, $h$, $t$ and $T$.
\end{lemma}

Using the same argument in the proof of Lemma \ref{lem:Reg}, 
we have the following regularity results for the semidiscrete problem \eqref{fem-operator}.
\begin{lemma}\label{THM:Reg_semi}
Let $u_0\in \dH p$ with $ p\in[0,2]$ and $f(\cdot)$ satisfy the Lipschitz condition \eqref{eqn:lipconstant}.
Then semidiscrete problem \eqref{fem-operator} has a unique solution $u_h$ such that for $ t\in(0,T]$
\begin{align*}   
 \|\partial_tu_h(t)\|_{L^2\II} \le c_Tt^{p\alpha/2-1}.
\end{align*}
The constant $c$ above is independent of the mesh size $h$,
but may depend on $T$ and Lipschitz constant  $L$ in \eqref{eqn:lipconstant}.
\end{lemma}

The semidiscrete scheme to the regularized  problems \eqref{eqn:fde-back-reg}  read as: 
find $u_{\gamma,h}(t)\in X_h$ such that 
\begin{equation}\label{eqn:fde:reg:semi}
\begin{aligned}
   \partial^\alpha_t u_{\gamma,h}+A_h u_{\gamma,h}=P_hf(u_{\gamma,h})\quad 
 \text{with} \quad \gamma u_{\gamma,h}(0)+ u_{\gamma,h}(T)=P_h g.
\end{aligned}
\end{equation}
For the problem \eqref{eqn:fde:reg:noisy}, the semidiscrete solution is to find $u_{\gamma,h}^\delta(t)\in X_h$ satisfying
\begin{equation}\label{eqn:fde:reg:noisysemi}
\begin{aligned}
   \partial^\alpha_t u^\delta_{\gamma,h}+A_h u^\delta_{\gamma,h}=P_hf(u^\delta_{\gamma,h})\quad
 \text{with} \quad 
 \gamma u^\delta_{\gamma,h}(0)+ u^\delta_{\gamma,h}(T)=P_h g_\delta.
\end{aligned}
\end{equation}
Employing the solution representation~\eqref{sol_dis_op}, we obtain
 \begin{align}\label{eqn:fde:reg:semisol}
    u_{\gamma,h}(0)&=(\gamma I+F_h(T))^{-1}\left(P_h g-\int_0^T E_h(T-s)P_hf(u_{\gamma,h}(s))\ \ds
    \right),\\
       u_{\gamma,h}^\delta(0)&=(\gamma I+F_h(T))^{-1}\left(P_h g_\delta-\int_0^T E_h(T-s)P_hf(u^\delta_{\gamma,h}(s))\ \ds
    \right).\label{eqn:fde:reg:noisysemisol}
    \end{align}
We shall prove that the existence and uniqueness of $ u_{\gamma,h}(0)$ and  $ u_{\gamma,h}^\delta(0)$ for $T\in (0,T_*]$ with $B(T_*)<1$ defined in \eqref{eqn:contraction}. To this end, for a given $\Tilde{g}\in  X_h$, we define a mapping $M_h:  X_h\rightarrow X_h$ by
\begin{equation}\label{eqn:Mh}
    M_h \phi_h =(\gamma I+F_h(T))^{-1}\left(\Tilde{g}-\int_0^T E_h(T-s)P_hf(S_h(s)\phi_h)\ \ds
    \right) ~~\text{for any}~\phi_h\in X_h,
\end{equation}
where $S_h(t)$ is the solution operator defined in \eqref{eq:semisolutionrep}.
Similar to Lemma \ref{lam:phi12}, it is easy to obtain for all $t \in (0,T]$    \begin{align}\label{eqn:phi12h}
     \| S_h(t)\phi_h^1-S_h(t)\phi_h^2\|_{L^2\II}
     \le c_T\|\phi_h^1-\phi_h^2\|_{L^2\II}, \quad \text{for all}~\phi_h^1, \phi_h^2 \in X_h,
\end{align}
where the constant $c_T$ depends on $T$, but it is independent of $t$ and $h$.
The following lemma provides a discrete analogue to Lemma \ref{lam:phi12} and serves as an important preliminary to the proof of the contraction mapping.


\begin{lemma}\label{lam:phi12h}
Let $S_h(t)$ be the solution operator defined in \eqref{eq:semisolutionrep},  and let $L$ be the Lipschitz constant in \eqref{eqn:lipconstant}. Then, for any $\phi^1_h, \ \phi_h^2 \in X_h$  with $\mu\in(0,1]$ the following inequality holds:
\begin{align*}
     \|S_h(t)\phi^1_h-S_h(t)\phi^2_h\|_{L^2\II}
     &\le (B_0(\al,T,L,\mu) + c_T h^{2-\mu}|\log h|) t^{-\al\mu/2}\|\phi^1_h-\phi^2_h\|_{\dH{-\mu}}\quad \text{for}\quad   t \in (0,T],
\end{align*}
where the constant $B_0(\al,T,L,\mu)$ is  identical to the constant in Lemma~\ref{lam:phi12}.
    \end{lemma}

\begin{proof}
Note that $\|P_hv\|_{L^2(\Omega)}\le \|v\|_{L^2(\Omega)}$ for any $v\in L^2(\Omega)$. Then 
from the relation \eqref{eq:semisolutionrep}, we have 
\begin{align*}
  & \|S_h(t)\phi_h^1-S_h(t)\phi_h^2\|_{L^2\II}\\
\le& \|F_h(t) (\phi_h^1-\phi_h^2)\|_{L^2\II}+\|\int_0^tE_h(t-s)P_h[f(S_h(s)\phi_h^1)-f(S_h(s)\phi_h^2)]\ \ds\|_{L^2\II}\\
\le& \|F(t) (\phi_h^1-\phi_h^2)\|_{L^2\II}+\int_0^t \|E(t-s)P_h[f(S_h(s)\phi_h^1)-f(S_h(s)\phi_h^2)]_{L^2\II}\| \ds\\
&+\| (F(t) - F_h(t)) (\phi_h^1-\phi_h^2)\|_{L^2\II} +\int_0^t  
\|[E(t-s)- E_h(t-s)]P_h[f(S_h(s)\phi_h^1)-f(S_h(s)\phi_h^2)]\|_{L^2\II}\ \ds .
\end{align*}
Then we use Lemma \ref{lem:op} (i) and   Lemma \ref{lem:op-d} (i) to obtain that
\begin{align*}
&\|F(t) (\phi_h^1-\phi_h^2)\|_{L^2\II} +  \|\int_0^tE(t-s)P_h[f(S_h(s)\phi_h^1)-f(S_h(s)\phi_h^2)]\ \ds\|_{L^2\II}\\
\le& c_1 t^{-\al\mu/2}\|\phi_h^1-\phi_h^2\|_{\dH{-\mu}}+ c_1 L\int_0^t(t-s)^{\alpha-1} \|S_h(s)\phi_h^1-S_h(s)\phi_h^2\|_{L^2\II}\ds.
\end{align*}
Moreover, applying the finite element approximation result \cite[Remark 2.1]{LiYangZhou:2024} gives
\begin{align*}
  \|(F(t) - F_h(t)) (\phi_h^1-\phi_h^2)\|_{L^2\II} \le c h^{2-\mu} t^{-\al}\|\phi_h^1-\phi_h^2\|_{\dH{-\mu}}.
\end{align*}
Meanwhile, we use the smoothing properties in Lemmas \ref{lem:op}(i) and \ref{lem:op-d}(i), and the error estimate that \cite[Theorem 2.5]{JinZhou:2023book}, to obtain
\begin{align}\label{eqn:Eh-approx}
  \|(E(t) - E_h(t)) \phi_h\|_{L^2\II} \le c \min(h^{2} t^{-1}, t^{\alpha-1})\|\phi_h\|_{L^2\II}\quad \text{for all}~~\phi_h\in X_h.
\end{align}
This together with the stability of $P_h$, Lipschitz continuity of $f$,
and the estimate \eqref{eqn:phi12h}
leads to
\begin{equation}\label{eqn:Eh-approx-1}
\begin{aligned}
& \int_0^{t-h^{\frac2\alpha}}  
\|[E(t-s)- E_h(t-s)]P_h[f(S_h(s)\phi_h^1)-f(S_h(s)\phi_h^2)]\|_{L^2\II}\ \ds \\
\le & c h^{2}\int_0^{t-h^{\frac2\alpha}}   (t-s)^{-1}
\| S_h(s)\phi_h^1- S_h(s)\phi_h^2\|_{L^2\II}\ \ds\\
\le& c_T h^{2} \|  \phi_h^1-  \phi_h^2\|_{L^2\II} \int_0^{t-h^{\frac2\alpha}} (t-s)^{-1}\ \ds  \le  c_T  h^2 |\log h| \|  \phi_h^1-  \phi_h^2\|_{L^2\II}.
\end{aligned}
\end{equation}
On the other hand, we derive
\begin{equation}\label{eqn:Eh-approx-2}
\begin{aligned}
& \int_{t-h^{\frac2\alpha}}^t  
\|[E(t-s)- E_h(t-s)]P_h[f(S_h(s)\phi_h^1)-f(S_h(s)\phi_h^2)]\|_{L^2\II}\ \ds \\
\le & c \int_{t-h^{\frac2\alpha}}^t   (t-s)^{\alpha-1}
\| S_h(s)\phi_h^1- S_h(s)\phi_h^2\|_{L^2\II}\ \ds \\
\le & c_T \| \phi_h^1- \phi_h^2\|_{L^2\II} \int_{t-h^{\frac2\alpha}}^t   (t-s)^{\alpha-1}\ \ds
\le c_T h^2 \| \phi_h^1- \phi_h^2\|_{L^2\II}.
\end{aligned}
\end{equation}
As a result, we use the inverse inequality \eqref{eqn:inverse}, the norm equivalence \eqref{eqn:equiv}, and arrive at
\begin{align*}
\int_0^t  
\|[E(t-s)- E_h(t-s)]P_h[f(S_h(s)\phi_h^1)-f(S_h(s)\phi_h^2)]\|_{L^2\II}\ \ds 
\le   c_T   h^{2-\mu} |\log h| \|  \phi_h^1-  \phi_h^2\|_{\dH{-\mu}}.
\end{align*}
Combining these estimates with the Gronwall's inequality in Lemma \ref{lemma:Gronwall} leads to the desired result.

\end{proof}

\begin{theorem}\label{lemma:stability_semi} 
For a fixed parameter $\mu \in (0,1]$, let $T_*$ be the constant such that $B_{\mu}(T_*) < 1$, where the function $B_{\mu}(\cdot)$ is defined in \eqref{eqn:contraction} and assume that $T < T_*$. Then, there exists a constant $h_0$ such that, for $\gamma^{-1 + \frac{\mu}{2}} h^{2-\mu} |\log h| \leq h_0$,
the mapping $M_h$ defined in \eqref{eqn:Mh} is a contraction.
\end{theorem}
 
\begin{proof}
We aim to show that $M_h$ is a contraction with the norm $\dH{-\mu}$. For $\phi^1_h,\phi^2_h \in X_h$, we consider the splitting 
    \begin{align*}
       M_h(\phi^1_h-\phi^2_h)
        =[(\gamma I+F_h(T))^{-1}-(\gamma I+F(T))^{-1}]\mathcal{G}_h+(\gamma I+F(T))^{-1}\mathcal{G}_h,
    \end{align*}
where  $\mathcal{G}_h$ is defined by $\mathcal{G}_h=\int_0^T E_h(T-s)P_h[f(S_h(s)\phi_h^2)-f(S_h(s)\phi_h^1)]\ \ds$. 
Using the error estimate for the direct problem \cite[Theorem 2.4]{JinZhou:2023book} gives
\begin{align*}
 &\|[(\gamma I+F_h(T))^{-1}-(\gamma I+F(T))^{-1}]\mathcal{G}_h\|_{\dH{-\mu}}\\
=& \|(\gamma I+F(T))^{-1}(F_h(T) - F(T))(\gamma I+F_h(T))^{-1}\mathcal{G}_h\|_{\dH{-\mu}}\\
=& c_T\gamma^{-1+\frac\mu2}\|(F_h(T) - F(T))(\gamma I+F_h(T))^{-1}\mathcal{G}_h\|_{L^2\II}
\le  c_T\gamma^{-1+\frac\mu2}h^2\|(\gamma I+F_h(T))^{-1}\mathcal{G}_h\|_{L^2\II}\\
\le & c_T\gamma^{-1+\frac\mu2}h^2\| F_h(T)^{-1}\mathcal{G}_h\|_{L^2\II}
\le  c_T\gamma^{-1+\frac\mu2}h^{2-\mu}\|A_h^{-\frac\mu2}F_h(T)^{-1}\mathcal{G}_h\|_{L^2\II}.
\end{align*}  
where in the last inequality, we use the inverse inequality \eqref{eqn:inverse} with $s = \mu/2$. 
Next, applying the smoothing properties in Lemma \ref{lem:op-d} (i) and (iii) yields
\begin{align*}
&\|[(\gamma I+F_h(T))^{-1}-(\gamma I+F(T))^{-1}]\mathcal{G}_h\|_{\dH{-\mu}} \le  c_T\gamma^{-1+\frac\mu2}h^{2-\mu}\|A_h^{-\frac\mu2}F_h(T)^{-1}\mathcal{G}_h\|_{L^2\II} \\
\le & c_T\gamma^{-1+\frac\mu2}h^{2-\mu}  \int_0^T \| A_h^{1-\frac\mu2} E_h(T-s)  P_h[f(S_h(s)\phi_h^2)-f(S_h(s)\phi_h^1)]\|_{L^2\II}\ \ds \\
\le& c_T\gamma^{-1+\frac\mu2}h^{2-\mu}\int_0^T(T-s)^{\al\mu/2-1}\|P_h[f(S_h(s)\phi_h^2)-f(S_h(s)\phi_h^1)]\|_{L^2\II}\ \ds.
\end{align*}  
Then applying the stability of  $P_h$, the Lipchitz continuity of $f$ and Lemma \ref{lam:phi12h}, we derive
\begin{equation*}
\begin{aligned}
&\|[(\gamma I+F_h(T))^{-1}-(\gamma I+F(T))^{-1}]\mathcal{G}_h\|_{\dH{-\mu}}  \\
\le& c_T\gamma^{-1+\frac\mu2}h^{2-\mu}\int_0^T(T-s)^{\al\mu/2-1}\| S_h(s)\phi_h^2-S_h(s)\phi_h^1\|_{L^2\II}\ \ds.\\
\le& c_T\gamma^{-1+\frac\mu2}h^{2-\mu}\int_0^T(T-s)^{\al\mu/2-1}s^{-\alpha\mu/2}\ \ds \|  \phi_h^2- \phi_h^1\|_{\dH{-\mu}}\\
\le& c_T\gamma^{-1+\frac\mu2}h^{2-\mu}\|  \phi_h^2- \phi_h^1\|_{\dH{-\mu}}.
\end{aligned}  
\end{equation*}
Additionally, using Lemma~\ref{lam:phi12h},  and applying the same argument in \eqref{eqn:contraction-0}-\eqref{eqn:contraction} together with the stability of $P_h$, we have
    \begin{equation*}
        \|(\gamma I+F(T))^{-1}\mathcal{G}_h\|_{\dH{-\mu}}\le \|F(T)^{-1}\mathcal{G}_h\|_{\dH{-\mu}}\le (B_\mu(T)+c_T h^{2-\mu}|\log h|)\|\phi_h^1-\phi_h^2\|_{\dH{-\mu}}.
    \end{equation*}
    Hence, we arrive at the estimate 
    \begin{equation*}
        \| M_h(\phi^1_h-\phi^2_h)\|_{\dH{-\mu}}\le (c_T\gamma^{-1+\frac\mu2}h^{2-\mu}|\log h|+B_\mu(T))\|\phi^1_h-\phi^2_h\|_{\dH{-\mu}}.
    \end{equation*}
Since $B_\mu(T)<1$ for any $T \in (0, T_*]$, then we deduce that there exists a constant $h_0$ such that $c_Th_0+B_\mu(T)<1$.
Then for any $h$ satisfying $\gamma^{-1+\frac\mu2}h^{2-\mu}|\log h|<h_0$, the operator $M_h$ is a contraction in $\dH{-\mu}$ and hence admits a unique fixed point.
\end{proof}

We now derive the error between $u_{\gamma,h}(0)$ and $u_\gamma(0)$. 
\begin{lemma} \label{lem:esitugh}  
Let $\mu \in (0,1]$ be a fixed parameter, and let $u_0 \in \dH{-\mu}$. Define $T_*$ as the constant such that $B_\mu(T_*) < 1$, where the function $B_\mu(\cdot)$ is given in \eqref{eqn:contraction}. Assume that $T < T_*$ and  $\gamma^{-1 + \frac{\mu}{2}} h^{2-\mu} |\log h| \leq h_0$ with $h_0$ being given in Theorem \ref{lemma:stability_semi}.
Let $u_\gamma(t)$ and $u_{\gamma, h}(t)$ denote the solutions to the regularized problem \eqref{eqn:fde-back-reg} and the semi-discrete problem \eqref{eqn:fde:reg:semi}, respectively. Then, the following estimate holds:
\[
\| u_{\gamma, h}(0) - u_\gamma(0) \|_{\dH{-\mu}} \leq c \gamma^{-1} h^2 |\log h|,
\]
where $c$ is a constant independent of $\gamma$ and $h$.
\end{lemma}
\begin{proof} 
We  shall use the splitting
\begin{equation*}
    u_{\gamma,h}(0)- u_{\gamma}(0)=(u_{\gamma,h}(0)-P_h u_{\gamma}(0))+(P_hu_{\gamma}(0)- u_{\gamma}(0))\coloneqq \zeta_h(0)+\varrho(0).
\end{equation*}
From Lemma \ref{lem:regugamma} and the approximation property of $P_h$ in \eqref{eqn:Ph-neg}, we have 
\begin{equation*}
\|\varrho(0)\|_{\dH{-\mu}}\le ch^{2}\|u_\gamma(0)\|_{\dH{2-\mu}}\le ch^{2}\gamma^{-1}.
\end{equation*} 
    
Now we turn to the bound of $\zeta_h(0)$. Using the fact $P_h A v=A_h R_hv$ leads to 
\begin{align*}
   \partial^\alpha_t\zeta_h(t)- \Delta_h\zeta_h(t)= P_h(f(u_{\gamma,h})-f(u_{\gamma}))+\Delta_h(R_h-P_h)u_\gamma(t),\quad
 \gamma \zeta_h(0)+ \zeta_h(T)= 0.
\end{align*}
Applying the solution representation~\eqref{sol_dis_op} yields 
\begin{align*}
\zeta_h(t)=&F_h(t)\zeta_h(0)+\int_0^t E_h(t-s)[P_h(f(u_{\gamma,h})-f(u_{\gamma})+\Delta_h(R_h-P_h)u_\gamma(s)]\ \ds.
\end{align*}
Using $ \gamma \zeta_h(0)+ \zeta_h(T)= 0$ gives
\begin{align*}
    \zeta_h(0)=& -(\gamma I+F_h(T))^{-1}\int_0^T E_h(T-s)[P_h(f(u_{\gamma,h})-f(u_{\gamma})+\Delta_h(R_h-P_h)u_\gamma(s)]\ \ds\\
   = &-(\gamma I+F_h(T))^{-1}\bigg(\int_0^T E_h(T-s)[P_h(f(u_{\gamma,h})-f(u^h_{\gamma}))]\ds\\&+\int_0^TE_h(T-s)[P_h(f(u^h_{\gamma})-f(u_{\gamma})+\Delta_h(R_h-P_h)u_\gamma(s)]\ \ds\bigg),
\end{align*}
where $u_{\gamma}^h(t)$ solves the semidiscrete problem \eqref{fem-operator} with $u_{\gamma}^h(0)=P_hu_\gamma(0)$.
From \cite[Theorem 4.4]{Karaa:2019},  Lipschitz condition~\eqref{eqn:lipconstant}, Lemma \ref{lem:op-d} (iii) and Lemma \ref{lem:regugamma}, we arrive at 
\begin{align*}
& \|(\gamma I+F_h(T))^{-1}\int_0^T E_h(T-s)P_h(f(u_{\gamma}^h)-f(u_{\gamma}))\ \ds\|_{\dH{-\mu}}\\
\le& c \| A_h^{-\frac\mu2}(\gamma I+F_h(T))^{-1}\int_0^T E_h(T-s)P_h(f(u_{\gamma}^h)-f(u_{\gamma}))\ \ds\|_{L^2\II}\\
\le& c \gamma^{-1+\frac\mu2}\int_0^T (T-s)^{\al-1}\|u_{\gamma}^h(s)-u_{\gamma}(s)\|_{L^2\II}\ds\\
\le& c \gamma^{-1+\frac\mu2}h^2 |\log h| \| u_{\gamma}(0) \|_{L^2\II} \int_0^T (T-s)^{\al-1}s^{-\al} \ \ds
\le  c\gamma^{-1} h^2 |\log h|.
\end{align*}

Then, using Lemma \ref{lem:op-d} (iii), Lemma \ref{lem:regugamma}, Lemma \ref{lem:Reg}, and choosing $\epsilon = 1/|\log h|$, we deduce that
\begin{align*}
& \|(\gamma I+F_h(T))^{-1}\int_0^T E_h(T-s)\Delta_h(R_h-P_h)u_\gamma(s)\ \ds\|_{\dH{-\mu}} \\
\le&  c\gamma^{-1+\frac\mu2}h^{2-2\epsilon}\int_0^T (T-s)^{\alpha\epsilon-1}\|u_\gamma(s)\|_{\dot H^2(\Omega)}\ds \\
\le&c\gamma^{-1+\frac{\mu}{2}}h^{2-2\epsilon} \|u_{\gamma}(0)  \|_{L^2\II}
\int_0^T (T-s)^{\alpha\epsilon-1}s^{-\al}\ds
   \le c\gamma^{-1}h^{2-2\epsilon}\frac{1}{\epsilon} \le c\gamma^{-1}h^2|\log h|.
\end{align*}


The desired results follow from Theorem~\ref{lemma:stability_semi}.
\end{proof}

 Following the argument in Theorem~\ref{thm:err-reg}, we obtain the following error estimate.
 \begin{theorem} \label{thm:err-reg-semi}
Let $\mu\in(0,1]$ be a fixed parameter and $u_0 \in \dH{-\mu+q}$ with $q\in (0,2]$. Define $T_*$ as the constant such that $B_\mu(T_*) < 1$, where the function $B_\mu(\cdot)$ is given in \eqref{eqn:contraction}. Assume that $T < T_*$ and  $\gamma^{-1 + \frac{\mu}{2}} h^{2-\mu} |\log h| \leq h_0$ with $h_0$ being given in Theorem \ref{lemma:stability_semi}. Let $u$ and $u_{\gamma,h}^\delta$ be the solutions to the backward problem \eqref{eq:back_non} 
 and regularized problem \eqref{eqn:fde:reg:noisysemi}, respectively. Then
\begin{align*}
 \|u_{\gamma,h}^\delta(0)-u_0\|_{\dH{-\mu}}\le c(\gamma^{-1} \delta+\gamma^{-1}h^2|\log h|+ \gamma^{\frac{q}{2}}).
\end{align*}
Moreover, for $u_0\in \dH{-\mu}$, there holds 
     \begin{align*}
           \|u_{\gamma,h}^\delta(0)-u_0\|_{\dH{-\mu}}\rightarrow 0\quad \text{as} \quad\delta,\ \gamma,\ h\rightarrow 0^+,\   \ \frac{\delta}{\gamma}\rightarrow 0^+\ \text{and } \ 
           \frac{h^2|\log h|} {\gamma}  \rightarrow 0^+.
     \end{align*}
 \end{theorem}

\subsection{Fully discretization and error analysis}
In this section, we propose an inversion algorithm with space-time discretization and establish an error bound for the numerical reconstruction. 
Firstly, we describe the fully discrete scheme for the direct problem.
We partition the time interval $[0,T]$ into a uniform grid, with $ t_n=n\tau$, $n=0,\ldots,N$, and $\tau=T/N$ representing
the time step size. We then approximate the fractional derivative using 
the backward Euler convolution quadrature (with $\varphi^j=\varphi(t_j)$) as referenced in~\cite{lubich1988convolution, JinZhou:2023book}:
\begin{align*}
\bar\partial_\tau^\alpha \varphi^n = \sum_{j=0}^n \omega_{n-j}^{(\alpha)} (\varphi^{j} - \varphi^0)
\quad\mbox{ with }  ~ \omega_j^{(\alpha)}=  (-1)^j\frac{\Gamma(\alpha+1)}{\Gamma(\alpha-j+1)\Gamma(j+1)}.
\end{align*}
Consider the  linearized  fully discrete scheme for problem \eqref{eqn:fde}: find ${U_h^n}\in X_h$ such that for $1\le n\le N$
\begin{align}\label{eqn:fully}
\bDal U_h^n +A_hU_h^n= P_h f(U_h^{n-1})\quad \text{with}~~U_h^0 = P_hu_0.
\end{align}  
By means of Laplace transform 
with $1\le n \le N$, 
the solution representation of fully discrete solution $U_h^n$ can be written as \cite{zhang2022identification, zhang2023stability}
\begin{align}\label{eqn:back-fully-non}
    U_h^n=F_{h,\tau}^nU_h^0+\tau \sum_{k=1}^nE_{h,\tau}^{n-k}P_hf(U_h^{k-1}) :=S_{h,\tau}^nU_h^{0},
\end{align}
where
    \begin{align}\label{eq:op-fully}
        F_{h,\tau}^n=\frac{1}{2\pi i}\int_{\Gamma_{\theta,\sigma}^\tau} e^{zt_{n-1}}\delta_{\tau}(e^{-z\tau})^{\alpha-1}G_h(z)\ \dz, \ \   E_{h,\tau}^n=\frac{1}{2\pi i}\int_{\Gamma_{\theta,\sigma}^\tau} e^{zt_n}
        G_h(z)\ \dz 
    \end{align}
with $G_h(z)=(\delta_\tau(e^{-z\tau})^{\alpha}+A_h)^{-1},\ \delta_\tau(\xi)=(1-\xi)/\tau$ and the contour
$\Gamma_{\theta,\sigma}^\tau :=\{ z\in \Gamma_{\theta,\sigma}:|\Im(z)|\le {\pi}/{\tau} \}$, Oriented with an increasing imaginary part, where $\theta\in(\pi/2,\pi)$ is close to $\pi/2$. Here, we employ $S_{h,\tau}^n$ to denote the fully discrete scheme solution operator. Then we can rewrite \eqref{eqn:back-fully-non} as 
\begin{equation}\label{eqn:back-fully-non-new}
    U_h^n=S_{h,\tau}^nU_h^{0}=F_{h,\tau}^nU_h^0+\tau \sum_{k=1}^nE_{h,\tau}^{n-k}P_hf(S_{h,\tau}^{k-1}U_h^0).
\end{equation}

 Observe that the solution operators $F_{h,\tau}^n$ and $E_{h,\tau}^n$ satisfy the following smoothing properties. The proof of these properties is identical to the one provided in Lemma \ref{lem:op}.

\begin{lemma}\label{lem:op:fully}
Let $F_{h,\tau}^n$ and $E_{h,\tau}^n$ be the operators in \eqref{eq:op-fully}.
Then they satisfy the following properties for any $n\ge 1$ and $v_h\in X_h$,
\begin{itemize}
\item[$\rm(i)$] $\|
A_h^\nu F_{h,\tau}^n v_h\|_{L^2\II}  + t_{n+1}^{1-\alpha}  \| A_h^\nu  E_{h,\tau}^n v_h  \|_{L^2\II}  \le c  t_{n+1}^{-\nu\alpha} \|v_h\|_{L^2\II}$ with $0\le \nu\le 1$;
\item[$\rm(ii)$] $\|( F_{h,\tau}^n )^{-1} v_h \|_{L^2\II}\le c (1+ t_n^\alpha) \| A_h v _h\|_{L^2\II}$.
\end{itemize}
The constant $c$ is independent of  $n$.
 \end{lemma}
We now present the fully discrete scheme for solving the backward problem~\eqref{eqn:fde:reg:noisy}: find 
$U_{h,\gamma}^{n,\delta} \in X_h$  such that: for $1\le n\le N$ 
\begin{equation}\label{eqn:back-fully}
\Bar{\partial}^\alpha_\tau U_{h,\gamma}^{n,\delta} +A_hU_{h,\gamma}^{n,\delta} = P_hf(U_{h,\gamma}^{n-1,\delta} ) 
\quad \text{with}~~ 
\gamma U_{h,\gamma}^{0,\delta} + U_{h,\gamma}^{N,\delta} = P_h g_\delta.
\end{equation}
Using the solution representation~\eqref{eqn:back-fully-non-new} gives
\begin{equation}\label{eqn:back-fully-gdelta}
\begin{aligned}
U_{h,\gamma}^{0,\delta} = (\gamma I+F_{h,\tau}^N)^{-1}\Big[P_hg_\delta-\tau \sum_{k=1}^N E_{h,\tau}^{N-k}P_hf(U_{h,\gamma}^{k-1,\delta})\Big].
\end{aligned}
\end{equation} 
The next lemma provides some approximation properties of solution operators 
$F_{h,\tau}^n$ and $E_{h,\tau}^n$.
See \cite[Lemma 4.2]{zhang2022identification} for the proof of the first estimate,
and
\cite[Lemma 9.5]{JinZhou:2023book}
for the second estimate.
\begin{lemma}
\label{lem:op-err:fully}
For the operator $F_{h,\tau}^n$ and $E_{h,\tau}^n$ defined in \eqref{eq:op-fully},  for $\nu \in [0,1]$, we have 
\begin{align*}
&\| A_h^\nu (F_{h,\tau}^n-  F_{h}(t_n)\|\le c\tau t_n^{-1-\nu\al},\\
&\Big\|\tau A_h^\nu E_{h,\tau}^{n-k}-\int_{t_{k-1}}^{t_k} A_h^\nu E_h(t_n-s)\ds\Big\|\le c\tau^2(t_{n}-t_k+\tau)^{-(2-(1-\nu)\al)}.
\end{align*}
\end{lemma}
The following lemma provides a useful estimate of the discrete operator $(\gamma I+F_{h,\tau}^N)^{-1}$; see a detailed proof in \cite[Lemma 4.4]{zhang2022backward}.
\begin{lemma}\label{lem:op-reg-fully}
Let $F_{h,\tau}^n$ and $E_{h,\tau}^n$ be the operators defined in \eqref{eq:op-fully}.
Then there holds 
\begin{equation*}
\|(\gamma I+F_{h,\tau}^N)^{-1}v_h\|_{L^2\II}\le c\gamma^{-1}\|v_h\|_{L^2\II}~~ \text{and} \ ~~
\|F_{h,\tau}^N(\gamma I+F_{h,\tau}^N)^{-1}v_h \|_{L^2\II}\le \|v_h\|_{L^2\II},
\end{equation*}
where $c$ is uniform in $T$, $h$, $\tau$ and $\gamma$.
\end{lemma}


We proceed to examine the existence and uniqueness of $U_{h,\gamma}^{0,\delta}$ in \eqref{eqn:back-fully-gdelta} provided that $T\in (0,T_*]$ with $B_\mu(T_*)<1$, where the function $B_\mu(\cdot)$ is defined in \eqref{eqn:contraction}.  
To this end, for a given $\hat{g}\in X_h$, we define a mapping $M_{h,\tau}:  X_h\rightarrow X_h$ by
\begin{equation}\label{eqn:Mhtaut}
   M_{h,\tau} \phi_h =(\gamma I+F_{h,\tau}^N)^{-1}\left( \hat{g}-\tau \sum_{k=1}^NE_{h,\tau}^{N-k}P_hf(S_{h,\tau}^{k-1}\phi_h)\right) ~~\text{for any}~\phi_h\in X_h,
\end{equation}
where 
 $S_{h,\tau}^k$ is the fully discrete scheme solution operator defined in \eqref{eqn:back-fully-non}.
\begin{lemma}\label{lama:phi12-ful}
Let $S_{h,\tau}^n$ be the solution operator defined in \eqref{eqn:back-fully-non},  and let $L$ be the Lipschitz constant in \eqref{eqn:lipconstant}. Then, for any $\phi^1_h, \ \phi_h^2 \in X_h$  with $\mu\in(0,1]$ the following inequality holds:
\begin{align*}
 \|S_{h,\tau}^n\phi_h^1-S_{h,\tau}^n\phi_h^2\|_{L^2\II}
\le\Big(B_0(\al,T,L,\mu) t_n^{-\alpha\mu/2}+a(t_n) \Big)\|\phi_h^1-\phi_h^2\|_{\dH{-\mu}},
\end{align*}
where the constant $B_0(\al,T,L,\mu)$  is  given in  Lemma \ref{lam:phi12} and 
$$a(t_n) = c_T (\tau^{\alpha} h^{-\mu}(t_n^{-\al}+1) +h^{2-\mu}|\log h|t_n^{-\al} )$$
with a generic constant $c_T$ only depending on $T$.
\end{lemma}
\begin{proof}
Define $\eta(t) = S_{h,\tau}^n \phi_h^1 - S_{h,\tau}^n \phi_h^2$, for $t \in (t_{n-1}, t_n]$. 
First, by applying Gronwall's inequality, it follows directly that  
\begin{equation}\label{eqn:bl2}
\| \eta(t) \|_{L^2(\Omega)} \leq c \| \phi_h^1 - \phi_h^2 \|_{L^2(\Omega)}.
\end{equation}

Next, we address the more challenging case: bounding $\| \eta(t) \|_{L^2(\Omega)}$ in terms of $\| \phi_h^1 - \phi_h^2 \|_{\dot{H}^{-\mu}(\Omega)}$.
For $t\in (t_{n-1},t_n]$, $n\ge2$,
applying the representation \eqref{eqn:back-fully-non} gives
\begin{equation}\label{eqn:split-eta}
\begin{aligned}
\eta(t)=&[F_{h,\tau}-F_h(t_n)](\phi_h^1-\phi_h^2)+F_h(t_n)(\phi_h^1-\phi_h^2)+\tau E_{h,\tau}^{n-1}P_h[f(\phi_h^1)- f(\phi_h^2)]\\&+  \sum_{k=2}^n[\tau E_{h,\tau}^{n-k}-\int_{t_{k-1}}^{t_k}E_h(t_n-s)\ds] P_h[f(S_{h,\tau}^{k-1}\phi_h^2)-f(S_{h,\tau}^{k-1}\phi_h^1)]\\
& +  \sum_{k=2}^n\int_{t_{k-1}}^{t_k}(E_h(t_n-s)-E(t_n-s))\ds P_h[f(S_{h,\tau}^{k-1}\phi_h^2)-f(S_{h,\tau}^{k-1}\phi_h^1)]\\ & +  \sum_{k=2}^n\int_{t_{k-1}}^{t_k}E(t_n-s)\ds P_h[f(S_{h,\tau}^{k-1}\phi_h^2)-f(S_{h,\tau}^{k-1}\phi_h^1)].
\end{aligned} 
\end{equation}
By applying Lemma \ref{lem:op-err:fully} and the argument in the proof of Lemma \ref{lam:phi12h}, we derive  for $t\in (t_{n-1},t_{n}]$
\begin{align*}
& \|[F_{h,\tau}-F_h(t_n)](\phi_h^1-\phi_h^2)+F_h(t_n)(\phi_h^1-\phi_h^2)\|_{L^2\II} \\
\le & \|[F_{h,\tau}-F_h(t_n)](\phi_h^1-\phi_h^2)+(F_h(t_n)-F(t_n))(\phi_h^1-\phi_h^2)  
+F(t_n)(\phi_h^1-\phi_h^2) \|_{L^2\II} \\
\le& \Big(c\tau^\al h^{-\mu}t_n^{-\al} +ch^{2-\mu}t_n^{-\al}+ c_1 t_n^{-\al\mu/2} \Big)  \| \phi_h^1 - \phi_h^2 \|_{\dH{-\mu}}.
\end{align*}
Moreover, using Lemma \ref{lem:op:fully} (i) and the inverse inequality \eqref{eqn:inverse}, we obtain
\begin{align*}
 \| \tau E_{h,\tau}^{n-1}P_h[f(\phi_h^1)- f(\phi_h^2)]\|_{L^2\II}
\le c\tau h^{-\mu}t_n^{\al-1} \| \phi_h^1 - \phi_h^2 \|_{\dH{-\mu}} 
 \le   c\tau^\alpha h^{-\mu}  \| \phi_h^1 - \phi_h^2 \|_{\dH{-\mu}}.
\end{align*}
Similarly, using Lemma \ref{lem:op-err:fully} and the estimate \eqref{eqn:bl2} also leads to
\begin{align*}
& \sum_{k=2}^n \|[\tau E_{h,\tau}^{n-k}-\int_{t_{k-1}}^{t_k}E_h(t_n-s)\ds] P_h[f(S_{h,\tau}^{k-1}\phi_h^2)-f(S_{h,\tau}^{k-1}\phi_h^1)] \|_{L^2\II}\\
\le & c\tau^2\sum_{k=2}^n(t_n-t_{k-1})^{\al-2}\|\eta(t_{k-1})\|_{L^2\II} \le  c\tau^2\sum_{k=2}^n(t_n-t_{k-1})^{\al-2}\| \phi_h^1 - \phi_h^2  \|_{L^2\II}\\ 
\le&  c \tau^\alpha \| \phi_h^1 - \phi_h^2 
 \|_{L^2\II} \le c \tau^\alpha h^{-\mu}\| \phi_h^1 - \phi_h^2 
 \|_{\dH{-\mu}}.
\end{align*} 
Next, we apply the estimate \eqref{eqn:Eh-approx} and similar argument in \eqref{eqn:Eh-approx-1} and \eqref{eqn:Eh-approx-2} to obtain
\begin{align*}
& \sum_{k=2}^n\int_{t_{k-1}}^{t_k} \| (E_h(t_n-s)-E(t_n-s))\|  \, \ds \| P_h[f(S_{h,\tau}^{k-1}\phi_h^2)-f(S_{h,\tau}^{k-1}\phi_h^1)]\|_{L^2\II} \\
\le &c\sum_{k=2}^n\int_{t_{k-1}}^{t_k} \min\Big(h^2 (t_n-s)^{-1}, (t_n-s)^{\alpha-1}\Big) \, \ds \, \| \phi_h^1 - \phi_h^2  \|_{L^2\II}\\
\le & c h^2 |\log h| \| \phi_h^1 - \phi_h^2  \|_{L^2\II} \le c h^{2-\mu} |\log h| \| \phi_h^1 - \phi_h^2  \|_{\dH{-\mu}}.
\end{align*} 
For the last term in \eqref{eqn:split-eta}, we apply Lemma \ref{lem:op} (i) to derive
\begin{align*}
&\sum_{k=2}^n\int_{t_{k-1}}^{t_k} \| E(t_n-s)\|\, \ds \,\|P_h[f(S_{h,\tau}^{k-1}\phi_h^2)-f(S_{h,\tau}^{k-1}\phi_h^1)] \|_{L^2\II}\\
\le &c_1L\sum_{k=2}^n\int_{t_{k-1}}^{t_k}(t_n-s)^{\al-1}\ds \|\eta(t_{k-1})\|_{L^2\II}  \\
\le &c_1L\sum_{k=2}^n\int_{t_{k-1}}^{t_k}(t+\tau-s)^{\al-1}\ds \|\eta(t_{k-1})\|_{L^2\II}  + c\tau^2\sum_{k=2}^n(t_n-t_{k-1})^{\al-2}\|\eta(t_{k-1})\|_{L^2\II} \\
\le & c_1L \int_{0}^{t_{n-1}}(t-s)^{\al-1}\|\eta(s )\|_{L^2\II}\ds+ c\tau^2\sum_{k=2}^n(t_n-t_{k-1})^{\al-2}\| \phi_h^1 - \phi_h^2  \|_{L^2\II}
\\ \le  & c_1L \int_{0}^{t}(t-s)^{\al-1}\|\eta(s )\|_{L^2\II}\ds+  c \tau^\alpha h^{-\mu}
\| \phi_h^1 - \phi_h^2  \|_{\dH{-\mu}}.
\end{align*}
In summary, we arrive at 
\begin{align*}
 \| \eta(t) \|_{L^2\II} \le &\left(c_1 t^{-\al\mu/2} + c_T \Big(\tau^\al h^{-\mu}(t^{-\al}+1)+h^{2-\mu}|\log h|t^{-\al}  \Big)\right)\| \phi_h^1 - \phi_h^2  \|_{\dH{-\mu}} \\
 &\qquad + c_1L \int_{0}^{t}(t-s)^{\al-1}\|\eta(s )\|_{L^2\II}\ds,\quad t\in (t_{n-1}, t_n],\quad n\ge 2.
\end{align*}
For $t\in (0,\tau]$,   $\| \eta(t) \|_{L^2\II}= \|S_{h,\tau}^1 \phi_h^1 - S_{h,\tau}^1 \phi_h^2\|_{L^2\II}$, it is straightforward to derive
\begin{align*}
     \| \eta(t) \|_{L^2\II}\le&\left( c_1 t_1^{-\al\mu/2} + c_T \Big(\tau^\al h^{-\mu}(t_1^{-\al}+1)+h^{2-\mu}|\log h|t_1^{-\al}  \Big)\right)\| \phi_h^1 - \phi_h^2  \|_{\dH{-\mu}}\\ \le &\left(c_1 t^{-\al\mu/2} + c_T \Big(\tau^\al h^{-\mu}(t^{-\al}+1)+h^{2-\mu}|\log h|t^{-\al}  \Big)\right)\| \phi_h^1 - \phi_h^2  \|_{\dH{-\mu}} \\
 &\qquad + c_1L \int_{0}^{t}(t-s)^{\al-1}\|\eta(s )\|_{L^2\II}\ds.
\end{align*}
Then the desired result follows from the Gronwall's inequality in Lemma \ref{lemma:Gronwall}.
\end{proof}

\begin{theorem}\label{lemma:stability_fully} 
For a fixed parameter $\mu \in (0,1]$, let $T_*$ be the constant such that $B_{\mu}(T_*) < 1$, where the function $B_{\mu}(\cdot)$ is defined in \eqref{eqn:contraction} and assume that $T < T_*$.
There exists a constant $c_*$ such that, if $\gamma$, $h$, and $\tau$ satisfy the condition  
$\gamma^{-1 + \frac{\mu}{2}} h^{2-\mu} |\log h| + \tau^{\alpha \mu / 2} + \tau^\alpha h^{-\mu} \leq c_*$,
then the mapping $M_{h, \tau}$ defined in \eqref{eqn:Mhtaut} is a contraction.
\end{theorem}
\begin{proof}
We consider the splitting
  \begin{align*}
      M_{h,\tau}(\phi^1_h-\phi^2_h)
        =&[(\gamma I+F_{h,\tau}^N)^{-1}-(\gamma I+F_h(T))^{-1}]\mathcal{G}_{h,\tau} +(\gamma I+F_h(T))^{-1}[\mathcal{G}_{h,\tau} -\Tilde{\mathcal{G}}_{h,\tau}]\\&+(\gamma I+F_h(T))^{-1}\Tilde{\mathcal{G}}_{h,\tau},
    \end{align*}
where $\mathcal{G}_{h,\tau}$ and $\tilde{ \mathcal{G}}_{h,\tau}$ are respectively defined by
\begin{align*}
\mathcal{G}_{h,\tau}=&\tau \sum_{k=1}^NE_{h,\tau}^{N-k}P_h[f(S_{h,\tau}^{k-1}\phi_h^2)-f(S_{h,\tau}^{k-1}\phi_h^1)],\\
\tilde{ \mathcal{G}}_{h,\tau}=& \sum_{k=1}^N\int_{t_{k-1}}^{t_k}E_h(T-s)\ \ds \ P_h[f(S_{h,\tau}^{k-1}\phi_h^2)-f(S_{h,\tau}^{k-1}\phi_h^1)].
\end{align*}
From \cite[Lemma 15.8]{JinZhou:2023book} and Lemma \ref{lem:op-reg-fully}, we obtain
\begin{align*}
\|[(\gamma I+F_{h,\tau}^N)^{-1}-(\gamma I+F_h(T))^{-1}]\mathcal{G}_{h,\tau} \|_{\dH{-\mu}}&\le c_T\tau \|A_h^{1-\frac \mu 2}\mathcal{G}_{h,\tau}\|_{L^2\II} ,\\
\|(\gamma I+F_h(T))^{-1}[\mathcal{G}_{h,\tau} -\Tilde{\mathcal{G}}_{h,\tau}]\|_{\dH{-\mu}}&\le c_T\|A_h^{1-\frac \mu 2}(\mathcal{G}_{h,\tau} -\Tilde{\mathcal{G}}_{h,\tau})\|_{L^2\II}.
\end{align*} 
Using Lemma \ref{lem:op:fully}, the Lipschitz condition~\eqref{eqn:lipconstant}, the estimate in \eqref{eqn:bl2} and the inverse inequality~\eqref{eqn:inverse} yields
\begin{align*} 
\|A_h^{1-\frac \mu 2}\mathcal{G}_{h,\tau}\|_{L^2\II}
\le &\|\tau \sum_{k=1}^NA_h^{1-\frac{\mu}{2}}E_h^{N-k}P_h[f(S_{h,\tau}^{k-1}\phi_h^2)-f(S_{h,\tau}^{k-1}\phi_h^1)]\|_{L^2\II}\\
\le &c\tau \sum_{k=1}^N(T-t_{k-1})^{\al\mu/2-1}\|\phi_h^2-\phi_h^1\|_{L^2\II}
       \le c_T h^{-\mu}\|\phi_h^2-\phi_h^1\|_{\dH{-\mu}}.
\end{align*} 
Additionally, applying Lemmas \ref{lem:op-err:fully} and \ref{lama:phi12-ful}, along with the inverse inequality \eqref{eqn:inverse}, we derive the following estimate
\begin{align*}
        &\|A_h^{1-\frac \mu 2}(\mathcal{G}_{h,\tau} -\Tilde{\mathcal{G}}_{h,\tau})\|_{L^2\II}\\\le& \sum _{k=2}^N\|\tau A_h^{1-\frac{\mu}{2}} E_{h,\tau}^{N-k}-\int_{t_{k-1}}^{t_k}A_h^{1-\frac{\mu}{2}} E_h(T-s)\ds\|\|P_hf(S_{h,\tau}^{k-1}\phi_h^2)-f(S_{h,\tau}^{k-1}\phi_h^1)]\|_{L^2\II}\\&+
        \|\tau A_h^{1-\frac \mu 2}E_{h,\tau}^{N-1}-\int_{0}^{t_1}A_h^{1-\frac \mu 2} E_h(T-s)\ds\|\|P_hf(\phi_h^2)-f(\phi_h^1)]\|_{L^2\II}\\
        \le &c\tau^2 \sum_{k=2}^N(T-t_{k-1})^{\al\mu/2-2} \Big(B_0(\al,T,L,\mu) t_{k-1}^{-\alpha\mu/2}+a(t_{k-1}) \Big)\|\phi_h^2-\phi_h^1\|_{\dH{-\mu}}+c\tau^2T^{\al\mu/2-2}\|\phi_h^2-\phi_h^1\|_{L^2\II}\\
        \le&  c_T\tau^{\alpha\mu/2}(1+\tau^\al h^{-\mu}+h^{2-\mu}|\log h|)\|\phi_h^2-\phi_h^1\|_{\dH{-\mu}}.
\end{align*}
In the last inequality, we use the fact that
  $\tau^2 \sum_{k=2}^N(T-t_{k-1})^{\al\mu/2}t_{k-1}^{-\beta}\le c\tau^{\al\mu/2}$ for $0\le \beta <1$,as shown in \cite[Lemma 3.11]{JinZhou:2023book}.
  
Based on  Lemma \ref{lama:phi12-ful}, applying the  arguments in the proof of Theorem \ref{lemma:stability_semi} and Lemma \ref{lama:phi12-ful} gives
\begin{equation*}
\|(\gamma I+F_h(T))^{-1}\Tilde{\mathcal{G}}_{h,\tau}\|_{\dH{-\mu}}\le  (c_T\gamma^{-1+\frac\mu2}h^{2-\mu}|\log h| +\tau^{\al}h^{-\mu}+B_\mu(T))\|\phi^1_h-\phi^2_h\|_{\dH{-\mu}}.
\end{equation*}
Therefore, we arrive at the estimate 
\begin{equation*}
        \|  M_{h,\tau}(\phi^1_h-\phi^2_h)\|_{\dH{-\mu}}\le (c_T(\gamma^{-1+\frac\mu2}h^{2-\mu}|\log h|+\tau^{\al\mu/2}+\tau^{\al}h^{-\mu})+B_\mu(T))\|\phi^1_h-\phi^2_h\|_{\dH{-\mu}}.
\end{equation*}
Since $B_\mu(T)<1$ for any $T \in (0, T_*]$, we conclude that there exists a  constant $c_*>0$, such that $c_T c_* + B_\mu(T) <1$. Then for algorithmic parameters $\gamma,h,\tau$ satisfying
$$\gamma^{-1}h^2|\log h|+\tau^{\al\mu/2}+\tau^{\al}h^{-\mu}<c_*,$$
 the operator $M_{h,\tau}$ is a contraction in $\dH{-\mu}$, and hence admits a unique fixed point.
\end{proof}

\begin{remark}\label{rem:iter}
The contraction property of $M_{h,\tau}$, established in Theorem \ref{lemma:stability_fully}, naturally motivates the development of an iterative algorithm for solving $U_{h,\gamma}^{0,\delta}$ in the scheme \eqref{eqn:back-fully-gdelta}. In each iteration, one needs to solve a linear backward problem, which can be efficiently addressed using the conjugate gradient method \cite{zhang2020numerical,zhang2023stability}. The details of the algorithm are summarized in Algorithm \ref{alg}.  The contraction property proved in Theorem \ref{lemma:stability_fully} ensures linear convergence of the iterative process in the $\dot{H}^{-\mu}$ norm for a fixed $\mu > 0$.  

In practice, for ease of implementation, we replace the $\dot{H}^{-\mu}$ norm with the $L^2$ norm. Numerical experiments demonstrate stable convergence and accurate reconstruction in this setting. However, from a theoretical perspective, proving convergence in the $L^2$ norm requires the restrictive condition \eqref{eqn:lipconstant-2}. Removing this restriction remains an open problem and warrants further theoretical investigation.
\end{remark}

\begin{algorithm}
\caption{An iterative algorithm for solving scheme \eqref{eqn:back-fully-gdelta} to find $U_{h,\gamma}^{0,\delta}$.}\label{alg}
\begin{algorithmic}[h!]
 \STATE{\textbf{Input:} Order $\alpha$ terminal time $T$, noisy observation $g_\delta$, discretization parameters $h$ and $\tau$.}
 \STATE{\textbf{Output:} Approximate initial data $U_{h,\gamma}^{0,\delta}$.}
\STATE{Initialize $U_{0,0}$ randomly, set $e^0=1$, $j=0$. Using scheme \eqref{eqn:back-fully-non-new}, compute 
$$\tau\sum_{k=1}^N E_{h,\tau}^{N-k}P_hf(S_{h,\tau}^{k-1} U_{0,j})=S_{h,\tau}^N U_{0,j}-F_{h,\tau}^N U_{0,j}.$$ }
\STATE{\textbf{while} $ e^j\ge tol=10^{-10} $ \textbf{do}}
\STATE{Update $U_{0,j+1}$ using the conjugate gradient method:
$$  (\gamma I +F_{h,\tau}^N )U_{0,j+1}= P_hg_\delta-\tau\sum_{k=1}^N E_{h,\tau}^{N-k}P_hf(S_{h,\tau}^{k-1} U_{0,j}). $$}
\STATE{Compute error  $e^j = \|U_{0,j+1} - U_{0,j}\|_{L^2\II}$ and set $j=j+1$.} 
\STATE{\textbf{end while}}
\STATE{\textbf{Output:} $U_{h,\gamma}^{0,\delta}\approx U_{0,j}$.}
\end{algorithmic}
\end{algorithm}

To show the error between the numerical reconstruction $U_{h,\gamma}^{0,\delta}$ and the exact initial data $u_0$, we introduce an
auxiliary function  $\bar U_{h,\gamma}^n  \in X_h $ such that 
\begin{equation}\label{eqn:fde:fulbar}
\Bar{\partial}^\alpha_\tau \bar U_{h,\gamma}^n +A_h \bar U_{h,\gamma}^n = P_hf(\bar U_{h,\gamma}^{n-1})\quad  \text{for}~~ ~1\le n\le N,
\quad \text{with}~~~ \bar U_{h,\gamma}^0 = u_{\gamma,h}(0).
\end{equation}

In the following, we derive novel error estimates for the direct problem. To achieve this, we first establish preliminary estimates for the linear problem. Consider the semidiscrete scheme for the linear problem: given $v_h^0 = P_h v_0$, find $v_h(t)  \in X_h$ such that
\begin{equation}\label{eqn:linear}
(\partial_t^\alpha v_h(t), \varphi_h) + (\nabla v_h(t), \nabla \varphi_h) = (f(t),\varphi_h), \quad \forall \varphi_h\in X_h, \forall t\in(0,T],
\end{equation}
and its fully discrete scheme: given $v_h^0 = P_h v_0$, find $v_h^n\in X_h$ such that: for $1\le n\le N$ 
\begin{align}\label{eqn:linear-fully}
   (\bar \partial_\tau^\alpha v_h^n , \varphi_h) + (\nabla v_h^n , \nabla \varphi_h)  = (f(t_n),\varphi_h) , \quad \forall \varphi_h\in X_h.
\end{align}

Next, we provide a nonstandard error estimate in stronger norms for the direct problem. The detailed proof is lengthy and is therefore presented in the Appendix.
\begin{lemma}\label{lem:error-linear}
Let $v_h$ and $v_h^n$ solve problems \eqref{eqn:linear} and \eqref{eqn:linear-fully}, respectively, with $v_0\in L^2(\Omega)$.
Then  the following error estimate holds for any $0\le p\le 1$
\begin{align*}
   \|A_h^{p}(v_h(t_n)- v_h^n)\|_{L^2(\Omega)} \leq c\bigg(&\tau t_n^{-1-p\al} \|v_0\|_{L^2\II}+\tau t_n^{(1-p)\alpha-1} \|f(s)\|_{L^{\infty}(0,\tau; L^2(\Omega))}\\&+\tau\int_{\tau}^{t_n}(t_{n+1}-s)^{(1-p)\alpha-1}\|f'(s)\|_{L^2\II}\ds\bigg),  
  \end{align*}
where the constant $c$ is independent on $t_n$, $T$, $h$ and $\tau$.
\end{lemma}

Building on this error estimate, we derive the following error estimate for the nonlinear problem. The proof is provided in the Appendix.
 
\begin{lemma}\label{lem:err-reg:H2} 
Let $\tuh(t)$ and $\bar U_{h,\gamma}^{n}$ be the solutions to \eqref{eqn:fde:reg:semi} and \eqref{eqn:fde:fulbar} respectively. 
Then there holds for $0\le p\le1$ 
\begin{equation*}
    \| A_h^p(\tuh(t_n)-\bar  U_{h,\gamma}^{n}) \|_{L^2\II} \le c_T\tau|\log \tau |^2  t_n^{-1-p\al}\|\tuh(0)\|_{L^2\II},
\end{equation*}
where the constant $c_T$ depends on $T$, but it is independent on $\gamma$, $t_n$, $h$ and $\tau$.
\end{lemma}

 We also introduce another
auxiliary function  $ U_{h,\gamma}^n  \in X_h $ such that: for $1\le n\le N$
\begin{equation}\label{eqn:fde:ful}
\Bar{\partial}^\alpha_\tau  U_{h,\gamma}^n +A_h  U_{h,\gamma}^n = P_hf( U_{h,\gamma}^{n-1})
\quad \text{with}~~ \gamma U_{h,\gamma}^0 + U_{h,\gamma}^N= P_hg.
\end{equation}

The next lemma provides an estimate for $U_{h,\gamma}^{0,\delta}-U_{h,\gamma}^{0}$.
\begin{lemma}
\label{lem:err-ful-noise}
For a fixed parameter $\mu \in (0,1]$, let $T_*$ be the constant such that $B_{\mu}(T_*) < 1$, where the function $B_{\mu}(\cdot)$ is defined in \eqref{eqn:contraction}. Assume that $T < T_*$ and  
$$\gamma^{-1 + \frac{\mu}{2}} h^{2-\mu} |\log h| + \tau^{\alpha\mu/2} + \tau^{\alpha} h^{-\mu} \leq c_*,$$
where $c_*$ is the constant given in Theorem \ref{lemma:stability_fully}.  
Let $U_{h,\gamma}^{n,\delta}$ and $U_{h,\gamma}^{n}$ be the solutions to problems \eqref{eqn:back-fully} and \eqref{eqn:fde:ful}, respectively. Then, the following estimate holds  $$\|U_{h,\gamma}^{0,\delta} - U_{h,\gamma}^{0}\|_{\dH{-\mu}} \leq c \gamma^{-1} \delta,$$ 
where the constant $c$ is independent of $\gamma$, $h$, and $\tau$.
\end{lemma}
\begin{proof}
Let $e_n^{\delta} = U_{h,\gamma}^{n,\delta} - U_{h,\gamma}^{n}$.
Then $e_n^{\delta}$ satisfies the relation that: for $1\le n\le N$
\begin{equation}\label{eqn:en-1}
\Bar{\partial}^\alpha_\tau e_n^{\delta} +A_h e_n^{\delta}  =P_h[f(U_{h,\gamma}^{n-1,\delta})-f(U_{h,\gamma}^{n-1})]
\quad \text{with}~~ 
 \gamma  e_0^{\delta}+ e_N^{\delta} = P_h(g_\delta -g).
\end{equation}
Using the solution representation~\eqref{eqn:back-fully-non-new} yields
\begin{equation*}
\begin{aligned}
 e_0^{\delta}&=(\gamma I+F_{h,\tau}^{N})^{-1} \Big[P_h(g_\delta-g) - \tau \sum_{k=1}^N E_{h,\tau}^{N-k}P_h[f(U_{h,\gamma}^{k-1,\delta})-f(U_{h,\gamma}^{k-1})]\Big].
\end{aligned}
\end{equation*}
Now we apply Lemma \ref{lem:op-reg-fully} to obtain
\begin{equation*}
\begin{aligned}
&\| e_0^{\delta}\|_{\dH{-\mu}}\le c\gamma^{-1} \delta + \left\|(\gamma I+F_{h,\tau}^{N})^{-1}\tau\sum_{k=1}^NE_{h,\tau}^{N-k}P_h[f(U_{h,\gamma}^{k-1,\delta})-f(U_{h,\gamma}^{k-1})]\right\|_{\dH{-\mu}}.
\end{aligned}
\end{equation*}
Applying the argument in Theorem~\ref{lemma:stability_fully} leads to the desired result.
\end{proof}

Time discretization would give the following fully error estimate.
\begin{lemma}\label{lem:err-reg:ful} 
Let $\tuh(t)$ and $U_{h,\gamma}^{n}$ be the solutions to \eqref{eqn:fde:reg:semi} and \eqref{eqn:fde:ful} respectively.
For a fixed parameter $\mu\in(0,1]$, 
let $T_*$ be the constant such that $B_{\mu}(T_*)<1$ with the function $B_{\mu}(\cdot)$ defined in \eqref{eqn:contraction}.
 Assume that $T < T_*$ and  
$$\gamma^{-1 + \frac{\mu}{2}} h^{2-\mu} |\log h| + \tau^{\alpha\mu/2} + \tau^{\alpha} h^{-\mu} \leq c_*,$$
where $c_*$ is the constant given in Theorem \ref{lemma:stability_fully}.  
Under these conditions, for $\nu \leq \mu$, the following estimate holds:
\[
\|\tuh(0) - U_{h,\gamma}^{0} \| _{\dH{-\mu}} \leq c \tau |\log \tau|^2 
\big(h^{\min\{-\nu,0\}} + \gamma^{-1}h^2|\log h| \big) 
\| u_0 \|_{\dH{\min\{-\nu,0\}}} ,
\]
where the constant $c$ is independent of $\gamma$, $h$, and $\tau$.
\end{lemma}
\begin{proof}
    Let $\bar U_{h,\gamma}^n$ be the solution to \eqref{eqn:fde:fulbar} 
and $e_n = \bar U_{h,\gamma}^n- U_{h,\gamma}^n$, which satisfies the following equation: for $1\le n\le N$
\begin{equation*}
\label{eqn:ful:err}
\bDal e_n  +A_he_n =P_h(f(\bar U_{h,\gamma}^{n-1})-f(U_{h,\gamma}^{n-1})) 
\quad \text{with}~~ 
\gamma e_0 + e_N =  \bar U_{h,\gamma}^N -  \tuh(T).
\end{equation*}
Then we apply the  representation of the fully discrete scheme to derive
\begin{equation*}
\label{eqn:repre:err:ful}
e_0 = (\gamma I+F_{h,\tau}^N)^{-1}\Big[(\bar U_{h,\gamma}^N -  \tuh(T))-\sum_{k=1}^N \tau E_{h,\tau}^{N-k}P_h(f(\bar U_{h,\gamma}^{k-1})-f(U_{h,\gamma}^{k-1})) \Big].
\end{equation*}
Thus we have 
\begin{equation*}
   \| e_0\|_{\dH{-\mu}}\le c\|A_h^{1-\frac\mu 2}(\bar U_{h,\gamma}^N -  \tuh(T))\|_{L^2\II}+\left\|(\gamma I+F_{h,\tau}^N)^{-1}\sum_{k=1}^N \tau E_{h,\tau}^{N-k}P_h(f(\bar U_{h,\gamma}^{k-1})-f(U_{h,\gamma}^{k-1}))\right\|_{\dH{-\mu}}.
\end{equation*}
Using Lemma~\ref{lem:err-reg:H2} and applying the argument in Theorem~\ref{lemma:stability_fully} give 
\begin{align*}
    \|\tuh(0)- U_{h,\gamma}^{0} \| _{\dH{-\mu}}\le c \tau |\log \tau|^2 \|u_{\gamma,h}(0)\| _{L^2\II}.
\end{align*}
We note that the equation \eqref{eqn:fde:reg:semisol} implies
\begin{align*}
    \|u_{\gamma,h}(0)\| _{L^2\II}\le c \|A_hP_hg\|_{L^2\II}+c\|A_h\int_0^TE_h(T-s)P_hf(u_{\gamma,h}(s))\ds\|_{L^2\II}.
\end{align*}
Applying the same argument in Corollary \ref{lem:reg-weak} and using Lemmas  \ref{lem:esitugh}, \ref{lemma:gammainequ} lead to 
\begin{align*}
    \|A_h\int_0^TE_h(T-s)P_hf(u_{\gamma,h}(s))\ds\|_{L^2\II}\le c_T \|u_{\gamma,h}(0)\|_{\dH{-\mu}}\le c_T(1+\gamma^{-1}h^2|\log h| ) \|u_0\|_{\dH{-\mu}}.
\end{align*}
By applying the inverse inequality in equation~\eqref{eqn:inverse} and utilizing the bound $\|A_h^{\frac{s}{2}} P_h g\|_{L^2\II} \leq c \|g\|_{\dH{s}}$ for $0 \leq s \leq 2$ (\cite[Theorem 4.2]{Karaa:2019}), along with the regularity results from Lemma~\ref{lem:Reg} and Corollary~\ref{lem:reg-weak}, we obtain
    \begin{align*}
        \|A_hP_hg\|_{L^2\II}\le ch^{\min\{-\nu,0\}}\|g\|_{\dH{\min\{2-\nu,2\}}}\le ch^{\min\{-\nu,0\}}\|u_0\|_{\dH{\min\{-\nu,0\}}}.
    \end{align*}
Therefore, we arrive at 
\begin{align*}
    \|u_{\gamma,h}(0)\| _{L^2\II}&\le ch^{\min\{-\nu,0\}}\|u_0\|_{\dH{\min\{-\nu,0\}}}+c(1+\gamma^{-1}h^2|\log h|) \|u_0\|_{\dH{-\mu}}\\&\le c(h^{\min\{-\nu,0\}}+\gamma^{-1}h^2|\log h| )\|
    u_0\|_{\dH{\min\{-\nu,0\}}}. 
\end{align*}
This completes the proof of the lemma.  
\end{proof}

Now we are ready to state the main theorem which shows the error of the numerical reconstruction from noisy data. The proof is a direct result of Lemma~\ref{lemma:gammainequ}, Lemma~\ref{lem:esitugh}, Lemma~\ref{lem:err-ful-noise}, and Lemma~\ref{lem:err-reg:ful}.


\begin{theorem}\label{thm:fully-err}
For a fixed parameter $\mu\in(0,1]$, 
let $T_*$ be the constant such that $B_{\mu}(T_*)<1$ with the function $B_{\mu}(\cdot)$ defined in \eqref{eqn:contraction}. 
 Assume that $T < T_*$ and  
$$\gamma^{-1 + \frac{\mu}{2}} h^{2-\mu} |\log h| + \tau^{\alpha\mu/2} + \tau^{\alpha} h^{-\mu} \leq c_*,$$
where $c_*$ is the constant given in Theorem \ref{lemma:stability_fully}.  
Let $U_{h,\gamma}^{0,\delta}$ be the numerically reconstructed initial data using the fully discrete scheme \eqref{eqn:back-fully}, and let $u_0$ be the exact initial data. Then, if $\|u_0\|_{\dot{H}^{-\mu+q}(\Omega)} \leq c$ with $q \in (0,2]$, the following estimate holds
\begin{equation*}
\|U_{h,\gamma}^{0,\delta}-u_0\|_{\dot{H}^{-\mu}(\Omega)}\le c\left(\gamma^\frac q2+ \gamma^{-1} \delta +
\gamma^{-1}h^2|\log h|+\tau {|\log \tau|^2} \left(\gamma^{-1}h^2|\log h|+h^{\min\{-\mu+q,0\}} \right)  \right).
\end{equation*}
Moreover, if $u_0 \in \dot{H}^{-\mu}(\Omega)$, then there holds
\begin{equation*}
\| U_{h,\gamma}^{0,\delta} -  u_0\|_{\dH {-\mu}}\rightarrow 0\quad \text{as}~~\delta,\gamma,h\rightarrow0^+,
~~\frac\delta\gamma\rightarrow0^+,~~\frac{\tau{|\log \tau|^2}}{h^{\mu}}\rightarrow0^+ ~~\text{and}~~ \frac{h^2|\log h|}{\gamma} \rightarrow0^+.
\end{equation*}

\end{theorem}

\begin{remark}
The \textsl{a priori} error estimate in Theorem \ref{thm:fully-err} provides a useful guideline for choosing the regularization parameter $\gamma$, as well as the discretization parameters $h$ and $\tau$, based on the noise level $\delta$. In particular, if $u_0 \in \dot{H}^ {-\mu+q}(\Omega)$, with $\mu>0, \ q\in [0,2]$ by choosing
\begin{equation*}
\gamma \sim \delta^{\frac{2}{q+2}},~~h^2|\log h|\sim \delta,~\text{and}~~\tau |\log \tau|^2 h^{\min\{-\mu+q,0\} }\sim \delta^{\frac{q}{q+2}},
\end{equation*}
we obtain the optimal approximation error
\begin{equation*}
\|U_{h,\gamma}^{0,\delta} - u_0\|_{\dot{H}^{-\mu}(\Omega)}\le c\delta^{\frac{q}{q+2}}.
\end{equation*}
Our theory requires $\mu > 0$, and the generic constant in the estimate may diverge as $\mu \rightarrow 0$. The result can be extended to the case $\mu = 0$ under the strong condition \eqref{eqn:lipconstant-2}, as discussed in Remarks \ref{rem:-mu} and \ref{rem:iter}. However, avoiding the use of condition \eqref{eqn:lipconstant-2} in general remains an open problem and warrants further investigation. 
\end{remark}

\section{Numerical examples}\label{sec:5}
In this section, we test several two-dimensional examples to illustrate our theoretical results and to examine the necessity of our assumptions. We consider the two-dimensional subdiffusion model \eqref{eqn:fde} in the domain $\Omega = (0,1)^2$. For spatial discretization, we employ the standard Galerkin piecewise linear Finite Element Method with a uniform mesh size of $h $. For temporal discretization, we use the backward Euler convolution quadrature method with a uniform time step size of $\tau$.

To obtain the exact solution $u(T)$ as the observational data, we solve the direct problem using fine meshes, specifically setting $h=1/256$ and $\tau = T/1000$. Subsequently, we compute the noisy observational data as follows:
\begin{equation*}
g_\delta = u(T) + \epsilon\delta\sup_{x\in\Omega} u(x,T),
\end{equation*}
where $\epsilon$ is generated from a standard Gaussian distribution, and $\delta$ represents the associated noise level. We will compute the numerical reconstruction of the initial data based on Algorithm \ref{alg}.
All the computations are carried out on a personal desktop using MATLAB 2022.


We apply two nonlinear functions $f(u)$:
\begin{equation*}
  f(u)=\sqrt{1+u^2}\quad \text{and}\quad f(u)=1-u^3,
\end{equation*}
and test the following two types of initial data:  
\begin{itemize}
\item[(1)] \textbf{Example 1.}  Smooth initial data:
\begin{equation*}
    u_0=\sin(2\pi x)\sin(2\pi y) \in \dot H^2(\Omega).
\end{equation*} 
\item[(2)] \textbf{Example 2.} Nonsmooth initial data:
\begin{equation}\label{eqn:nonsm}
u_0 = \begin{cases}
1,~~\text{if}~(x,y)\in [0,0.5]^2\cup [0.5,1]^2,\\
0,~~\text{otherwise.}
\end{cases}\in \dH{\frac12-\epsilon}\quad \forall ~\epsilon\in(0,\frac12).
\end{equation}
\end{itemize}

For a given noise level $\delta$, we select the discretization parameters $\gamma, h,\tau$ based on Theorem \ref{thm:fully-err}.  For ease of implementation, we test the case $\mu=0$ that is beyond Theorem \ref{thm:fully-err}.
We evaluate the relative error in $L^2$ norm, defined as
\begin{equation}\label{eqn:equ}
e_u =  \|U_{h,\gamma}^{0,\delta} - u_0\|_{L^2\II}/\|u_0\|_{L^2\II},
\end{equation}
where $u_0$ is the exact initial data and $U_{h,\gamma}^{0,\delta} $ is the numerical reconstruction obtained by using Algorithm \ref{alg}.

For \textbf{Example 1} with smooth initial data, we compute $U_{h,\gamma}^{0,\delta}$ with $\gamma, \tau, h \sim \sqrt{\delta}$ and expect a convergence of order $O(\sqrt{\delta})$ according to Theorem \ref{thm:fully-err}. In our numerical experiments, we set $T=1$, $\delta = 1/K$, $\gamma = \sqrt{\delta}/75$, $\tau = \sqrt{\delta}/5$, and $h =  5 \sqrt{\delta}/{8}$, with $K=80$, $160$, $320$, $480$ and $640$. The errors in reconstruction are presented in Tables \ref{fig:fully:smooth:err1}--\ref{fig:fully:smooth:err2}. The numerical results fully support our expectations. Furthermore, our numerical results indicate that the recovery is stable for all $\alpha \in (0,1)$.

When the initial data is nonsmooth, then the convergence rate deteriorates. For \textbf{Example 2} (nonsmooth data), the initial data 
$u_0 \in \dot{H}^{\frac{1}{2}-\varepsilon}$ for any $\varepsilon\in (0,\frac{1}{2})$. According to Theorem \ref{thm:fully-err} (with $\mu=0$), we expect an optimal rate  $O(\delta^{\frac15})$ provided that $\gamma \sim \delta^{\frac45}$, $h \sim \sqrt{\delta}$, and $\tau \sim \delta^{\frac15}$. This is fully supported by the numerical results presented in Tables \ref{fig:fully:err:case3}--\ref{fig:fully:err:case4},
where we set $T=1$, $\delta =1/K$, $\gamma=\delta^{\frac45}/10$, $\tau =\delta^{\frac15}/10$, $h=5\sqrt{\delta}/6$ with $K=400$, $800$, $1200$, $1600$ and $2000$.

\begin{table}[h!]
\renewcommand\arraystretch{0.9}
\begin{center}
\def\temptablewidth{0.85\textwidth}
\caption{Reconstruction error: Example 1 with $f=\sqrt{1+u^2}$ and $\delta =1/K$.}\label{fig:fully:smooth:err1}\vspace{-0.2in}
{\rule{\temptablewidth}{1pt}}
\begin{tabular*}{\temptablewidth}{@{\extracolsep{\fill}}cccccc}
 &$K=80$&$K=160 $&  $K=320 $ & $K=480  $&$K=640$\\ \hline
$\alpha=0.1$ &  3.551e-1   &  2.532e-1   &  1.808e-1 &    1.472e-1 &    1.270e-1 \\
  order& -&     0.4879    &0.4858    &0.5066    &0.5125\\
  \hline
$\alpha=0.3$ &    3.991e-1  &   2.749e-1   &  2.006e-1  &   1.607e-1   &  1.381e-1 \\
  order& -&     0.5378   & 0.4546 &   0.5470   & 0.5270\\
  \hline
$\alpha=0.5$ &  4.642e-1  &   3.291e-1     &2.349e-1 &    1.825e-1  &   1.593e-1 \\
  order& -&  0.4965   & 0.4863  &  0.6221  &  0.4742
\\
  \hline
    $\alpha=0.7$ &   6.018e-1   &  4.281e-1    & 3.045e-1   &  2.334e-1   &  1.977e-1 \\
  order& -&       0.4912  &  0.4917   & 0.6551 &   0.5779
\\
\end{tabular*}
{\rule{\temptablewidth}{1pt}}
\end{center}
\end{table}

\begin{table}[h!]
\renewcommand\arraystretch{0.9}
\begin{center}
\def\temptablewidth{0.85\textwidth}
\caption{Reconstruction error: Example 1 with $f=1 - u^3$ and $\delta =1/K$.  }\label{fig:fully:smooth:err2}\vspace{-0.2in}
{\rule{\temptablewidth}{1pt}}
\begin{tabular*}{\temptablewidth}{@{\extracolsep{\fill}}cccccc}
 &$K=80$&$K=160 $&  $K=320 $ & $K=480  $&$K=640$\\ \hline
$\alpha=0.1$ &  3.630e-1   &  2.561e-1   &  1.824e-1     &1.477e-1     &1.261e-1 \\
  order& -&     0.5029  &  0.4902   & 0.5200   & 0.5507\\
  \hline
$\alpha=0.3$ &   3.909e-1    & 2.804e-1    & 1.947e-1    & 1.591e-1   &  1.373e-1 \\
  order& -& 0.4796  &  0.5264  &  0.4974   & 0.5110\\
  \hline
$\alpha=0.5$ &    4.629e-1   &  3.232e-1    & 2.304e-1     &1.836e-1    & 1.581e-1 \\
  order& -& 0.5184   & 0.4882    &0.5602    &0.5192
\\
  \hline
    $\alpha=0.7$ &   6.080e-1   &  4.354e-1 &    3.080e-1  &   2.344e-1   &  2.017e-1 \\
  order& -&       0.4818  &  0.4993&    0.6736  &  0.5219
\\
\end{tabular*}
{\rule{\temptablewidth}{1pt}}
\end{center}
\end{table}

\begin{table}[h!]
\renewcommand\arraystretch{0.9}
\begin{center}
\def\temptablewidth{0.85\textwidth}
\caption{Reconstruction error: Example 2 with $f=\sqrt{1+u^2}$ and $\delta =1/K$.}\label{fig:fully:err:case3}
\vspace{-0.2in}
{\rule{\temptablewidth}{1pt}}
\begin{tabular*}{\temptablewidth}{@{\extracolsep{\fill}}cccccc}
&  $K=400 $ & $K=800  $&$K=1200 $&$K=1600$&$K=2000 $\\ \hline
$\alpha=0.2$ &     3.017e-1   & 2.583e-1   & 2.365e-1  &  2.225e-1 &   2.131e-1\\
  order& -&  
 0.2243 &   0.2166   & 0.2135   & 0.1919\\
  \hline
$\alpha=0.4$ &    3.095e-1  &  2.661e-1  &  2.440e-1&    2.306e-1 &   2.194e-1\\
  order& -&    0.2179   & 0.2146   & 0.1954 &   0.2239\\
  \hline
$\alpha=0.6$ &   3.275e-1   & 2.810e-1&    2.587e-1 &   2.456e-1  &  2.346e-1\\
  order& -&   0.2208   & 0.2042   & 0.1806 &   0.2058
\\
  \hline
    $\alpha=0.8$ &  3.667e-1  &  3.150e-1    &2.923e-1    &2.788e-1 &   2.649e-1 \\
  order& -&       0.2191   & 0.1845   & 0.1643   & 0.2301
\end{tabular*}
{\rule{\temptablewidth}{1pt}}
\end{center}
\end{table}

\begin{table}[h!]
\renewcommand\arraystretch{0.9}
\begin{center}
\def\temptablewidth{0.85\textwidth}
\caption{Reconstruction error: Example 2 with $f=1 - u^3$ and $\delta =1/K$. }\label{fig:fully:err:case4}\vspace{-0.2in}
{\rule{\temptablewidth}{1pt}}
\begin{tabular*}{\temptablewidth}{@{\extracolsep{\fill}}cccccc}
&  $K=400 $ & $K=800  $&$K=1200 $&$K=1600$&$K=2000 $\\ \hline
$\alpha=0.2$ & 3.014e-1  &  2.583e-1  &   2.364e-1 &    2.228e-1  &   2.120e-1  \\
  order& -&  
 0.2225 &   0.2182  &  0.2069   & 0.2216\\
  \hline
$\alpha=0.4$ & 3.102e-1   &  2.663e-1   &  2.442e-1  &   2.291e-1  &   2.199e-1  \\
  order& -&   0.2198 &   0.2144  &  0.2210 &   0.1835\\
  \hline
$\alpha=0.6$ & 3.278e-1    & 2.805e-1   &  2.585e-1  &   2.457e-1 &    2.348e-1   \\
  order& -&      0.2251   & 0.2008  &  0.1777    &0.2032
\\
  \hline
    $\alpha=0.8$ &3.647e-1    & 3.167e-1  &   2.925e-1  &   2.773e-1  &   2.661e-1  \\
  order& -&       0.2037   & 0.1959 &   0.1855   & 0.1845
\end{tabular*}
{\rule{\temptablewidth}{1pt}}
\end{center}
\end{table}



Next, we examine the convergence of the iteration in Algorithm \ref{alg} with different $\alpha$, $T$, and the Lipschitz constant $L$. For this test, we select the nonlinear function as 
$$f(u)=L\sqrt{1+u^2},$$ 
and use smooth initial data. Additionally, we fix the values of $\delta = 10^{-4}$, $h = 10^{-2}$, and $\tau = T/100$.
Let $U_{0,j}$ denote the numerical reconstruction obtained after the $j$-th iteration of Algorithm \ref{alg}, and calculate the error at each iteration as follows:
$$
e_j = \|U_{0,j} - u_0\|_{L^2\II}/\|u_0\|_{L^2\II} \quad \text{for all}~j\ge 0.
$$

Figures \ref{fig:10} and \ref{fig:20} present the convergence histories with different values of $T$, $L$, and $\alpha$. The numerical results clearly show that when $L$ is small, the iteration converges linearly even with a relatively large $T$, thus achieving a reasonable reconstruction of the initial data. Moreover, we observe that the convergence rate increases as either $L$, $T$, or $\alpha$ decreases.
Conversely, when $L$ is large, we observe that if $T$ is not small enough, the iteration might diverge, as shown in Figure \ref{fig:20}. These phenomena indicate the necessity of the assumption on $T$ in the stability estimate (Theorem \ref{thm:stability}) and error estimates (Theorem \ref{thm:fully-err}). 
\begin{figure}[h!]
		\centering
	\includegraphics[trim={0.8in 0.2in 1in 0.2in},clip,width=0.3\textwidth]{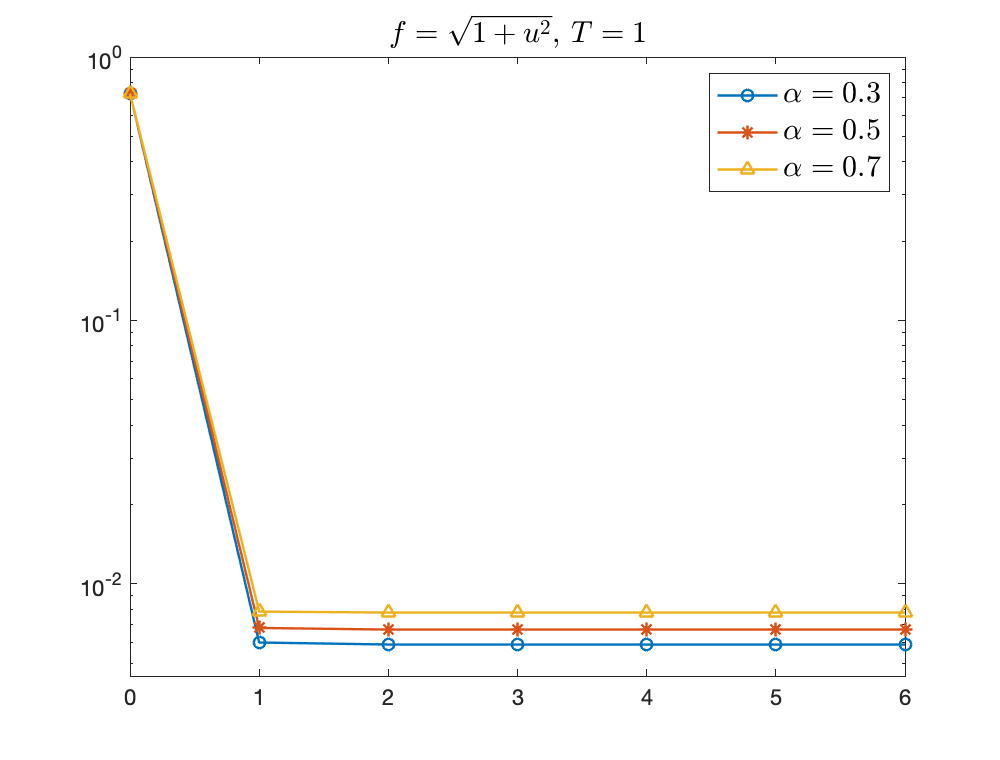}
	\includegraphics[trim={0.8in 0.2in 1in 0.2in},clip,width=0.3\textwidth]{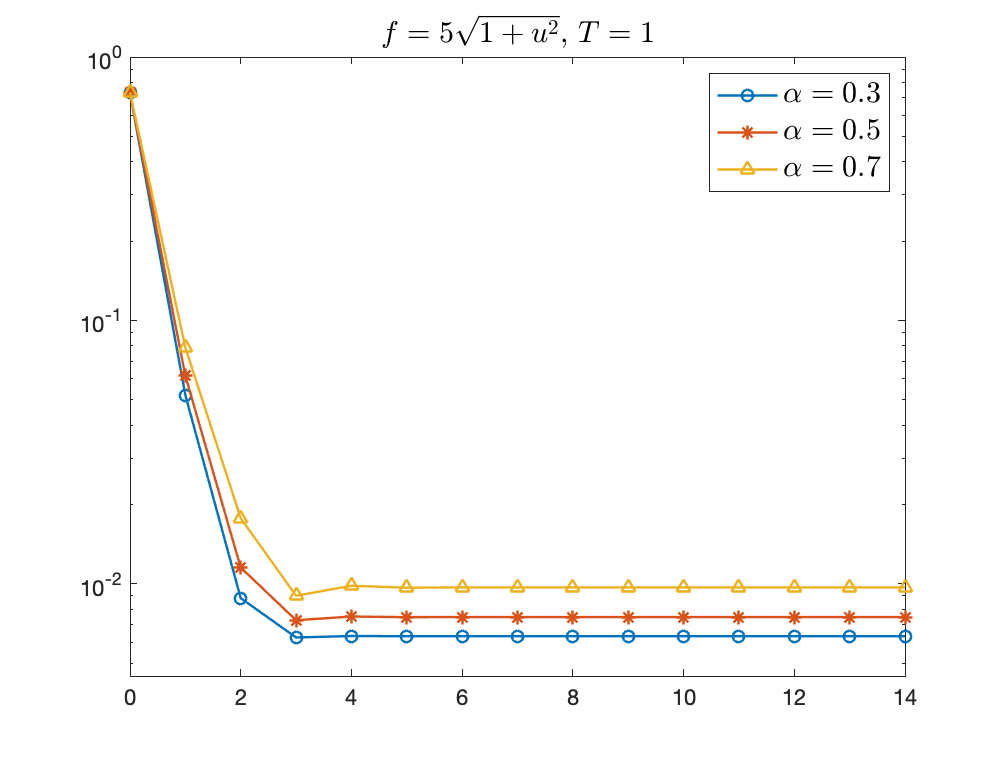}
 \includegraphics[trim={0.8in 0.2in 1in 0.2in},clip,width=0.3\textwidth]{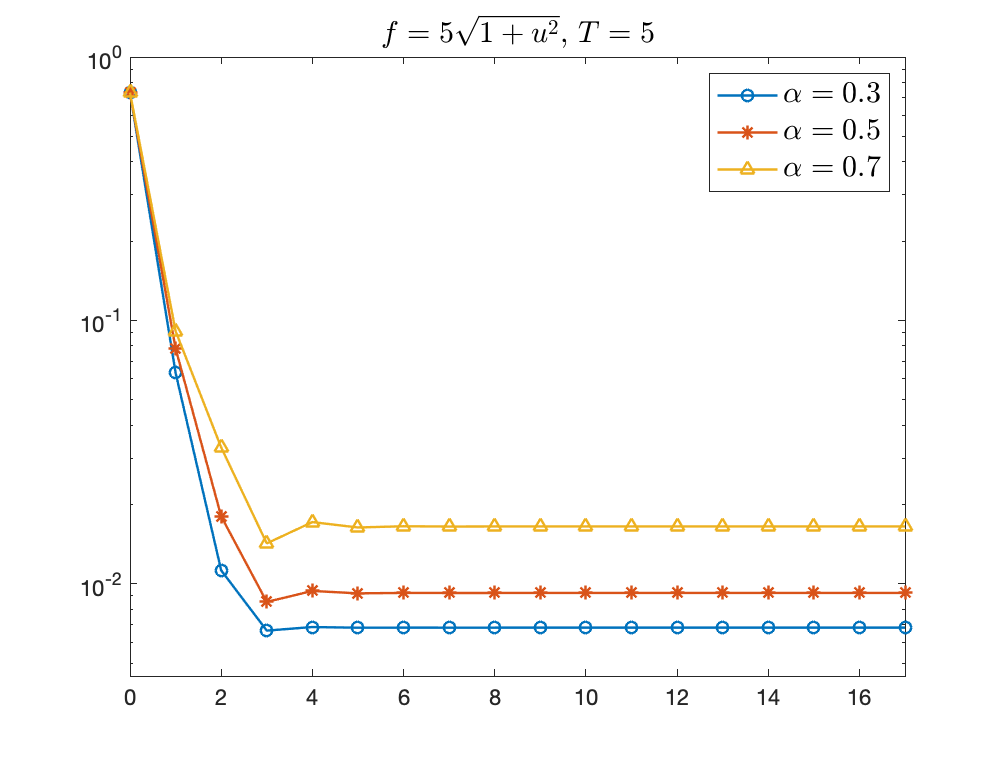}
	\caption{Convergence histories of Algorithm~\ref{alg} with different $T$, $\alpha$ and $L$. }\label{fig:10}
\end{figure}
\begin{figure}[h!]
		\centering
	\includegraphics[trim={0.8in 0.2in 1in 0.2in},clip,width=0.32\textwidth]{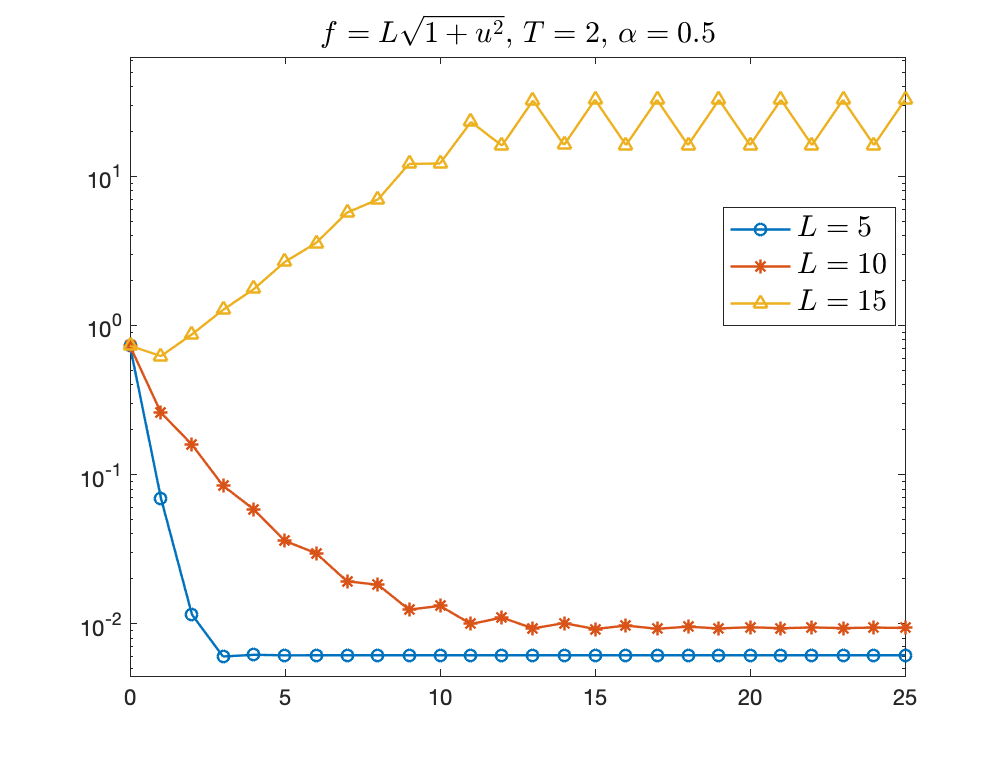}
 \includegraphics[trim={0.8in 0.2in 1in 0.2in},clip,width=0.32\textwidth]{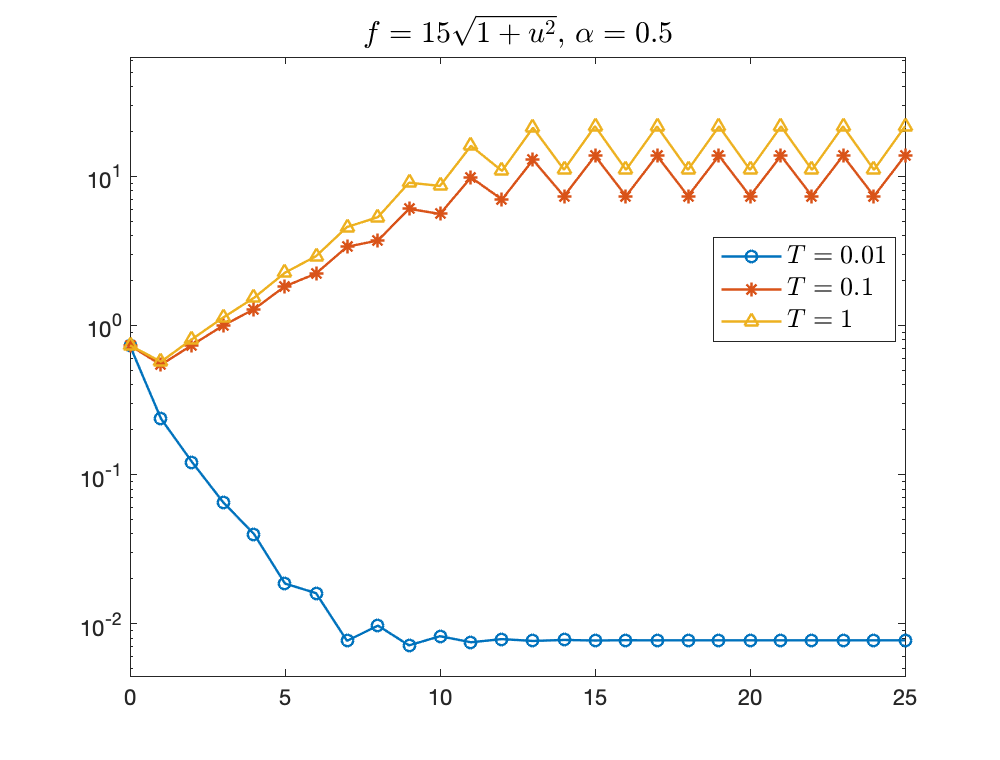}
	\caption{Convergence histories of Algorithm~\ref{alg} with different $T$, and $L$.}\label{fig:20}
\end{figure}

Finally, to illustrate the significant difference between the classical diffusion and the subdiffusion, we test several numerical experiments with the nonlinear term $$f(u) = u - u^3$$
and the piecewise constant initial data \eqref{eqn:nonsm}. First, we fix the terminal time $T = 1$ and examine the influence of the fractional order $\alpha$ on the reconstruction of the initial data.  In Figure \ref{fig:1}, we test the reconstruction of the initial data $ U_{h,\gamma}^{0,\delta}$ for $\alpha = 0.9, 0.99,$ and $ \delta = 10^{-3}, 5\times10^{-4}, 2\times10^{-4}$. As expected, recovering the initial data becomes increasingly difficult as $ \alpha $ approaches 1.

\begin{figure}[htp!]
 \centering
	\begin{tabular}{ccc}
  \includegraphics[trim={0.5in 0.5in 0.5in 0in},clip,width=0.275\textwidth]{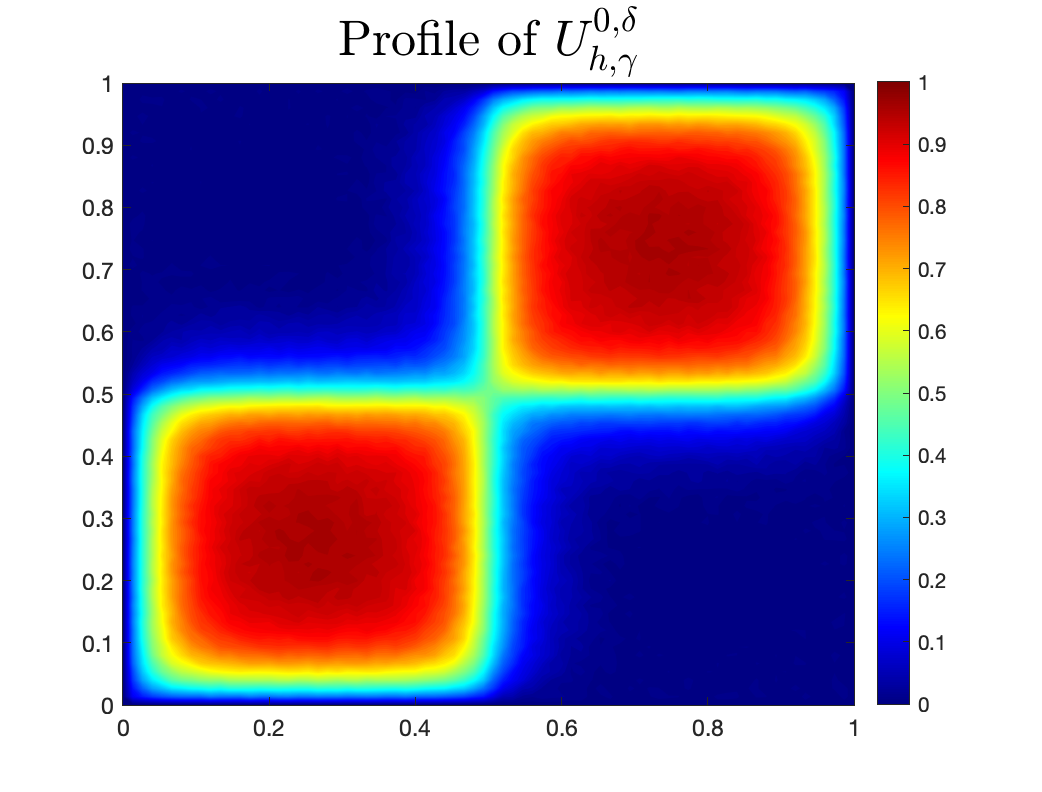}&
	  \includegraphics[trim={0.5in 0.5in 0.5in 0in},clip,width=0.275\textwidth]{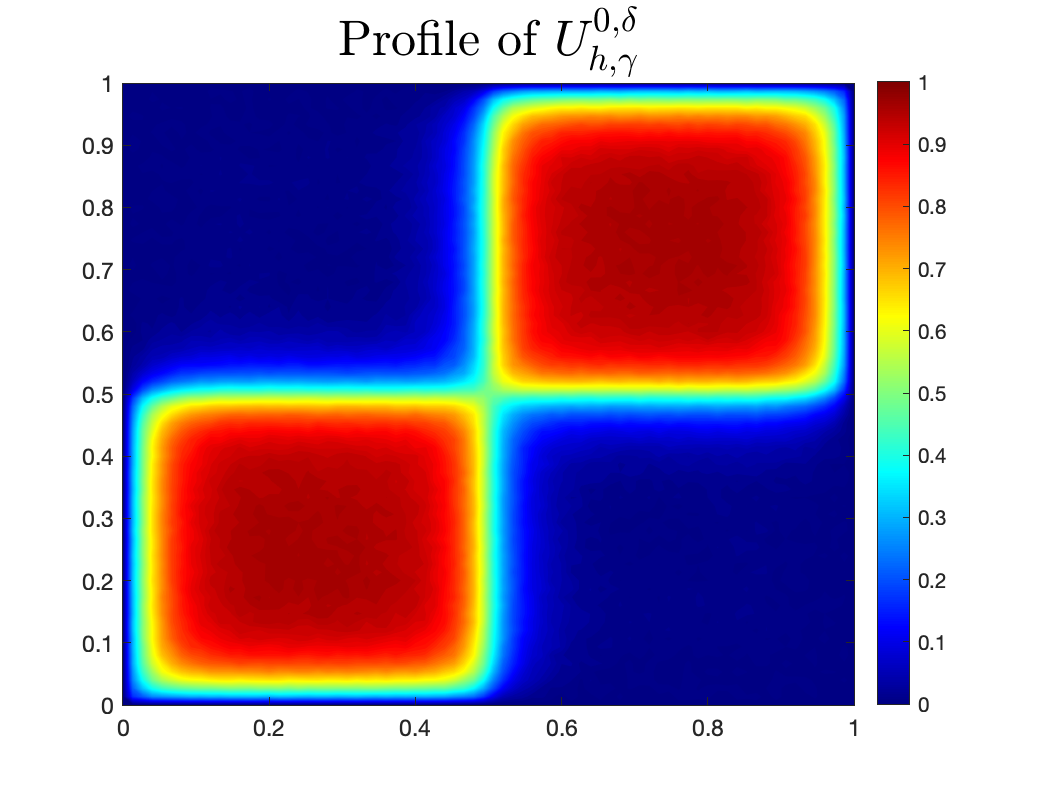}&
   \includegraphics[trim={0.5in 0.5in 0.5in 0in},clip,width=0.275\textwidth]{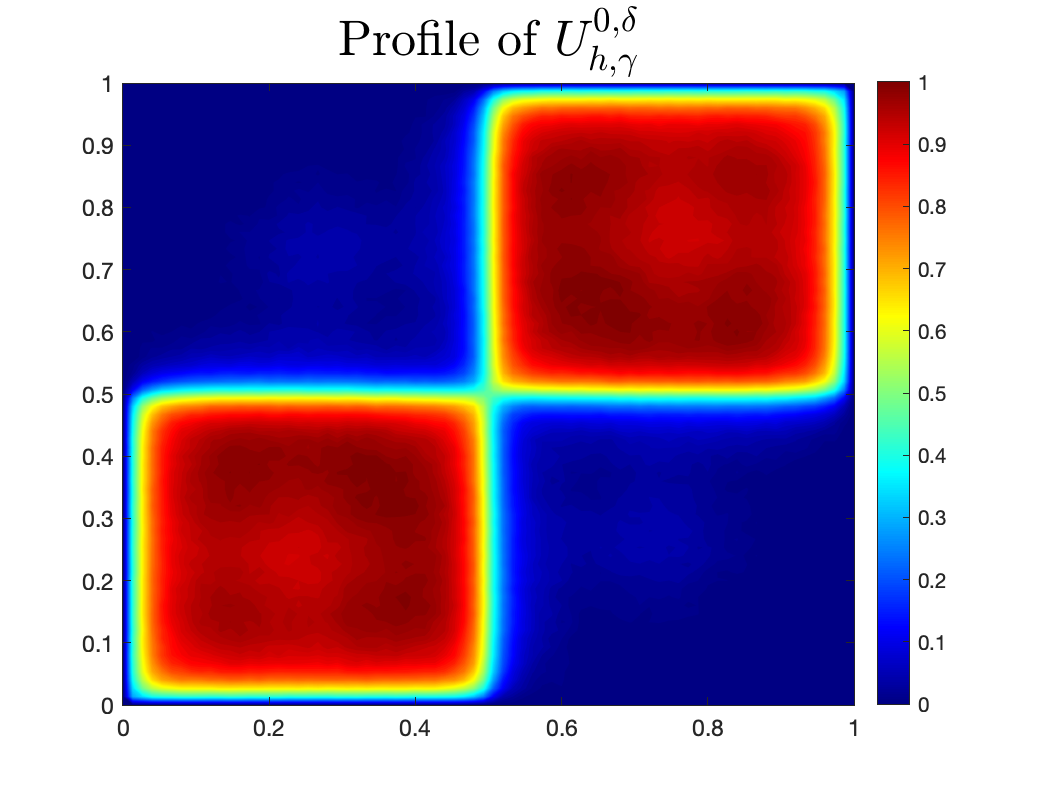}\\
   (a) $\alpha=0.9,\delta=10^{-3}$&  (b)   $\alpha=0.9,\delta=5\times10^{-4}$ &  (c) $\alpha=0.9,\delta=2\times10^{-4}$ \\
    \includegraphics[trim={0.5in 0.5in 0.5in 0in},clip,width=0.275\textwidth]{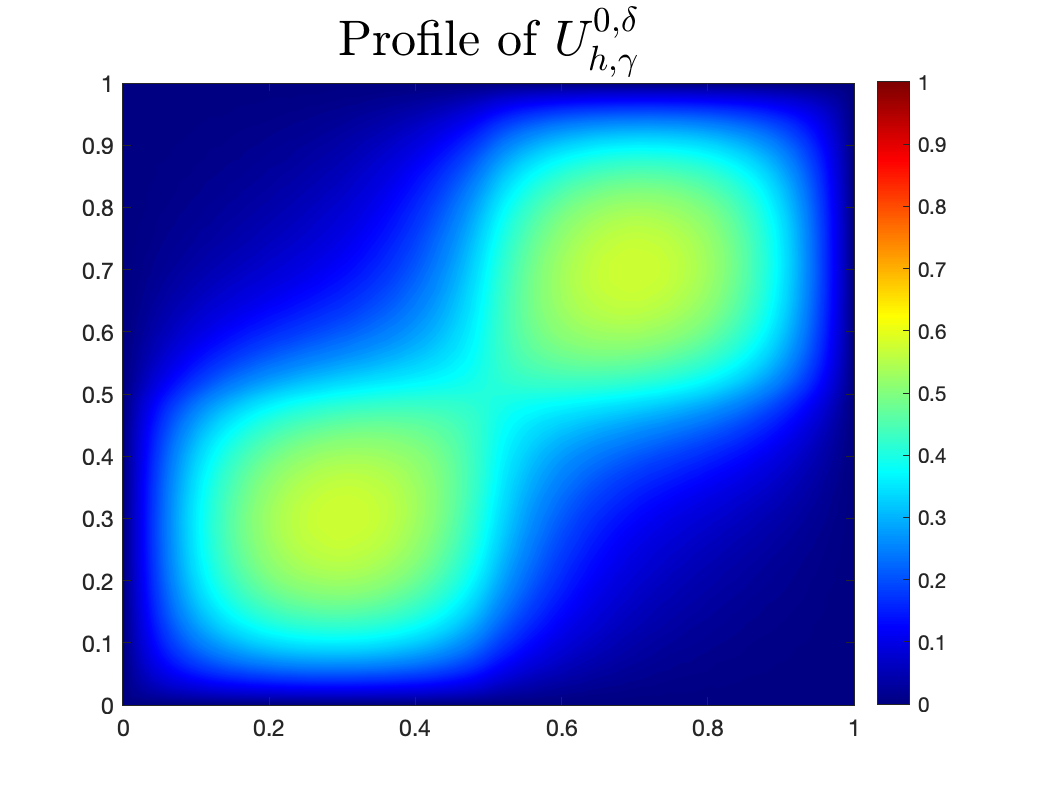}&
	  \includegraphics[trim={0.5in 0.5in 0.5in 0in},clip,width=0.275\textwidth]{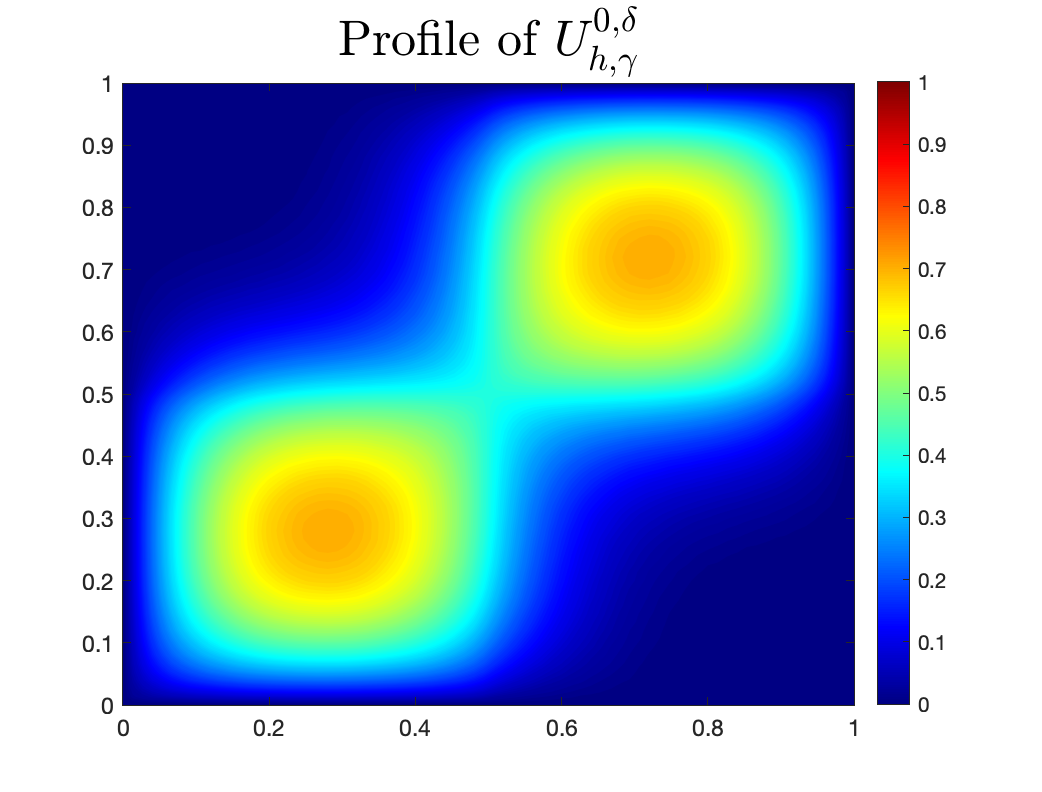}&
   \includegraphics[trim={0.5in 0.5in 0.5in 0in},clip,width=0.275\textwidth]{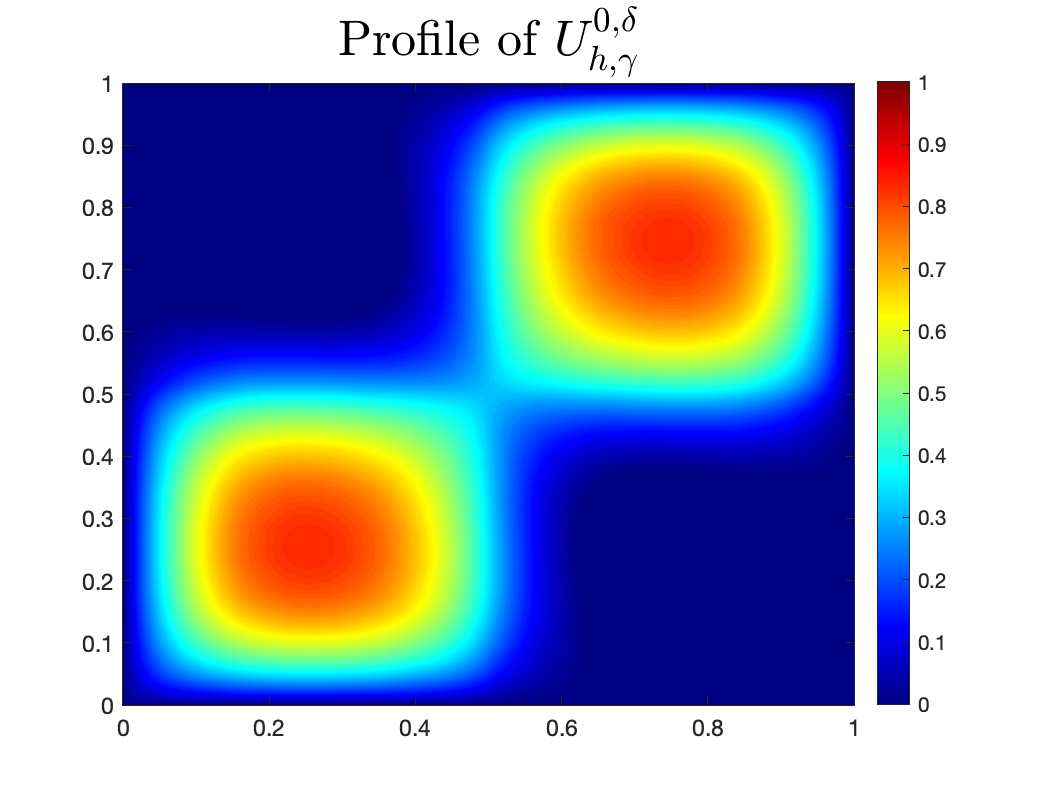}\\
		 (d) $\alpha=0.99,\delta=10^{-3}$&  (e)   $\alpha=0.99,\delta=5\times10^{-4}$ &  (f) $\alpha=0.99,\delta=2\times10^{-4}$ 
	\end{tabular}
\caption{The numerical reconstruction $U_{h,\gamma}^{0,\delta}$  for  $T=1$ with different $\alpha$ and $\delta$.}\label{fig:1}
\end{figure}

We also examine the more interesting case of a relatively large terminal time, e.g. $T = 10$, in our computation. As shown in Figure \ref{fig:2}, for $\alpha = 0.9$, we still observe a reasonable reconstruction; however, it is less accurate compared to the reconstruction for a shorter terminal time $T = 1$ (cf. Figure \ref{fig:1}). Moreover, as $\alpha$ approaches one, the numerical recovery of the initial condition becomes increasingly challenging; for example, see case $\alpha = 0.99$ in Figure \ref{fig:2}.  
In particular, for $\alpha = 1$, even with a very small noise level and a small terminal time $T$, accurately capturing the correct profile of the initial data becomes extremely difficult due to the severe ill-posedness of the parabolic backward problem, as illustrated in Figure \ref{fig:3}. This highlights the fundamentally different ill-posed nature of the subdiffusion model compared to the classical diffusion model.

\begin{figure}[htp!]
 \centering
	\begin{tabular}{ccc}
  \includegraphics[trim={0.5in 0.5in 0.5in 0in},clip,width=0.275\textwidth]{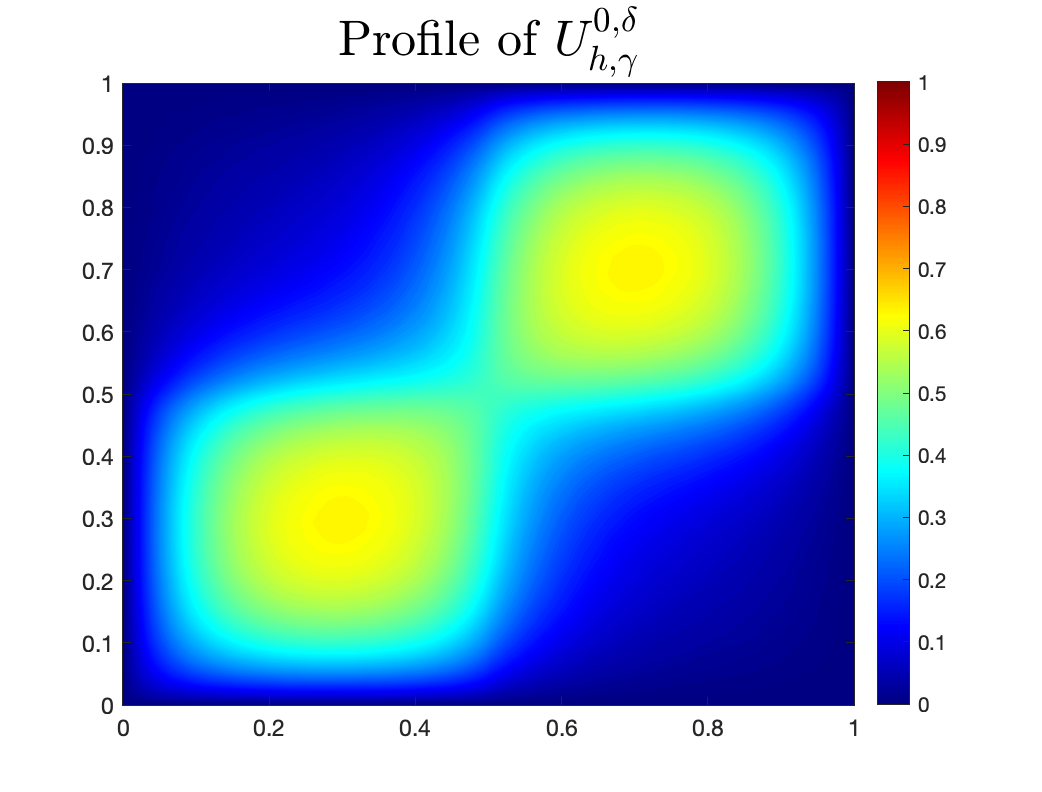}&
	  \includegraphics[trim={0.5in 0.5in 0.5in 0in},clip,width=0.275\textwidth]{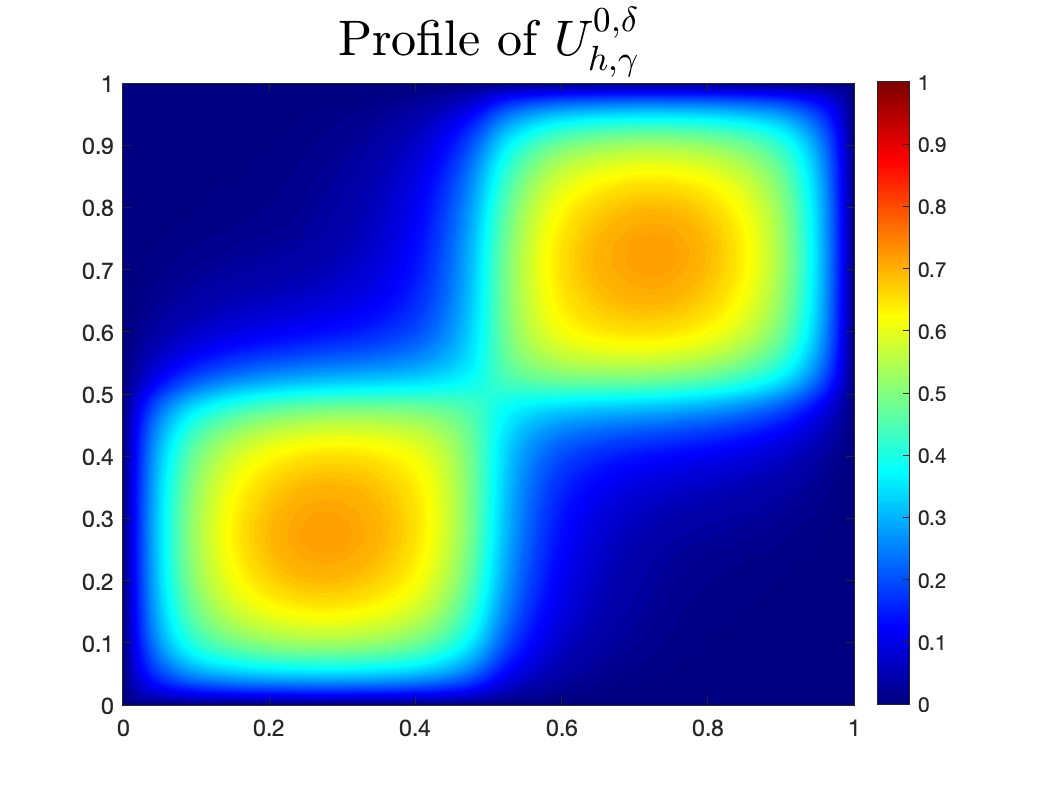}&
   \includegraphics[trim={0.5in 0.5in 0.5in 0in},clip,width=0.275\textwidth]{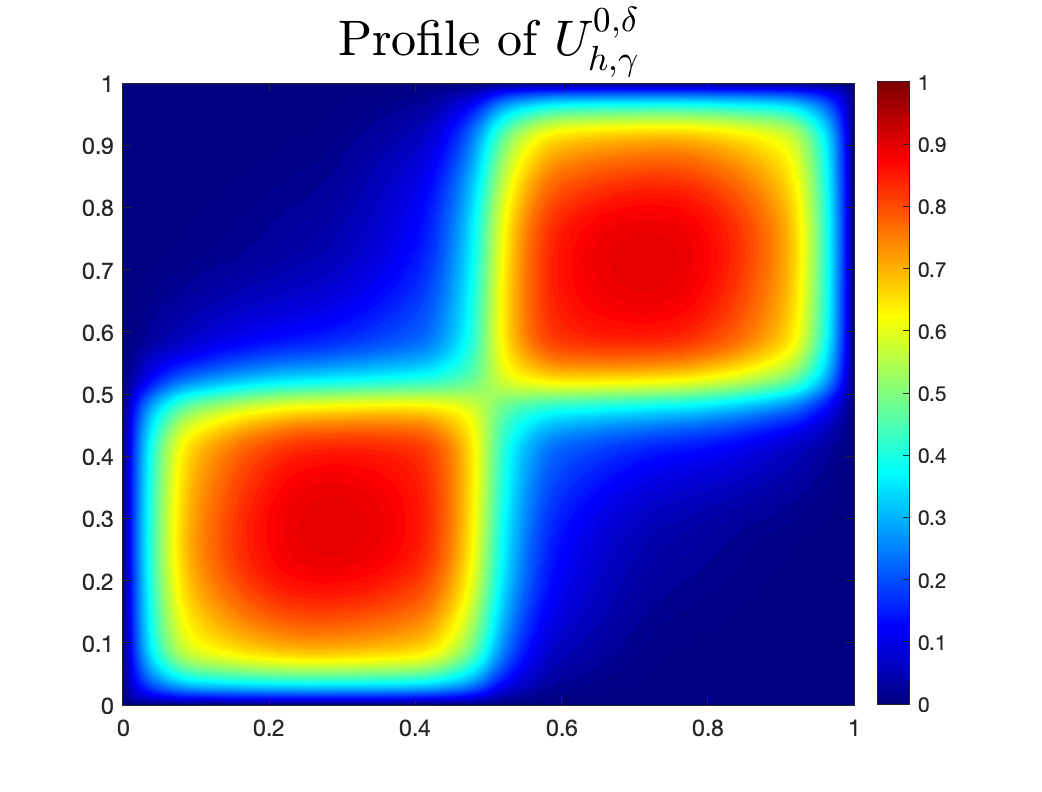}\\
   (a) $\alpha=0.9,\delta=10^{-3}$&  (b)   $\alpha=0.9,\delta=5\times10^{-4}$ &  (c) $\alpha=0.9,\delta=2\times10^{-4}$ \\
    \includegraphics[trim={0.5in 0.5in 0.5in 0in},clip,width=0.275\textwidth]{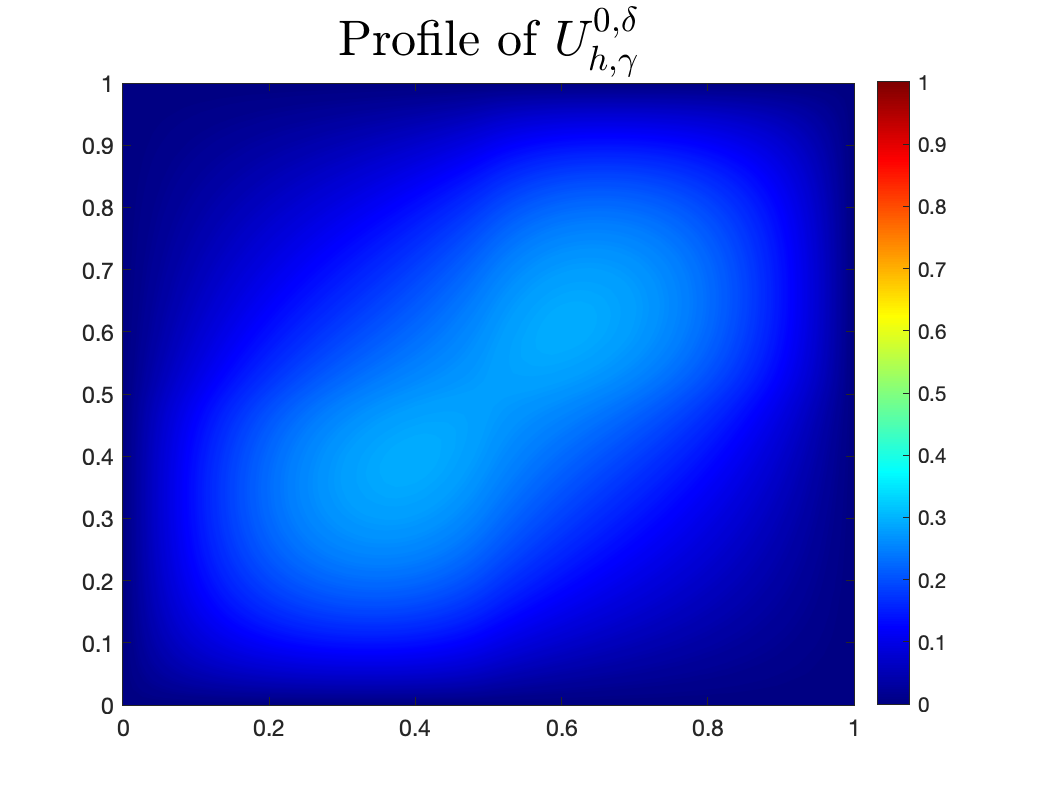}&
	  \includegraphics[trim={0.5in 0.5in 0.5in 0in},clip,width=0.275\textwidth]{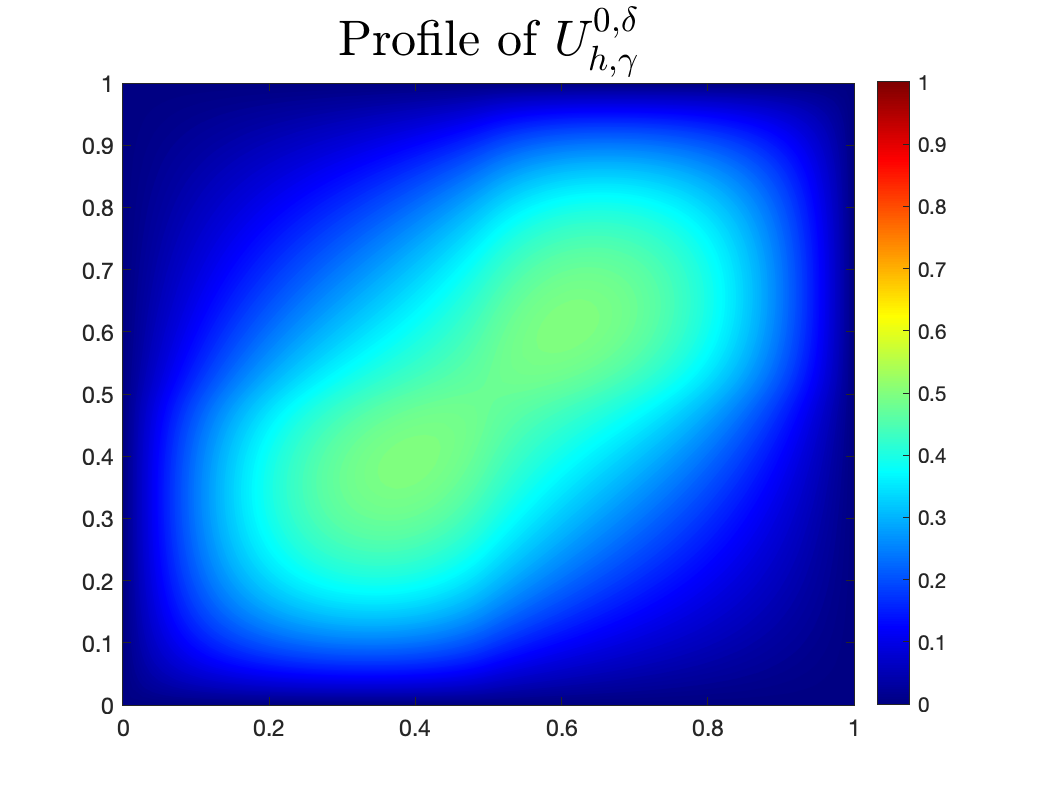}&
   \includegraphics[trim={0.5in 0.5in 0.5in 0in},clip,width=0.275\textwidth]{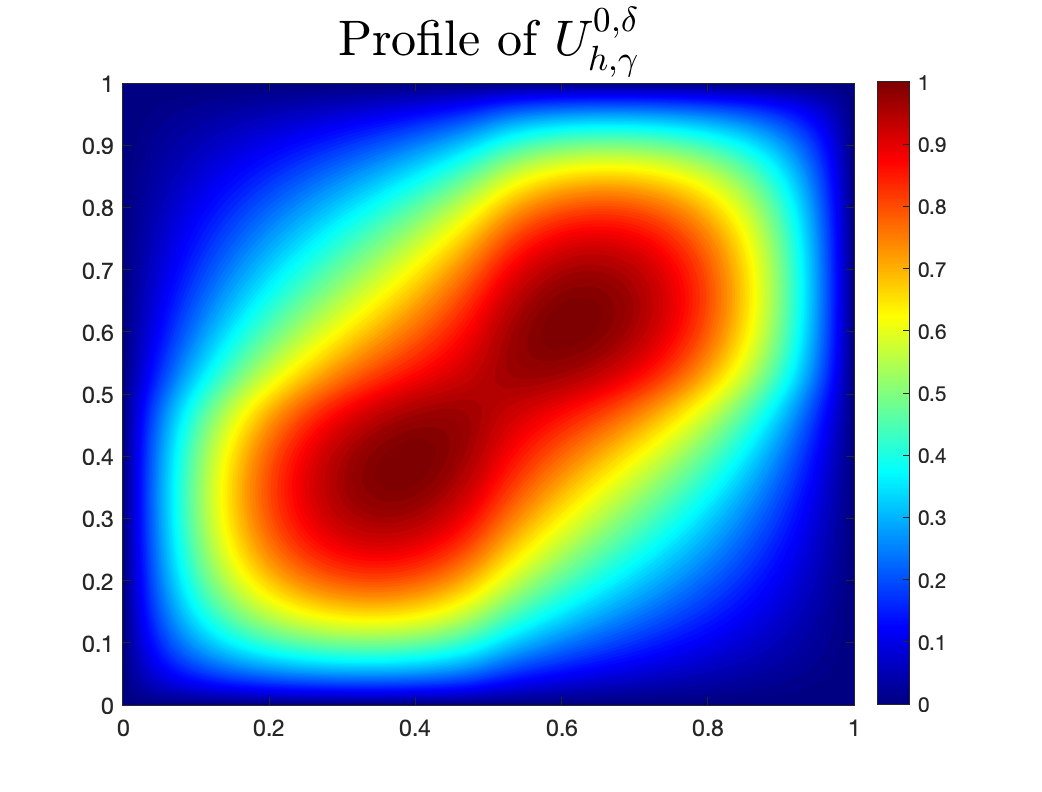}\\
		 (d) $\alpha=0.99,\delta=10^{-3}$&  (e)   $\alpha=0.99,\delta=5\times10^{-4}$ &  (f) $\alpha=0.99,\delta=2\times10^{-4}$ 
	\end{tabular}
\caption{The numerical reconstruction $U_{h,\gamma}^{0,\delta}$ for large $T=10$ with different $\alpha$ and $\delta$.}\label{fig:2}
\end{figure}

\begin{figure}[htp!]
 \centering
	\begin{tabular}{ccc}
  \includegraphics[trim={0.5in 0.5in 0.5in 0in},clip,width=0.275\textwidth]{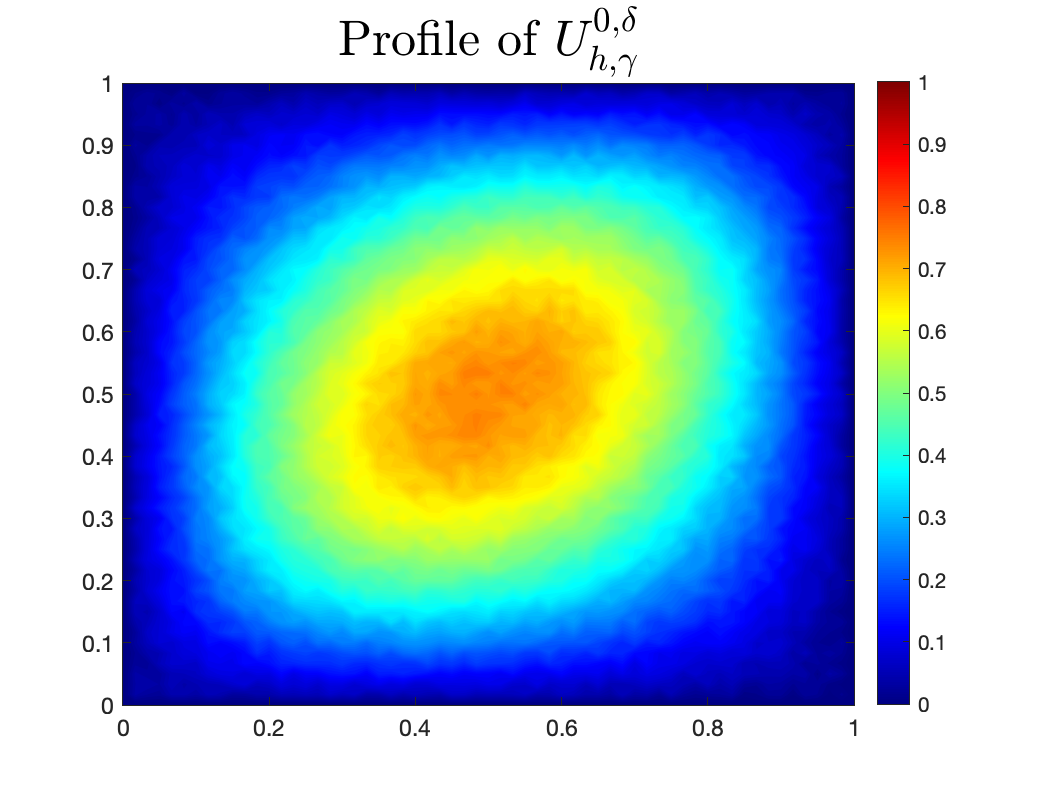}&
	  \includegraphics[trim={0.5in 0.5in 0.5in 0in},clip,width=0.275\textwidth]{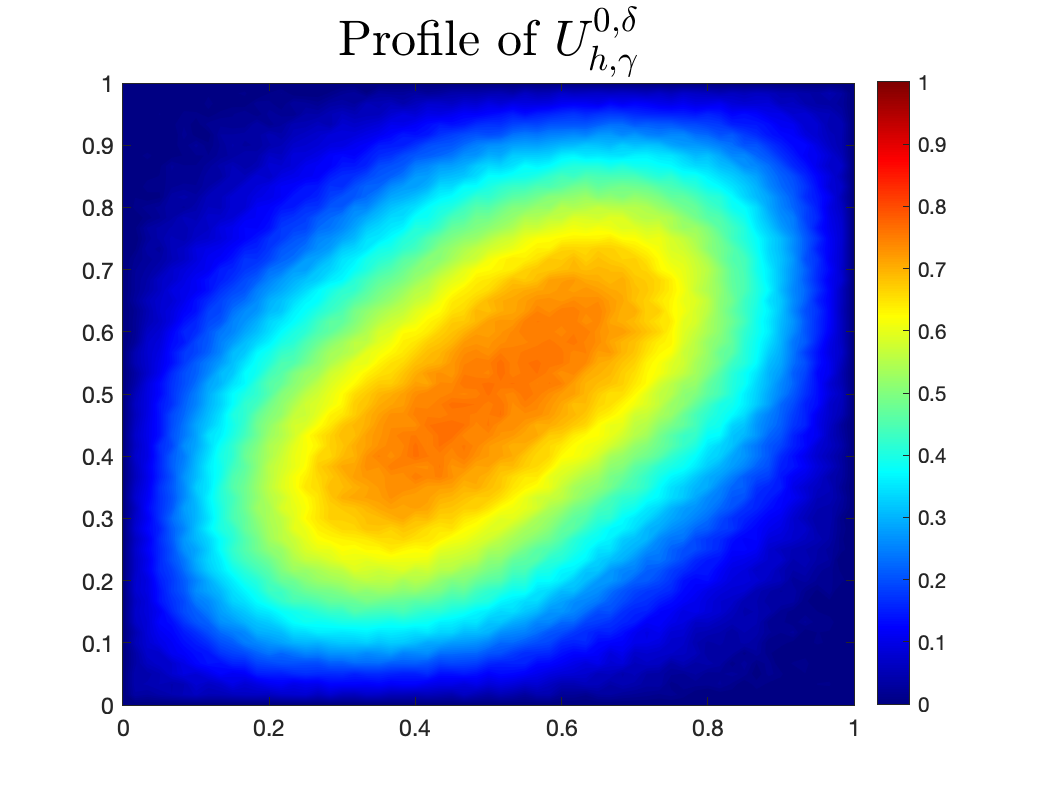}&
   \includegraphics[trim={0.5in 0.5in 0.5in 0in},clip,width=0.275\textwidth]{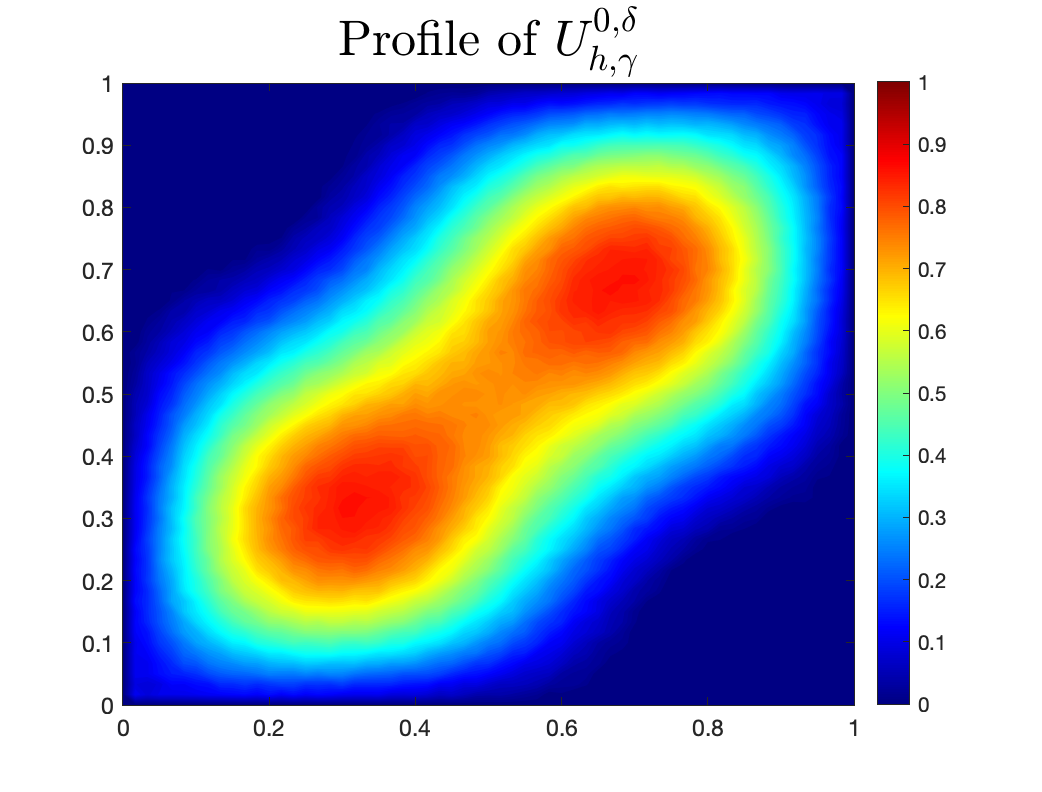}\\
   (a) $\delta=2\times10^{-3}$&  (b)   $\delta=4\times10^{-4}$ &  (c) $\delta=8\times10^{-5}$ 
	\end{tabular}
\caption{The numerical reconstruction $U_{h,\gamma}^{0,\delta}$  for  $T=0.1$ with $\alpha=1$ and different $\delta$.}\label{fig:3}
\end{figure}

\section{Concluding remarks}\label{sec:conclusion}
In this work, we study the backward problem of nonlinear subdiffusion equations. From the terminal observation $u(T)$, we reconstruct the initial data $u_0$.
Under some mild conditions on $T$, the existence, uniqueness, and conditional stability of the solution to the inverse problem are theoretically established by applying the smoothing and asymptotic properties of solution operators and constructing a fixed-point iteration.  Furthermore, in case of noisy observations, we utilize the quasi-boundary value method to regularize the ”mildly” ill-posed problem and demonstrate the convergence of the regularized solution. Moreover, in order to numerically solve the regularized problem, we proposed a fully
discrete scheme by using finite element method in space and convolution quadrature in time. Sharp error
bounds of the fully discrete scheme are established in both cases of smooth and non-smooth data. Additionally, we propose an easy-to-implement iterative algorithm for solving the fully discrete scheme and prove its linear convergence. Numerical examples are provided to illustrate the theoretical estimates and demonstrate the necessity of the assumption required in the analysis.
 
Several interesting questions remain open. First, our theory imposes a restriction on the terminal time $T$, which cannot be arbitrarily large, even though the solution to the direct problem exists for any $T>0$. 
Numerical experiments demonstrate the necessity of this restriction.
This presents a significant difference from its linear counterpart \cite{sakamoto2011initial, zhang2020numerical} where the reconstruction is always feasible for any $T>0$. It would be interesting to explore the identification of initial data from terminal observation at large $T$. One potential strategy could involve utilizing multiple observations, such as $u(T_1)$ and $u(T_2)$, at two different times $T_1$ and $T_2$. However, the analysis of this approach remains unclear. Moreover, we are interested in the simultaneous recovery of the nonlinear reaction function $f(\cdot)$ and the initial data $u_0$ from two terminal observations. Note that this problem is much more challenging, due to the different types of ill-posedness associated with the recovery of these two parameters \cite{KaltenbacherRundell:2019a,KaltenbacherRundell:2020}.

\section*{Acknowledgements}
The work of J. Yang is supported by the National Science Foundation of China (No.12271240, 12426312), the fund of the Guangdong Provincial Key Laboratory of Computational Science and Material Design, China (No.2019B030301001), and the Shenzhen Natural Science Fund (RCJC20210609103819018).
The work of Z. Zhou is supported by by National Natural Science Foundation of China (Project 12422117), Hong Kong Research Grants Council (15303122) and an internal grant of Hong Kong Polytechnic University (Project ID: P0038888, Work Programme: ZVX3). The work of Z. Zhou and X. Wu is also supported by the CAS AMSS-PolyU Joint Laboratory of Applied Mathematics.

\section*{Appendix}\label{apped}

\noindent\textbf{A. Proof of Corollary \ref{lem:reg-weak}} 
\begin{proof}
To begin, we note that the standard argument in \cite[Theorem 3.1 and 3.2]{Karaa:2019} directly yields the estimate
\begin{equation}\label{eqn:u-mu}
\|u(t)\|_{L^2\II}\le ct^{-\al\mu/2}\| u_0 \|_{\dH {-\mu}}.
\end{equation}
Next, using the solution representation \eqref{eqn:sol-rep}, we consider the splitting:  
\begin{align*}
  \partial_t   (t u(t))=  &\partial_t \left(tF(t)u_0+\int_0^t(t-s)E(t-s)f(u(s))\ds+\int_0^tE(s)(t-s)f(u(t-s))\ds \right)\\
 =&(F(t)+tF'(t))u_0+\int_0^t[(t-s)E'(t-s)+E(t-s)]f(u(s))\ds\\&+\int_0^tE(t-s)[f(u(s))+f'(u(s))su'(s))]\ds \\=&(F(t)+tF'(t))u_0+\int_0^t[(t-s)E'(t-s)+E(t-s)]f(u(s))\ds\\&+\int_0^tE(t-s)[f(u(s))-f'(u(s))u(s)+f'(u(s))\partial_s [su(s)]]\ds .
\end{align*}
Using  the smoothing properties in \cite[Theorem 1.6 (ii) and (iii)]{JinZhou:2023book}  and the Lipschitz condition \eqref{eqn:lipconstant} and the estimate \eqref{eqn:u-mu}, we obtain
\begin{equation*}
    \begin{aligned}
     \|\partial_t   (t u(t))\|_{L^2\II}&\le ct^{-\al\mu/2}\|u_0\|_{\dH{-\mu}}+c\int_0^t(t-s)^{\al-1}\|u(s)\|_{L^2\II} \ds\\&+c\int_0^t(t-s)^{\al-1}\|\partial_s   (s u(s))\|_{L^2\II}\ds \\&\le ct^{-\al\mu/2}\|u_0\|_{\dH{-\mu}}+\int_0^t(t-s)^{\al-1}\|\partial_s   (s u(s))\|_{L^2\II}\ds.
\end{aligned}
\end{equation*}
Applying Grönwall’s inequality in 
 Lemma~\ref{lemma:Gronwall}, we have 
\begin{align*}
      \|\partial_t   (t u(t))\|_{L^2\II}\le ct^{-\al\mu/2} \|u_0\|_{\dH{-\mu}}.
\end{align*}
Using the triangle inequality, we derive that for any $t > 0$, 
\begin{align*}
    \|u'(t)\|_{L^2\II}\le t^{-1}( \|\partial_t   (t u(t))\|_{L^2\II}+\|u(t)\|_{L^2\II})\le ct^{-\al\mu/2-1} \|u_0\|_{\dH{-\mu}}.
\end{align*}
Finally, by applying the same arguments as in Lemma \ref{lem:Reg}, we can derive the second estimate.
\end{proof}

\noindent\textbf{B. Proof of Lemma \ref{lem:error-linear}} 
\begin{proof} 
The proof for the case $ n = 1 $ is straightforward. Let us now consider the case $ n \geq 2 $.
Using the solution representation, we can obtain that 

    \begin{align*}
   v_h(t_n) -v_h^n&=(F_h(t_n)-F_{h,\tau}^n)P_hv_0+ \int_0^{t_n}E_h(t_n-s)P_hf(s)\ \ds-\tau \sum_{k=1}^nE_{h,\tau}^{n-k}P_hf(t_k)\\
   &=(F_h(t_n)-F_{h,\tau}^n)P_hv_0+\int_0^{t_n}(E_h-E_{h,\tau})(t_n-s)f_h(s)\ \ds
   :=\mathrm{I}_1+\mathrm{I}_2,
\end{align*}
where  $f_h(s)=P_hf(s)$ and  $E_{h,\tau}(t)=\tau \sum_{n=0}^\infty E_{h,\tau}^n\delta_{t_n}(t)$.

From  \cite[Lemma 4.2]{zhang2022identification}, for $0\le p\le 1$, it follows that
\begin{equation}\label{I1eq}
    \|A_h^{p} \mathrm{I}_1\|_{L^2\II}=\|A_h^{p}(F_h(t_n)-F_{h,\tau}^n)P_hv_0\|_{L^2\II}\le c\tau t_n^{-1-p\al}\|v_0\|_{L^2\II}.
\end{equation}
For the term $\mathrm{I}_2$, we can derive that
\begin{equation}\label{I2eq}
    \begin{aligned}
   \mathrm{I}_2=&\int_0^\tau (E_h-E_{h,\tau})(t_n-s)f_h(s)\ \ds+\int_\tau ^{t_n}(E_h-E_{h,\tau})(t_n-s)f_h(s)\ \ds\\
    =&\int_0^\tau (E_h-E_{h,\tau})(t_n-s)f_h(s)\ \ds+\int_\tau^{t_n}(E_h-E_{h,\tau})(t_n-s)\ds f_h(\tau)\\&+\int_\tau^{t_n}(E_h-E_{h,\tau})(t_n-s)\int_\tau^sf'_h(y)\dy\ds\\
    :=&\mathrm{I}_{2,1}+\mathrm{I}_{2,2}+\mathrm{I}_{2,3}.
\end{aligned}
\end{equation}
For the term $\mathrm{I}_{2,1}$,  it is evident that
\begin{equation}\label{I21}
\begin{aligned}
   \| A_h^{p} \mathrm{I}_{2,1}\|_{L^2\II}&\le(\int_0^\tau \|A_h^{p}E_h(t_n-s)\|\ds+\|\tau A_h^{p}E_{h,\tau}^{n-1}\|)\|f_h(s)\|_{L^{\infty}(0,\tau; L^2(\Omega))}\\&\le c\tau t_{n}^{(1-p)\alpha-1}\|f_h(s)\|_{L^{\infty}(0,\tau; L^2(\Omega))}.
\end{aligned}    
\end{equation}
Employing a similar argument as \cite[Theorem 3.4]{JinZhou:2023book} gives
\begin{equation}\label{I22}
\begin{aligned}
     \| A_h^{p}\mathrm{I}_{2,2}\|_{L^2\II}&\le \|A_h^{p}\int_\tau^{t_n}(E_h-E_{h,\tau})(t_n-s)\ds\|\|f_h(\tau)\|_{L^2\II}\\
&\le \|A_h^{p}\int_0^{t_{n-1}}(E_h-E_{h,\tau})(t_{n-1}-s)\ds\|\|f_h(\tau)\|_{L^2\II}\\
&\le c\tau t_{n-1}^{(1-p)\alpha-1}\|f_h(\tau)\|_{L^2\II}\le c\tau t_{n}^{(1-p)\alpha-1}\|f_h(\tau)\|_{L^2\II}.
     \end{aligned} 
\end{equation}
For the term $\mathrm{I}_{2,3}$, we have 
\begin{equation*}
    \mathrm{I}_{2,3}=\int_\tau^{t_n}\int_y^{t_n}(E_h-E_{h,\tau})(t_n-s)\ds f'_h(y)\dy=\int_\tau^{t_n}\int_0^{t_n-y}(E_h-E_{h,\tau})(s)\ds f'_h(y)\dy.
\end{equation*}
This leads to 
\begin{equation*}
    \| A_h^{p} \mathrm{I}_{2,3}\|_{L^2\II}\le \int_\tau^{t_n}\|A_h^{p}\int_0^{t_n-y}(E_h-E_{h,\tau})(s)\ds\| \|f'_h(y)\|_{L^2\II}\dy.
\end{equation*}
For $t_n-y\ge \tau$, we can use the same argument as 
\cite[Theorem 3.4]{JinZhou:2023book}
to derive that
\begin{equation*}
    \|A_h^{p}\int_0^{t_n-y}(E_h-E_{h,\tau})(s)\ds\|\le c\tau(t_n-y)^{(1-p)\alpha-1}\le  c\tau(t_{n+1}-y)^{(1-p)\alpha-1}.
\end{equation*}
For $0<t_n-y<\tau$,
 there are
\begin{align*}
     &\|A_h\int_0^{t_n-y}(E_h-E_{h,\tau})(s)\ds\|=   \|A_h\int_0^{t_n-y}E_h(s)\ds\|=\|\int_0^{t_n-y}F_h'(s)\ds\|
     \\= &  \|(F_h(t_n-y)-F_h(0))\|\le c\le c\tau(t_{n+1}-y)^{-1},
\end{align*}
and 
\begin{align*}
     &\|\int_0^{t_n-y}(E_h-E_{h,\tau})(s)\ds\|=   \|\int_0^{t_n-y}E_h(s)\ds\|\le c\int_0^{t_n-y} s^{\alpha-1}\ds\le
     c\tau(t_{n+1}-y)^{\alpha-1}.
\end{align*}
Using Sobolev interpolation leads to 
   \begin{align*}
     \|A_h^{p}\int_0^{t_n-y}(E_h-E_{h,\tau})(s)\ds\|\le c\tau(t_{n+1}-y)^{(1-p)\al-1}, \quad 0\le p\le 1.
\end{align*} 
Consequently, we arrive at
\begin{equation}\label{I23}
      \| A_h^{p} \mathrm{I}_{2,3}\|_{L^2\II}\le c\tau\int_\tau^{t_n}(t_{n+1}-y)^{(1-p)\alpha-1}\|f'_h(y)\|_{L^2\II} \ \dy .
\end{equation}
Combining equations~\eqref{I1eq}--\eqref{I23} yields the desired result.
\end{proof}\vskip5pt


\noindent\textbf{C. Proof of Lemma \ref{lem:err-reg:H2}}  
\begin{proof}
Let $\bar e^n= \tuh(t_n)-\bar  U_{h,\gamma}^{n}$. Using the solution representations \eqref{eq:semisolutionrep} and \eqref{eqn:back-fully-non} gives
\begin{equation*}
\begin{aligned}
       \bar e^n =&(F_h(t_n)-F_{h,\tau}^n)\tuh(0)+\left(\int_0^{t_n}E_h(t_n-s)P_hf(\tuh(s))\ds-\tau \sum_{k=1}^nE_{h,\tau}^{n-k}P_hf(\tuh(t_k)) \right)\\&+\tau \sum_{k=1}^nE_{h,\tau}^{n-k}P_h[f(\tuh(t_k))-f(\tuh(t_{k-1}))]+\tau \sum_{k=2}^nE_{h,\tau}^{n-k}P_h[f(\tuh(t_{k-1}))-f(\bar  U_{h,\gamma}^{k-1})]\\
      :=& \mathrm{I}_1+\mathrm{I}_2+\mathrm{I}_3+\mathrm{I}_4.
\end{aligned}
\end{equation*}
From the Lipschitz condition \eqref{eqn:lipconstant} and the regularity estimate in Lemma \ref{THM:Reg_semi}, we have 
$$\|f(\tuh(s))\|_{L^{\infty}(0,\tau; L^2(\Omega))}\le c\|\tuh(0)\|_{L^2\II},\quad \|\tuh'(s)\|_{L^2\II}\le cs^{-1}\|\tuh(0)\|_{L^2\II}.$$ 
 Consequently, from Lemma~\ref{lem:error-linear}, we arrive at  for $p\in [0,1]$
 \begin{align*}
        \|A_h^p (\mathrm{I}_1+\mathrm{I}_2)\|_{L^2\II}\le &c\bigg(\tau t_n^{-1-p\al}\|\tuh(0)\|_{L^2\II}+\tau t_n^{(1-p)\alpha-1} \|f(\tuh(s))\|_{L^{\infty}(0,\tau; L^2(\Omega))}\\&+\tau\int_{\tau}^{t_n}(t_{n+1}-s)^{(1-p)\alpha-1}\|f'(\tuh(s))\tuh'(s)\|_{L^2\II}\ds  \bigg)\\
      \le &c(\tau |\log \tau| t_n^{(1-p)\alpha-1}+\tau t_n^{-1-p\al})\|\tuh(0)\|_{L^2\II},
 \end{align*}
 and
  \begin{align*}
  \|A_h^p(\mathrm{I}_3+\mathrm{I}_4)\|_{L^2\II}\le &c\tau \sum_{k=2}^n(t_{n+1}-t_k)^{(1-p)\alpha-1}\|\tuh(t_k)-\tuh(t_{k-1})\|_{L^2\II}\\&+c\tau t_n^{(1-p)\alpha-1}\|\tuh(t_1)-\tuh(0)\|_{L^2\II}+c\tau\sum_{k=2}^n(t_{n+1}-t_k)^{(1-p)\al-1}\| \bar e_{k-1}\|_{L^2\II}\\
       \le &c\tau |\log \tau| t_n^{(1-p)\alpha-1}\|\tuh(0)\|_{L^2\II}+c\tau\sum_{k=2}^n(t_{n+1}-t_k)^{(1-p)\al-1}\| \bar e_{k-1}\|_{L^2\II},
        \end{align*}
where the last inequality follows from 
 \begin{align*}
    &\sum_{k=2}^n(t_{n+1}-t_k)^{(1-p)\alpha-1}\|\tuh(t_k)-\tuh(t_{k-1})\|_{L^2\II}
    \\\le& c\sum_{k=2}^n(t_{n+1}-t_k)^{(1-p)\alpha-1}\int_{t_{k-1}}^{t_k}\|\tuh'(s)\|_{L^2\II}\  \ds
    \\\le& c\sum_{k=2}^n(t_{n+1}-t_k)^{(1-p)\alpha-1}\int_{t_{k-1}}^{t_k}s^{-1}\  \ds\|\tuh(0)\| _{L^2\II}
    \\
\le&c\int_{\tau}^{t_n}(t_{n+1}-s)^{(1-p)\alpha-1}s^{-1}\  \ds \|\tuh(0)\|_{L^2\II}\le c|\log \tau|t_{n}^{(1-p)\alpha-1}\|\tuh(0)\|_{L^2\II}.
\end{align*}   
Then  we arrive at the following estimate for $0\le p\le 1$
\begin{align*}
    \|A_h^p\bar e_{n}\|_{L^2\II}\le c(\tau |\log \tau| t_n^{(1-p)\alpha-1}+\tau t_n^{-1-p\al})\|\tuh(0)\|_{L^2\II}+c\tau\sum_{k=2}^n(t_{n+1}-t_k)^{(1-p)\al-1}\| \bar e_{k-1}\|_{L^2\II}.
\end{align*}
Setting  $p=0$ and 
applying the discrete Gronwall's inequality in Lemma~\ref{disGr}  gives 
\begin{equation*}\label{apbare}
    \|\bar e_{n}\|_{L^2\II}\le  c\tau  |\log \tau| t_n^{-1}\|\tuh(0)\|_{L^2\II}.
\end{equation*}
Then we can derive that for $0<p\le 1$
\begin{align*}
      \|A_h^p\bar e_{n}\|_{L^2\II}&\le c((\tau |\log \tau| t_n^{(1-p)\alpha-1}+\tau t_n^{-1-p\al}+\tau^2|\log \tau| \sum_{k=2}^n(t_{n+1}-t_k)^{(1-p)\al-1}t_n^{-1})\|\tuh(0)\|_{L^2\II}\\&\le c(\tau|\log \tau |^2 t_n^{(1-p)\al-1}+\tau t_n^{-1-p\al})\|\tuh(0)\|_{L^2\II}\le c_T\tau|\log \tau |^2  t_n^{-1-p\al}\|\tuh(0)\|_{L^2\II}.
\end{align*}
\end{proof}
Below we have given a useful Gronwall’s inequality, which generalizes the standard variants in \cite[Lemma 9.9]{JinZhou:2023book}.
\begin{lemma}\label{disGr}
Let $ 0\le \varphi^n \leq R$ for $0 \leq t_n \leq T$. If
$$
\varphi^n \leq a_1 t_n^{-1}+a_2t_n^{\beta_1-1}+b \tau \sum_{j=2}^n t_{n-j+1}^{\beta_2-1}\varphi^{j-1}, \quad 0<t_n \leq T,
$$
for some $a, b \geq 0$, $\beta_1, \ \beta_2\in(0,1)$ and $p>0$, then there is $c=c(b, \beta_2, T, R)$ such that
$$
\varphi^n \leq  c(a_1 t_n^{-1}|\log \tau|+a_2t_n^{\beta_1-1}), \quad 0<t_n \leq T.
$$
\end{lemma}

\begin{proof} Define $\varphi(t)=\varphi^n$, for $t\in (t_{n-1}, t_n]$.    Let  $a_\beta(t)=a_1 t^{-1}+a_2t^{\beta_1-1}$ for $t\ge\tau$, and $a_\beta(t)=a_1 \tau^{-1}+a_2\tau^{\beta_1-1}$ for $0<t\le\tau$.
 It is straightforward to obtain that
 \begin{align*}
  \varphi(t)\le& a_1 t_n^{-1}+a_2t_n^{\beta_1-1}+b  \sum_{j=2}^n\int_{t_{j-2}}^{t_{j-1}}  t_{n-j+1}^{\beta_2-1} \varphi(s)\ds \\
\le &a_\beta(t)+c  \sum_{j=2}^n\int_{t_{j-2}}^{t_{j-1}}  (t-s)^{\beta_2-1} \varphi(s)\ds\le a_\beta(t)+c  \int_{0}^{t}  (t-s)^{\beta_2-1} \varphi(s)\ds.
 \end{align*}
Here we use $bt_{n-j+1}^{\beta_2-1}\le c(t-s)^{\beta_2-1}$ for $t\in (t_{n-1}, t_n], \ s \in  (t_{j-2}, t_{j-1}) $ and $b  \sum_{j=2}^n\int_{t_{j-2}}^{t_{j-1}}  t_{n-j+1}^{\beta_2-1} \varphi(s)\ds=0$ for $n=1$.
Applying the Gronwall's inequality in Lemma \ref{lemma:Gronwall} leads to the desired result.
\end{proof}
    

\end{document}